\documentclass[article]{article}
\usepackage{color}
\usepackage{graphicx}
\usepackage{amsmath}
\usepackage{amssymb}
\usepackage{mathrsfs}
\usepackage[numbers,sort&compress]{natbib}
\usepackage{amsthm}
\numberwithin{equation}{section}
\usepackage{pgf}
\usepackage{graphicx}
\usepackage{tikz}
\usepackage{stmaryrd}
\usepackage[latin1]{inputenc}
\usepackage[T1]{fontenc}
\usepackage[english]{babel}
\usepackage{listings}
\usepackage{xcolor,mathrsfs,url}
\usepackage{amssymb}
\usepackage{amsmath}
\usepackage{ifthen}
\newtheorem{theorem}{Theorem}
\newtheorem{lemma}{Lemma}

\newtheorem{corollary}{Corollary}
\newtheorem{proposition}{Proposition}
\newtheorem{Assumption}{Assumption}
\usepackage {caption}
\usepackage{xcolor}
\usepackage{cases}

\usepackage{float}
\usepackage{subfigure}
\usepackage{hyperref}
\hypersetup{
    colorlinks=true, %set true if you want colored links
    linktoc=all,     %set to all if you want both sections and subsections linked
    linkcolor=blue,  %choose some color if you want links to stand out
    citecolor=blue%choose some color if you want links to stand out
}

\begin{document}

\title{ Long-time asymptotics  for the focusing Fokas-Lenells equation in the solitonic region of space-time}
\author{Qiaoyuan Cheng$^1$,  Engui Fan$^1$\thanks{\ Corresponding author and email address: faneg@fudan.edu.cn } }
\footnotetext[1]{ \  School of Mathematical Sciences  and Key Laboratory of Mathematics for Nonlinear Science, Fudan University, Shanghai 200433, P.R. China.}

\date{}

\maketitle
\begin{abstract}
\baselineskip=16pt

We study the long-time asymptotic behavior of the focusing Fokas-Lenells (FL) equation
$$
u_{xt}+\alpha\beta^2u-2i\alpha\beta u_x-\alpha u_{xx}-i\alpha\beta^2|u|^2u_x=0 \label{cs}
$$
with generic initial data in a Sobolev space which supports bright soliton solutions. The FL equation is an integrable generalization of  the
 well-known Schrodinger equation, and also linked to the derivative Schrodinger model, but it  exhibits   several different characteristics   from  theirs.  (i) The Lax pair of the FL equation involves an additional spectral singularity at $k=0$.  (ii) four stationary phase points will appear during asymptotic analysis, which require a more detailed  necessary  description to obtain the long-time asymptotics of the focusing FL equation.
    Based on the  Riemann-Hilbert problem for the  initial value problem  of  the focusing FL equation,  we
    show that inside any fixed time-spatial  cone
$$\mathcal{C}\left(x_{1}, x_{2}, v_{1}, v_{2}\right)=\left\{(x, t) \in \mathbb{R}^{2} | x=x_{0}+v t, x_{0} \in\left[x_{1}, x_{2}\right], v \in\left[v_{1}, v_{2}\right]\right\},$$
the long-time asymptotic behavior of the solution $u(x,t)$ for the focusing  FL equation
can be characterized with an $N(\mathcal{I})$-soliton on discrete spectrums and a leading order term
$\mathcal{O}(|t|^{-1/2})$ on continuous spectrum up to a residual error order $\mathcal{O}(|t|^{-3/4})$. The main tool is  the  $\overline{\partial}$  nonlinear steepest descent method  and   the $\overline{\partial}$-analysis.
\end{abstract}
\baselineskip=17pt

\newpage

\tableofcontents

\section{Introduction}
\hspace*{\parindent}
  The study of long-time behaviors of solutions to nonlinear integrable systems goes
back  to the work of Zakharov and Manakov \cite{Zakharov}.  Using the monodromy theory, Its was able to reduce the Riemann-Hilbert (RH) problem formulation for the Schr\"{o}dinger (NLS) equation
to a model case, which can then be solved explicitly, giving the desired asymptotics \cite{Its}.
Deift and Zhou developed a rigorous nonlinear steepest descent method to
study the oscillatory RH problem  associated with the mKdV equation \cite{asd13}.
Later it becomes  a   powerful tool for the long-time asymptotics of the  integrable  nonlinear evolution equations, such as the NLS equation \cite{RN9,RN10},
  the KdV equation \cite{Grunert2009},    the Camassa-Holm equation \cite{MonvelCH1,MonvelCH2},
  the modified Camassa-Holm equation \cite{MonvelCH3}, the Degasperis-Procesi equation \cite{MonvelCH4,MonvelCH5},  the sine-Gordon equation \cite{Cheng,Huang}, the Fokas-Lenells equation \cite{Lta} and so on.
  Recently the Deift-Zhou steepest descent method   was further  extended to  the  $\overline{\partial}$-steepest descent method
of  McLaughlin and Miller, which first appeared in the
orthogonal polynomial setting \cite{tds23,tsd24}.  The  $\overline{\partial}$-method follows the general scheme of the Deift-Zhou steepest
descent argument, while the nonanalytic data are now continued to
the desired contours via the solution of a $\overline{\partial}$-equation.   The  $\overline{\partial}$-steepest descent method
allows asymptotic analysis for nonanalytic phases with two Lipschitz derivatives near the stationary points. This method was adapted to obtain the long-time
asymptotics for solutions to the NLS equation and the derivative NLS equation, with a sharp error bound for the weighted Sobolev initial data
\cite{lta29, oas30,lta31, gwp33}.

In this paper we study the long-time asymptotic behavior of the   Fokas-Lenells (FL) equation
\begin{equation}
u_{tx}+\alpha\beta^2u-2i\alpha\beta u_x-\alpha u_{xx}+\sigma i\alpha\beta^2|u|^2u_x=0,\label{cs}
\end{equation}
which is an integrable generalization of nonlinear Schr\"{o}dinger equation, and is also related to the derivative NLS model.
As is well-known, the Camassa-Holm equation can be  mathematically  derived by utilizing the two Hamiltonian operators associated with the KdV equation \cite{oac36}. Similarly, by utilizing the two Hamiltonian operators associated with the NLS equation, it is possible to derive the  FL equation  \cite{esm8}.
In optics, considering suitable higher-order linear and nonlinear optical effects, the FL equation has been derived as a model to describe the femtosecond
 pulse propagation through single mode optical silica fiber, for which several interesting solutions have been constructed. It also belongs to the deformed derivative  NLS hierarchy proposed by  Kundu \cite{tfi6,ith7}.

The soliton solutions for the  FL equation   have been constructed by the inverse scattering transform  method  and the dressing method \cite{oan36, dfa37}.
 Matsuno obtained  the bright and the dark soliton solutions for the  FL equation by  using the Hirota method \cite{adm38,adm39}.
The   double Wronskian solutions for the FL  equation was further  given  via bilinear approach \cite{zhang2021}.   The lattice representation and the n-dark solitons of the FL equation have been presented in \cite{lra40}. The breather solutions of the FL equation have also been constructed via a dressing-B\"{a}cklund transformation related to the RH  problem formulation  \cite{shc41}. It has been shown that the periodic initial value problem for the FL equation is well-posed in a Sobolev space   \cite{wpo10}.
A kind of rogue wave solution for the FL equation were obtained by using the Darboux transformation  \cite{rwo34}.  The Fokas method was used to investigate the initial-boundary value problem
for the FL equation  on the half-line and a finite interval, respectively \cite{Lenells2009,xiaofan}.
An algebro-geometric method was used to obtain algebro-geometric solutions for the  FL equation \cite{zhao2013}. The explicit one-soliton of  the  initial value problem for the FL equation
was given by the RH method \cite{xujian2}.
The inverse scattering transformation for the FL
equation with nonzero boundary conditions was further investigated by using the RH method  \cite{zf1}.

It is noted that for the defocusing case $\sigma=1$, the FL equation (\ref{cs}) with zero boundary conditions does not admit a soliton solution. For the Schwartz initial value  $u(x,0)\in \mathcal{S}(\mathbb{R})$, Xu obtained  the long-time asymptotics for the FL equation without solitons by using the Deift-Zhou method \cite{Lta}. However, for the focusing case  $\sigma=-1$,  the FL equation (\ref{cs}) takes
\begin{align}
&u_{tx}+\alpha\beta^2u-2i\alpha\beta u_x-\alpha u_{xx}-i\alpha\beta^2|u|^2u_x=0,\label{cs2}\\
&u(x, 0)=u_{0}(x),\label{cz}
\end{align}
whose soliton solutions will appear for zero boundary conditions which correspond to discrete spectrums. In this case,
the long-time asymptotic  behavior of solutions   is
more complicated than the defocusing case due to the presence of solitons.
A more detailed  necessary  description on the  asymptotic analysis and some new  techniques have to be adapted.

In this work, we apply the  $\bar\partial$-techniques  to obtain the long-time asymptotic behavior of solutions to the focusing FL equation (\ref{cs2}) with a Sobolev initial data $u_0(x)\in  H^{3,3}(\mathbb{R}).$
Comparing with the classical  NLS equation and the derivative NLS equation \cite{oas30,lta31,gwp33}, to conduct an asymptotic analysis on the    FL equation (\ref{cs2})
confronts with the following complications,

 \begin{itemize}

\item[$\blacktriangleright$]  The  FL equation (\ref{cs2})  belongs to a negative hierarchy,
 with two
 spectral singularities at $k=0$ and $k=\infty$ involved in its Lax pair.

\item[$\blacktriangleright$]  In comparison with the FL spectral problem  (\ref{Lax pair}), the Zakharov-Shabat spectral
  problem has no multiplication of matrix potential $U_x$ by $k$. As a result, Neumann series solutions for the
   Jost functions of Zakharov-Shabat spectral
  problem converge if $u_0(x)\in L^1(\mathbb{R})$.  However, for  the FL spectral problem  (\ref{Lax pair}),
    the Voterra equation  will   involve  the  term  $k u_{0x}(x)$ (see (\ref{gtit}) in next section ),
  which is not $L^2(\mathbb{R})$  bounded   since    $k$  may go  to  infinity.
Through the small $k$ and large $k$ estimates respectively,  we overcome this difficulty and prove  the existence and the differentiability of the eigenfunctions, further
    establish the scattering map from the initial data $  u_{0 }(x)$ to the scattering coefficient  $r(k)$.

  \item[$\blacktriangleright$]  In the   RH problem of the  FL equation (\ref{cs2}),
   the   oscillatory term $ e^{ it \theta(k) }$ involves four stationary phase  points: $z_1,z_3$ are  located on the real axis and $z_2, z_4$ are  located on the imaginary axis.
  The corresponding  local models  need two kinds of parabolic cylinder models to described in  asymptotic analysis.
  To simplify the expression of the RH problem  formula in our analysis, we propose to use a uniform parabolic cylinder model to match the solvable RH model  of the four stationary phase points.

 \item[$\blacktriangleright$]   When  the variables $ x $ and  $  t $  change   inside a time-spatial cone $\mathcal{C}(x_1,x_2,v_1,v_2)$,
  the corresponding  discrete spectrums fall in a   band  region   $ -v_2/2< {\rm Re}(k) <-v_1/2$  for  the NLS equation and derivative NLS equation \cite{oas30,lta31,gwp33}, while  for our  FL equation, the discrete spectrums fall in  a tyre region  $f(v_1) <|k| <f(v_2)$ (please see Figure \ref{division4} in Section \ref{sec:section7}).

\end{itemize}

The structure of the paper is as follows.   In Section \ref{sec:section2}, based on   the Lax pair of the FL equation (\ref{cs2}) and through the small $k$ and large $k$ estimates,  we  prove  the existence and the differentiability of the eigenfunctions. We also establish the scattering map from the initial data to the scattering coefficients, and study the analyticity, the symmetries and the asymptotics of the eigenfunctions. In Section \ref{sec:section3}, inspired by \cite{Lta}, we construct a RH problem for $M(k)$ to formulate the initial value problem of the FL equation.
In our  Section \ref{sec:section31},  we  prove  the existence and  uniqueness of the RH  problem. We first show that the Beals-Coifman  integral equation associated with the RH problem has
 a unique solution.  Then we  give  a  reconstruction formula  for
the  potential $u $ in terms of the solution $\mu(k)$ of the Beals-Coifman   integral equation.  In Section \ref{sec:section4}, we introduce a function $T(k)$ to transform $M(k)$ into a new RH problem for $M^{(1)}(k)$, which admits a regular discrete spectrum and two triangular decompositions of the jump matrix near four stationary-phase points $z_n$, $n=1,2,3,4$. By introducing a matrix-valued function $ R^{(2)}(k)$, we obtain a mixed $\overline{\partial}$-RH problem for $M^{(2)}(k)$ by a continuous extension of the  $M^{(1)}(k)$. In Section
\ref{sec:section5}, we decompose $M^{(2)}(k)$ into a model RH  problem for $M^{rhp}(k)$ and a pure $\overline{\partial}$-Problem
for $M^{(3)}(k)$. The $M^{rhp}(k)$ can be obtained via an outer model $M^{(out)}(k)$ for the soliton
components to be solved in Section \ref{sec:section6}, and an inner model $M^{(in)}(k)$ which are approximated by a solvable model for $M^{fl}(k)$ obtained   in Section
\ref{sec:section7}. In Section \ref{sec:section8}, we compute the error function $E(k)$ with a small-norm RH
problem. In Section \ref{sec:section9}, we analyze the $\overline{\partial}$-problem for $M^{(3)}(k)$. Finally, in Section \ref{sec:section10}, based
on the results obtained above, we find a decomposition formula
\begin{equation}
M(k)=M^{(3)}(k) E(k) M^{(out)}(k)  R^{(2)}(k)^{-1} T(k)^{\sigma_{3}},
\end{equation}
  which  leads to the  long-time asymptotic behavior for  the solutions of  the  initial value problem (\ref{cs2})-(\ref{cz})   for  the FL equation.

\section{Spectral analysis on Lax pair }
\label{sec:section2}
\hspace*{\parindent}
In this section,  we  show  the well-posedness for the initial value problem of the Fokas-Lenells (FL) equation (\ref{cs2}).
We  prove  the existence of the eigenfunctions and
 establish the  scattering  map from the initial data to the scattering coefficients.
\subsection{Existence and differentiability  of eigenfunctions }
\hspace*{\parindent}
The FL equation (\ref{cs2}) admits the Lax pair
\begin{equation}
\begin{cases}
\psi_x+ik^2\sigma_3\psi=kU_x\psi,\\
\psi_t+i\eta^2\sigma_3\psi=V\psi,
\end{cases}\label{Lax pair}
\end{equation}
where $\psi=\psi(x,t,k)$, and
\begin{align}
& \sigma_3=\begin{pmatrix}
1&0\\
0&-1
\end{pmatrix}, \ \ \ \  U=\begin{pmatrix}
0&u\\
-\bar{u}&0
\end{pmatrix},\label{qq}\\
&
\eta=\sqrt{\alpha}(k-\frac{\beta}{2k}),\quad V= \alpha kU_x+\frac{i\alpha\beta^2}{2}\sigma_3(\frac{1}{k}U-U^2). \nonumber
\end{align}

With the transformation
\begin{equation}
\psi=\varphi e^{-i(k^2x+\eta^2t)\sigma_3},\label{psi}
\end{equation}
$\varphi$ satisfies a new Lax pair
\begin{equation}
\begin{cases}
\varphi_x+ik^2[\sigma_3,\varphi]=kU_x\varphi,\\
\varphi_t+i\eta^2[\sigma_3,\varphi]=V\varphi,
\end{cases}\label{lax}
\end{equation}
which can be written in the full derivative form
\begin{equation}
d(e^{i(k^2x+\eta^2t)\widehat{\sigma}_3}\varphi )=e^{i(k^2x+\eta^2t)\widehat{\sigma}_3}(kU_x \mathrm{dx}+V\mathrm{dt}) \varphi,
\end{equation}

The Lax pair (\ref{lax})  has singularities at $k=0$ and $k=\infty$,  which is different from NLS and derivative NLS equations. To control the behavior of solutions of (\ref{lax}) and construct the solution $u(x,t)$ of the FL equation $(\ref{cs2})$, we should resort to the $t$-part and the expansion of the eigenfunction as the spectral parameter $k \rightarrow 0$. Specifically, we consider two different asymptotic expansions respectively to analyze these two singularities $k=0$ and $k=\infty$.

From the Lax pair  (\ref{lax}), we obtain the following asymptotic expansion
\begin{align}
&\varphi(x,t,k)=I+k U+ \mathcal{O}\left(k^2\right), \quad k \rightarrow 0,\\
&\varphi(x,t,k)=\varphi_0+\mathcal{O}(k^{-1}),\quad k\rightarrow\infty,
\end{align}
where $U$ is given by (\ref{qq}) and
\begin{equation}
\varphi_0(x,t)=e^{-\frac{i}{2}\int_{-\infty}^{x}|u_y(y,t)|^2dy\sigma_3}.\label{d}
\end{equation}
We introduce   a new function $\omega=\omega(x,t,k)$ by
\begin{equation}
\varphi(x,t,k)=\varphi_0(x,t)\omega(x,t,k),\label{xx}
\end{equation}
which satisfies  the  Lax pair
\begin{align}
&\omega_x+i k^2\left[\sigma_3, \omega\right]=V_1 \omega, \label{L3}\\
&\omega_t+i \eta^2\left[\sigma_3, \omega\right]=V_2 \omega,\label{L4}
\end{align}
and can be written in full derivative form
\begin{equation}
d\left(e^{i\left(k^{2} x+\eta^{2} t\right) \widehat{\sigma}_{3}} \omega \right)=W(x, t, k),
\end{equation}
where
\begin{equation}
W(x,t,k)=e^{\mathrm{i}(k^{2} x+\eta^{2} t) \widehat{\sigma}_{3}}(V_1  d x+V_2  dt) \omega,
\end{equation}
and the matrices $V_1 $ and $V_2 $  are given  respectively, by
\begin{align}
&V_1=e^{\frac{i}{2} \int_{-\infty}^x|u_y(y, t)|^2 dy \widehat{\sigma}_{3}}(k U_x+\frac{i}{2} |u_x|^2   \sigma_{3}),\nonumber \\
&V_2=e^{\frac{i}{2} \int_{-\infty}^x |u_y(y, t)|^2 dy \widehat{\sigma}_{3}}(V+\frac{i \alpha}{2}(|u_x|^2-\beta^{2} |u|^2) \sigma_3).\nonumber
\end{align}
Now we conduct a spectral analysis on the $x$-part of (\ref{L3}). Since the analysis takes place at a fixed time, the $t$-dependence will be suppressed. Define two solutions $\omega^\pm(x,k)$ of the $x$-part of (\ref{L3}) by
\begin{equation}
\omega^{\pm}(x,k)=I+\int_{\pm\infty}^{x} \mathrm{e}^{-\mathrm{i} k^2\left(x-y \right) \widehat{ \sigma}_3}\left(V_1 \omega^{\pm}\right)\left(y,k\right) \mathrm{d}y. \label{gtit}
\end{equation}

Denoting $\omega^{\pm}(x,k)=\left(\omega^{\pm}_1,\omega^{\pm}_2\right),$  where the scripts  $1$ and $2$ denote the first and second columns of $\omega^{\pm}(x,k)$, we  have

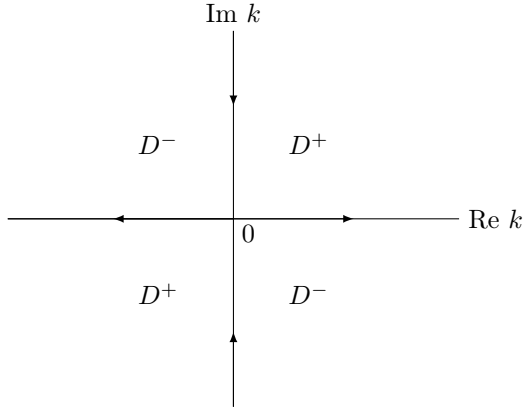
\begin{figure}[H]
\begin{center}
\begin{tikzpicture}
%\draw [fill=pink,ultra thick,color=orange!20] (0,0) rectangle (2,2);
%\draw [fill=pink,ultra thick,color=orange!20] (0,0) rectangle (-2,-2 );
%\draw [->](-2.5,0)--(2.5,0);
%\draw [->](0,-2.5)--(0,2.5);
 \node at (1, 1 )  {$D^+$};
\draw [ ](-3,0)--(3,0)  node[right, scale=1] {Re $k$};
\draw [-latex](-3,0)--(1.6,0);
\draw [-latex](0,0)--(-1.6,0);
\draw [ ](0,-2.5)--(0,2.5)  node[above, scale=1] {Im $k$};
\draw [-latex](0, 2.5)--(0,1.5);
\draw [-latex](0, -2.5)--(0,-1.5);
\node at (0.2, -0.2 )  { $0$};
 \node at (-1, -1 )  {$D^+$};
\node at (1, -1 )  {$D^-$};
 \node at (-1, 1 )  {$D^-$};

\end{tikzpicture}
\end{center}
\caption{The analytical domains $D^+$ and  $D^-$ for $\omega^{\pm}(x,k)$.}
\label{figure1}
 \end{figure}

\begin{proposition} \label{lemma1}
 For  $u(x,0)\in H^{3,3}(\mathbb{R})$, there exist   unique   eigenfunctions  $\omega^{\pm}(x, k)$
  which satisfy  (\ref{gtit}), respectively.
  Moreover,  $ \omega^{-}_1(x,k)$, $ \omega^+_{2}(x,k)$ and $a(k)$  are analytical in $ D^+;$   $ \omega^+_1(x,k)$ and $\omega^{-}_2(x,k)$ are analytical  in $D^-$, where
$$D^+=\{ k:  {\rm Im} k^2>0\}, \  \ D^-=\{ k:  {\rm Im} k^2<0\},$$
as Figure \ref{figure1} shows.
\end{proposition}
By denoting $\omega^{\pm}(x,k) =(\omega_{ij}^{\pm}(x,k) )$,
from (\ref{gtit}), we obtain two integral equations
\begin{align}
&\omega_{11}^{\pm}(x,k)=1-\int_{x}^{\pm \infty} ku_y(y) \omega_{21}^\pm(y) dy-\int_{x}^{\pm \infty} \frac{i}{2}|u_y(y)|^2 \omega_{11}^{\pm}(y, k) d y ,\label{o1}\\
&\omega_{21}^{\pm}(x,k)=\int_x^{\pm\infty} e^{2ik^2(x-y)}k \bar{u}_y(y) \omega_{11}^{\pm}(y,k) dy+\frac{i}{2}\int_x^{\pm \infty} e^{2ik^2(x-y)}|u_y(y)|^{2} \omega_{21}^{\pm}(y, k) dy,\label{o2}
\end{align}
where $\bar{u}_x(x)$ is the substitution of $\bar{u}_x e^{i\int_{-\infty}^x|u_y(y, t)|^2}$.

Without loss of generality,  we take
 $  \omega_{1}^+ =(\omega_{11}^+, \omega_{21}^+)^T$  as an illustrative example
 to prove Proposition  \ref{lemma1}   from  the small-$k$ estimates and large-$k$ estimates.

\subsubsection{Small-$k$ estimates}
\hspace*{\parindent}
 By denoting
\begin{equation}
\mathbf{w}=\left(\omega_{11}^+-1,  \omega_{21}^+\right)^{T},
\end{equation}
we write $(\ref{o1})-(\ref{o2})$ as the form
\begin{equation}
\mathbf{w}=\mathbf{w}_0+T_0 \mathbf{w}, \quad \mathbf{w}_0 \equiv T_0 \mathbf{e}_{1},\label{wc}
\end{equation}
where $T_{0}$ is an integral operator defined by
\begin{equation}
\left(T_{0} h\right)(x)=\int_{x}^{\infty} K_{0}(x, y, k)h(y) d y,
\end{equation}
and
\begin{equation}
K_{0}(x, y, k)=\begin{pmatrix}
-\frac{i}{2}|u_y|^2  & -ku_y  \\
e^{2ik^2(x-y)} k\bar{u}_y  & \frac{i}{2}|u_y|^2
\end{pmatrix}. \nonumber
\end{equation}
  To study the $k$-derivatives of the solution, we solve the integral equations
\begin{equation}
\mathbf{w}_k=\mathbf{w}_{0,k}+T_{0,k}\mathbf{w}+T_0\left(\mathbf{w}_k\right).\label{wc1}
\end{equation}
\vspace{2mm}

\noindent {\bf A. Estimates on $\mathbf{w}_0$, $T_0$ and their derivatives}
\hspace*{\parindent}

In this part, we derive estimates on the terms $\mathbf{w}_0$, $T_0$ and their derivatives  with respect to  $k$, which we later use  to estimate $r(k)$ and $r'(k)$.

\begin{lemma}\label{d11}
For  $k\in I_0=\{k\in \mathbb{R}\cup i\mathbb{R}:|k|<1\}$,  we have  the following estimates
\begin{align}
&\left\|\mathbf{w}_{0}\right\|_{C^0\left(\mathbb{R}^{+}, L^2\left(I_{0}\right)\right)}  \lesssim\left\|u\right\|_{H^{2,2}},\label{w011}\\
&\left\|\mathbf{w}_0\right\|_{L^2\left(\mathbb{R}^+ \times I_{0}\right)} \lesssim\|u\|_{H^{3,3}} ,\label{w01}\\
&\left\|\mathbf{w}_{0,k}\right\|_{C^0\left(\mathbb{R}^+, L^2\left(I_0\right)\right)}  \lesssim\|u\|_{H^{2,2}}, \\
&\left\|\mathbf{w}_{0,k}\right\|_{L^2\left(\mathbb{R}^+ \times I_0\right)} \lesssim\|u\|_{H^{2,2}}.\label{w02}
\end{align}

\end{lemma}
\begin{proof}
Noting that
$$
\mathbf{w}_{0}= \left(\int_x^\infty  -\frac{i}{2}|u_y|^2 dy,  \int_x^\infty e^{2ik^2(x-y)} k\bar{u}_y dy\right)^T,
$$
whose  first component is independent of $k$, bounded by $\left\|u\right\|_{H^{1}}$.

By letting $\varphi \in C_0^{\infty}\left(I_{0}\right)$, we compute
\begin{equation}
\left\|\int_x^{\infty} e^{2ik^2(x-y)}k\bar{u}(y) d y\right\|^2_{L^2(I_0)}=sup_{||\phi||=1}\int_{I_0} \varphi(k) \int_x^{\infty} e^{2ik^2(x-y)} k\bar{u}_y(y) dy
\lesssim c||u_y||_{L^1},
\end{equation}
The first estimate is immediate and the second follows by integration in $x$. Therefore for the second component, we have
\begin{equation}
\int_{\mathbb{R}^+} \int_{I_0}\left|\int_x^{\infty} e^{2ik^2(x-y)} ku_y(y) dy\right|^2 dk dx\lesssim\int_{\mathbb{R}^+}\left|\int_x^{\infty} e^{2 i k^2(x-y)} ku_y(y) dy\right|^2dx\lesssim||u||_{L^{1,1}}||u_y||_{L^1}.\nonumber
\end{equation}

Noting that
\begin{equation}
\mathbf{w}_{0,k}=\begin{pmatrix}
0 \\
\int_{x}^{\infty} g(x, y, k) dy
\end{pmatrix},
\end{equation}
where
\begin{equation}
g(x, y, k)=e^{2ik^2(x-y)}(1+4ik^2)(x-y)\bar{u}_y(y),\nonumber
\end{equation}
the estimates of $\mathbf{w}_{0,k}$ can be obtain in the same way.
\end{proof}

The operator $ T_{0,k}$ induces linear mappings from
$L^2\left(\mathbb{R}^+\times I_0\right)$ to $L^{2}\left(\mathbb{R}^+\times I_0\right)$
and $C^0\left(\mathbb{R}^+, L^2\left(I_0\right)\right)$ by the formula $g(x, k)= T_{0,k}(f(\cdot,k))(x)$ respectively. We need the following estimates on these induced maps.

\begin{lemma}
Suppose that $u(x,0)\in H^{3,3}(\mathbb{R})$, the following operator bounds hold uniformly  and the operators are Lipschitz functions of $u$.
\begin{align}
&\left\|T_0\right\|_{L^2\left(\mathbb{R}^+ \times I_0\right) \rightarrow C^0\left(\mathbb{R}^+, L^2\left(I_0\right)\right)} \lesssim\left\|u\right\|_{H^{2,2}},\nonumber\\
&\left\|T_{0,k}\right\|_{L^2\left(\mathbb{R}^+ \times I_{0}\right) \rightarrow L^{2}\left(\mathbb{R}^+ \times I_0\right)} \lesssim\|u\|_{H^{3,3}},\nonumber\\
&\left\|T_{0,k}\right\|_{L^2\left(\mathbb{R}^+ \times I_0\right) \rightarrow C^0\left(\mathbb{R}^+, L^2\left(I_0\right)\right)} \lesssim\|u\|_{H^{3,3}}.\nonumber
\end{align}
\end{lemma}
\begin{proof}
For an operator $T_0(k)$ with an integral kernel $K(x,y,k)$ satisfying the estimate
\begin{equation}
\sup _{k\in I_0}|K(x,y,k)| \leqslant h(y),
\end{equation}
and satisfying $K(x,y,k)=0$ if $x>y$, the $L^2\left(\mathbb{R}^+ \times I_0\right)$-norm is controlled by
\begin{equation}
\left(\int_0^{\infty} \int_x^{\infty} h(y)^{2} dy dx\right)^{1/2}=\left(\int_0^{\infty} y h(y)^2 d y\right)^{1/2},
\end{equation}
and the norm from $L^2\left(\mathbb{R}^+ \times I_0\right)$ to $C^0\left(\mathbb{R}^+, L^2\left(I_0\right)\right)$ is controlled by
$
\sup_x\left(\int_0^{+\infty} h(y)^2 d y\right)^{1/2}.
$
The conclusions follow from this observation and the estimates
\begin{equation}
\begin{aligned}
\int_{I_0}\varphi(k)\int_x^\infty\left|g(x, y, k)\right|dydk &\leq\int_{I_0}\varphi(k)\left(\int_x^\infty e^{2ik^2(x-y)}\bar{u}_ydy+\int_x^\infty |4ik^2|e^{2ik^2(x-y)}\bar{u}_ydy\right)dk\\
&\leq c(\int_{I_0}\varphi(k)^2dk)^{1/2}\int_x^\infty|\bar{u}_y|dy\lesssim||\varphi||_{L^2}||u||_{H^{1,1}}.\nonumber
\end{aligned}
\end{equation}
\end{proof}

 \noindent {\bf B.  Resolvent estimates}
\hspace*{\parindent}

With the following lemma, we obtain the resolvent estimates in Lemma \ref{lb}.
\begin{lemma}
Suppose that $X$ is a Banach space and consider the Volterra-type integral equation
\begin{equation}
u(x)=f(x)+(Tu)(x),\label{uux}
\end{equation}
on the space $C^{0}\left(\mathbb{R}^{+}, X\right)$, where $f \in C^{0}\left(\mathbb{R}^{+}, X\right)$ and $T$ is an integral operator on $C^{0}\left(\mathbb{R}^{+}, X\right)$. Let $f^{*}(x)=\sup _{y \geqslant x}\|f(y)\|_{X}$, and assume there is a nonnegative function $h \in L^{1}\left(\mathbb{R}^{+}\right)$ such that
$$
(T f)^*(x) \leqslant \int_{x}^{\infty} h(t) f^*(t) d t .
$$
Then the equation (\ref{uux}) has a unique solution for each $f$.  Moreover, the resolvent $(I-T)^{-1}$ obeys the bound
$$
\left\|(I-T)^{-1}\right\|_{\left(C^0(I, X)\right)} \leqslant \exp \left(\int_{0}^{\infty} h(t) d y\right).
$$
\end{lemma}
Our construction of the resolvent is based on the estimate
\begin{equation}
\left(T_0 f\right)^*(x) \leqslant \int_x^{\infty} \sigma(y) f^*(y)dy.
\end{equation}
In what follows, we define
$$
\sigma(y)=2|u_y(y)|+\left|iu_y(y)\right|^2,
$$
which implies $\left\|\sigma\right\|_{L^1}$ and $\left\|\sigma\right\|_{L^{2,2}}$  since  $  u \in  H^{3,3}(\mathbb{R})$.

\begin{lemma}\label{lb}
For each $k \in \mathbb{R}\cup i\mathbb{R}$ and $u(x,0)\in H^{3,3}(\mathbb{R})$, the operator $\left(I-T_0\right)^{-1}$ exists as a bounded operator from $C^0\left(\mathbb{R}^+\right) \otimes \mathbb{C}^{2}$ to itself. Moreover, $\left(I-T_{0}\right)^{-1}-I$ is an integral operator with a continuous integral kernel $L_0(x,y,k)$ such that $L_0(x,y,k)=0$ for $x>y$.
The integral kernel $L_0(x,y,k)$ satisfies the estimate
\begin{equation}
\left|L_0(x,y,k)\right| \leqslant \exp \left(\left\|\sigma\right\|_{L^1}\right) \sigma(y).\label{l0g}
\end{equation}
\end{lemma}
\begin{proof}
Since $T_{0}$ is a Volterra operator, we can obtain precise estimates on the resolvent through the Volterra series. The integral kernel $K_{0}(x, y, k)$ obeys the estimate $\left|K_{0}(x, y, k)\right| \leqslant \sigma(y)$. The operator
\begin{equation}
L_0 \equiv\left(I-T_0\right)^{-1}-I,
\end{equation}
is an integral operator with an integral kernel $L_{0}(x, y, k)$ given by
\begin{equation}
L_{0}(x, y, k)=\begin{cases}
\sum_{n=1}^{\infty} K_{n}(x, y, k), & x \leqslant y, \\
0, & x>y,
\end{cases}
\end{equation}
where
\begin{equation}
K_{n}(x, y, k)=\int_{x \leqslant y_1 \leqslant \cdots \leqslant y_{n-1}} K_0\left(x, y_1,k\right) K_{0}\left(y_{1}, y_2, k\right) \ldots K_0\left(y_{n-1}, y, k\right) d y_{n-1} \ldots d y_1,\nonumber
\end{equation}
and the estimate
\begin{equation}
\left|K_n(x, y, k)\right| \leqslant \frac{1}{(n-1) !}\left(\int_x^{\infty} \sigma(t)\right)^{n-1} \sigma(y),
\end{equation}
holds. Therefore follows the estimate (\ref{l0g}).
\end{proof}

Now, we use the above estimates to solve (\ref{wc}) and (\ref{wc1}). We first present the results here.
\begin{proposition}\label{po2}
We suppose that $u(x,0)\in H^{3,3}(\mathbb{R})$ and $I_0=\{k\in \mathbb{R}\cup i\mathbb{R}: |k|<1\}$, then there exists a unique solution $\mathbf{w} \in C^{0}\left(\mathbb{R}^{+}, L^{2}\left(I_{0}\right)\right) \cap L^{2}\left(\mathbb{R}^{+} \times I_{0}\right)$
for  (\ref{wc}) for each $k \in I_0$  satisfies:

{\rm (i)}\   The  map $u(x,0) \rightarrow \mathbf{w} (x,k)$ is Lipschitz continuous from $H^{3,3}(\mathbb{R})$ to $C^0\left(\mathbb{R}^+, L^2\left(I_0\right)\right) \cap L^2\left(\mathbb{R}^+ \times I_{0}\right)$;

{\rm  (ii) }\  The   map $u (x,0) \rightarrow \mathbf{w}_k(x,k)$ is Lipschitz continuous from $  H^{3,3}(\mathbb{R})$ to $C^0\left(\mathbb{R}^+, L^2\left(I_0\right)\right) \cap L^2\left(\mathbb{R}^+ \times I_0\right)$.

\end{proposition}

\subsubsection{Large-$k$ estimates}
\hspace*{\parindent}
First, we insert (\ref{L3}) to evaluate the derivative of $\omega_{11}^\pm$ in the following computation
\begin{equation}
\omega_{11,x}^\pm=\frac{i}{2}|u_x|^2\omega_{11}^\pm+k\bar{u}_x,
\end{equation}
from which, by canceling out some opposite terms, we derive
\begin{align}
&\omega_{11}^{\pm}(x, k)=1+\frac{i}{2} \int_x^{\pm \infty} \bar{u}_y(y) \int_y^{\pm \infty} e^{2ik^2(z-y)} \widetilde{u}(z)\omega_{11}^{\pm}(z,k) dz dy ,\label{d1}\\
&\omega_{21}^{\pm}(x, k)=-\frac{i}{2k} \bar{u}_x(x) \omega_{11}^{\pm}(x, k)-\frac{i}{2k} \int_x^{\pm \infty} e^{2ik^2(x-y)} \widetilde{u}(y) \omega_{11}^{\pm}(y, k) dy,\label{d2}
\end{align}
where
\begin{equation}
\widetilde{u}(x)=\bar{u}_{xx}(x)+\frac{i}{2}|u_x(x)|^2 \bar{u}_x(x).\label{ts}
\end{equation}
As $\omega_{21}^{\pm}$ does not occur in the equation for $\omega_{11}^{\pm}$, we solve the integral equation (\ref{d1}) for $\omega_{11}^{\pm}$, and use the solution to compute $\omega_{21}^{\pm}$.

We introduce the following notations
\begin{align}
&\eta_{11}^{\pm}(x, k)=\omega_{11}^{\pm}(x, k)-1,\label{e1}\\
&\eta_{21}^{\pm}(x, k)=\omega_{21}^{\pm}(x, k)+\frac{i}{2k} u_x(x),\label{e2}
\end{align}
which will facilitate the extraction of the leading order behavior of $\omega_{11}^{\pm}$ and $\omega_{21}^{\pm}$ for a large $k$.
From  (\ref{d1}), (\ref{d2}), (\ref{e1}) and (\ref{e2}), we conclude that
\begin{align}
&\eta_{11}^{\pm}(x, k)=F_{\pm}(x, k)+\left(T_{\pm} \eta_{11}^{\pm}\right)(x, k),\label{e11} \\
&\eta_{21}^{\pm}(x, k)=G_{\pm}(x, k)-\frac{i}{2k}u_x(x) \eta_{11}^{\pm}-\frac{i}{2k} \int_x^{\pm \infty} e^{2ik^2(x-y)} \widetilde{u}(y) \eta_{11}^{\pm}(y, k) d y,\label{e12}
\end{align}
where
\begin{align}
F_{\pm}(x,k) &=-\int_{x}^{\pm \infty} k\bar{u}(y) G_{\pm}(y, k) d y, \\
G_{\pm}(x, k) &=-\frac{i}{2k} \int_{x}^{\pm \infty} e^{2ik^2(x-y)} \widetilde{u}(y) d y, \\
\left(T_{\pm} f\right)(x, k) &=\frac{i}{2} \int_{x}^{\pm \infty} \bar{u}(y) \int_{y}^{\pm \infty} e^{2ik^2(y-z)} \widetilde{u}(z) f(z) d z.\label{t}
\end{align}

Let $\eta^{\pm}=\left(\eta_{11}^{\pm}, \eta_{21}^{\pm}\right)$ and $I_{\infty} \equiv\{k \in \mathbb{R}\cup i\mathbb{R}:|k|>1\}$. We utilize the similar method used in the above part to give the existence of $\omega^\pm$ and $\omega_k^\pm$.
To study $\eta_{11}^\pm$ and $\eta_{11,k}^\pm$, we take $k$-derivatives at both sides of (\ref{e11}) which gives
\begin{equation}
\eta_{11,k}^\pm=F_{\pm,k}+T_{\pm,k}\eta_{11}+T_\pm\eta_{11,k}^\pm.\label{ds1}
\end{equation}
With good estimates for $\eta_{11}^\pm$ and its derivatives at hands, it is straightforward to prove the corresponding estimates on $\eta_{21}^\pm$ using (\ref{e12}).

In the rest of this part, we drop the $\pm$ and derive estimates on $\eta_{11}^{+}$ and $\eta_{21}^+$. Similar procedures can be used to obtain estimates on $\eta_{11}^-$, $\eta_{21}^-$. For simplicity, we write $\eta_{11}$ for $\eta_{11}^+$, $F$ for $F_+$, $T$ for $T_+$, etc. Recall that $I_\infty=\{k \in \mathbb{R}\cup i\mathbb{R}:|k|>1\}$.
\vspace{2mm}

 \noindent {\bf A. Estimates on $F$, $G$, $T$ and their derivatives}

In this part, we mainly give estimates on $F$, $G$, and $T$, so as their derivatives.
\begin{lemma}
Suppose $u(x,0) \in H^{3,3}(\mathbb{R})$, the following terms define Lipschitz maps from $H^{3,3}(\mathbb{R})$ into $C^0\left(\mathbb{R}^+, L^2\left(I_{\infty}\right)\right) \cap L^{2}\left(\mathbb{R}^+ \times I_{\infty}\right)$
$$
{\rm (i)}\ \ G,\ \ \
{\rm (ii)}\ \  F,\ \ \
{\rm (iii)}\ \ \frac{\partial G}{\partial k},\ \ \
{\rm (iv)}\ \ \frac{\partial F}{\partial k}.
$$
\end{lemma}
\begin{proof}
Observing that
\begin{align}
&\|F\|_{C^{0}\left(\mathbb{R}^{+}, L^{2}\left(I_{\infty}\right)\right)}  \leqslant\|u\|_{W^{1,1}}\|kG\|_{C^{0}\left(\mathbb{R}^{+}, L^{2}\left(I_{\infty}\right)\right)}, \\
&\|F\|_{L^{2}\left(\mathbb{R}^{+} I_{\infty}\right)}  \leqslant\|u\|_{H^1}\|kG\|_{L^{2}\left(\mathbb{R}+\times I_{\infty}\right)},
\end{align}
it is obvious that (i) $\Rightarrow$ (ii). To prove (i) we pick $\varphi \in C_{0}^{\infty}\left(I_{\infty}\right)$ and mimic the proof of Lemma \ref{d11}. The Lipschitz continuity follows from the fact that $G$ is linear in $u_x$ and $F$ is bi linear in $u_x$.

One can similarly check that (iii) $\Rightarrow$ (iv), so it suffices to prove (iii). Next, we estimate the derivative of $G$. Since
\begin{equation}
\frac{\partial G}{\partial k}=h_{1}(x,k)+h_{2}(x, k),
\end{equation}
with
\begin{align}
&h_1(x,k)=m_1(x,k)+m_2(x,k),\\
&h_2(x,k)=s_1(x,k)+s_1(x,k),
\end{align}
where
\begin{align}
&m_1(x, k)=\frac{i}{4k^2}\int_x^{\infty}(x-y) e^{2ik^2(x-y)}\bar{u}_{yy}dy,\nonumber\\
&m_2(x, k)=-\frac{1}{8k^2}\int_x^{\infty}(x-y) e^{2ik^2(x-y)}|u_y|^2\bar{u}_ydy,\nonumber\\
&s_1(x, k)=2k\int_x^{\infty}(x-y) e^{2ik^2(x-y)}\bar{u}_{yy}dy,\nonumber\\
&s_2(x, k)=ik\int_x^{\infty}(x-y) e^{2ik^2(x-y)}|u_y|^2\bar{u}_ydy.\nonumber
\end{align}
We can estimate $m_2(x,k)$ as before except for $m_1(x,k)$, $s_1(x,k)$ and $s_2(x,k)$, to estimate which we integrate by parts to obtain
\begin{align}
&m_1(x, k)=\frac{i}{4k^2}\int_x^{\infty}\left(1+2ik^2(y-x)^{2}\right) \bar{u}_y(x) e^{2ik^2(x-y)}dy,\\
&s_1(x, k)=-\frac{i}{k}\int_x^{\infty}\left((x-y)\bar{u}_{yyy}-\bar{u}_{yy}\right)e^{2ik^2(x-y)}dy,\\
&s_2(x, k)=\frac{1}{2k}\int_x^{\infty}e^{2ik^2(x-y)}(|u_y|-(x-y)|u_y|^2\bar{u}_{yy}+2(x-y)|u_y|\bar{u}_{yy}|u_y|_y)dy.
\end{align}
Obviously, we can use the previous techniques to bound $m_{1}(x,k)$, $s_1(x,k)$ and $s_2(x,k)$ for $I_{\infty}$.
\end{proof}
Finally we elaborate on the mapping properties of the operators $T$ and $T_k$.
\begin{lemma}  For large  $|k|>1$, we have the following estimates
\begin{align}
&\left\|T\right\|_{L^2\left(\mathbb{R}^+ \times I_{\infty}\right)\rightarrow L^2\left(\mathbb{R}^+ \times I_{\infty}\right)} \leqslant\left\|\widetilde{u}\right\|_{L^{2,1}}\|u\|_{H^1},\label{T11}\\
&\left\|T\right\|_{L^2\left(\mathbb{R}^+ \times I_{\infty}\right)\rightarrow C^{0}\left(\mathbb{R}^+, L^2\left(I_{\infty}\right)\right)} \leqslant\left\|\widetilde{u}\right\|_{L^{2,1}}\|u\|_{H^{1,1}},\label{T12}\\
&\left\|T_k\right\|_{L^2\left(\mathbb{R}^+ \times I_{\infty}\right)\rightarrow L^2\left(\mathbb{R}^+ \times I_{\infty}\right)} \leqslant\left\|\widetilde{u}_x\right\|_{L^{2,1}}\|u\|_{H^1},\label{T3}\\
&\left\|T_k\right\|_{L^2\left(\mathbb{R}^+ \times I_{\infty}\right)\rightarrow C^{0}\left(\mathbb{R}^+, L^2\left(I_{\infty}\right)\right)} \leqslant\left\|\widetilde{u}_x\right\|_{L^{2,1}}\|u\|_{H^{2,2}},\label{T4}
\end{align}
\end{lemma}
\begin{proof}
We prove (\ref{T3}) and (\ref{T4}). The operator $T$ defined in (\ref{t}) has the integral kernel
\begin{equation}
K_{+}(x, y, k)=\begin{cases}
\left(\int_{x}^{y} e^{2ik^2(y-z)} \bar{u}_z(z) d z\right)\widetilde{u}(y) , & x<y, \\
0 & x>y.
\end{cases}
\end{equation}
From this computation, we have
\begin{equation}
(T f)^{*}(x) \leqslant\left(\|\bar{u}_x\|_{L^1} \int_{x}^{\infty}\left|\widetilde{u}(y)\right| d y\right) f^{*}(x).
\end{equation}
From the formula
\begin{equation}
\begin{aligned}
\frac{\partial T}{\partial k}[h](x,k)&=\frac{i}{2}\int_x^\infty \bar{u}_y(y) \int_{y}^{\infty}4ik(y-z) e^{2ik^2(y-z)} \widetilde{u}(z) h(z,k) dz dy\\
&=-\frac{i}{k}\int_x^\infty \bar{u}_y(y) \int_{y}^{\infty}e^{2ik^2(y-z)}(y-z)\left(\widetilde{u}(z) h(z,k)\right)_zdzdy
\end{aligned}
\end{equation}
Then we estimate
\begin{equation}
\left|\frac{\partial T}{\partial k}[h](x, k)\right| \lesssim\left\|\widetilde{u}_x\right\|_{L^{2,1}}\|h(\cdot, k)\|_{H^{1}} \int_{x}^{\infty}|u_y(y)| dy,
\end{equation}
with $\|h(\cdot, k)\|_{H^{1}}=1$. From which we conclude that
\begin{align}
&\left\|\frac{\partial T}{\partial k}\right\|_{C^0\left(\mathbb{R}^+, L^2\left(I_{\infty}\right)\right)} \lesssim\left\|\widetilde{u}_x\right\|_{L^{2,1}}\|u\|_{H^1}, \\
&\left\|\frac{\partial T}{\partial k}\right\|_{L^2\left(\mathbb{R}+\times I_{\infty}\right)} \lesssim\left\|\widetilde{u}_x\right\|_{L^{2,1}}\|u\|_{H^{2,2}}.
\end{align}
\end{proof}

 \noindent {\bf B. Resolvent Estimates.}
\hspace*{\parindent}

 As an analogy to the above deduction, we utilize Volterra estimates to construct the resolvent, from which an integral kernel is obtained. Finally, we extend the resolvent to a bounded operator on the spaces $C^{0}\left(\mathbb{R}^{+}, L^{2}\left(I_{\infty}\right)\right)$ and $L^{2}\left(\mathbb{R}^{+} \times I_{\infty}\right)$.

\begin{lemma}\label{8}
We suppose that $u(x,0)\in H^{3,3}(\mathbb{R})$. The resolvent $(I-T)^{-1}$ exists as a bounded operator in $C^{0}\left(\mathbb{R}^{+}\right)$ and the operator $L \equiv(I-T)^{-1}-I$ is an integral operator with an integral kernel $L(x, y, k)$ such that $L(x,y,k)=0$ for $x>y$, continuous in $(x,y,k)$ for $x<y$, and obeys the estimates
\begin{equation}
|L(x, y, k)| \leqslant \exp \left(\|u\|_{W^{1,1}}\left\|\widetilde{u}\right\|_{L^{1}}\right)\|u\|_{w^{1,1}}\left|\widetilde{u}(y)\right| .
\end{equation}
\end{lemma}
\begin{proof}
The rest of the proof is parallel to the proof of Lemma \ref{lb}.
\end{proof}
According to the above analysis, we can estimate $\eta_{11,k}$.

\begin{proposition}\label{po3}
Suppose that  $u(x,0)\in H^{3,3}(\mathbb{R})$ and $k \in I_{\infty}$,  the the equation (\ref{ds1}) admits a unique solution $\eta_{11} \in C^0\left(\mathbb{R}^+, L^2\left(I_{\infty}\right)\right) \cap L^{2}\left(\mathbb{R}^+ \times I_{\infty}\right)$.  Moreover,

{\rm (i) }\  $u(x,0) \rightarrow\eta_{11} (x,k) $ is  Lipschitz continuous as a map from $H^{3,3}(\mathbb{R})$ to $C^0\left(\mathbb{R}^+, L^2\left(I_{\infty}\right)\right) \cap L^2\left(\mathbb{R}^+ \times I_{\infty}\right)$.

{\rm  (ii)} \  $u (x,0)  \rightarrow\eta_{11,k} (x,k) $ is Lipschitz continuous
as a map from $H^{3,3}(\mathbb{R})$  to  $C^0\left(\mathbb{R}^+, L^2\left(I_{\infty}\right)\right) \cap L^2\left(\mathbb{R}^+ \times I_{\infty}\right)$.
\end{proposition}

\subsection{ Symmetries  and asymptotics of eigenfunctions }
\hspace*{\parindent}
According to the Lax pair (\ref{Lax pair}) and the transformation  in (\ref{psi}), there exists
a matrix function $S(k)$ such that
\begin{equation}
\omega^-(x,k)=\omega^+(x,k)e^{-ik^2x\widehat{\sigma}_{3}}S(k).\label{omega}
\end{equation}
Below, we characterize the symmetry properties of $ \omega^{\pm}(x,k)$.

\begin{proposition} \label{prop2}
  $ \omega^{\pm}(x,k)$ and S(k)   satisfy  the symmetry relations
\begin{align}
 &\omega^{\pm}(x,k) =\sigma_2 \overline{\omega^{\pm}(x,\bar {k} )} \sigma_2,\ \ \ \omega^{\pm}(x, k) =\sigma_3  \omega^{\pm}(x,-k) \sigma_3,\label{dcx}\\
 &S( k) =\sigma_2 \overline{S( \bar {k} )} \sigma_2,\ \ \ S(  k) =\sigma_3 S( -k) \sigma_3,\label{dcx2}
\end{align}
where
$$
\sigma_2=\begin{pmatrix}
0&-i\\
i&0
\end{pmatrix},\ \ \
\sigma_3=\begin{pmatrix}
1&0\\
0&-1
\end{pmatrix}.
$$
\label{ppp}
\end{proposition}

From the symmetries  (\ref{dcx2}),  $S(k)$ admits the form
\begin{equation}
 S(k)=\begin{pmatrix}
a(k)&b(k) \\[3pt]
-\overline{b(\overline{k})}&\overline{a(\overline{k})}
\end{pmatrix},\label{omegagx}
\end{equation}
where $a(k)$ and $b(k)$ are given by solving  (\ref{omega}), i.e.,
\begin{align}
&a(k)=\mathrm{det}(\omega_1^-,\omega_2^+)=\omega_{11}^-(0,k)\overline{\omega_{11}^+(0,\bar{k})}-\omega_{21}^-(0,k)\overline{\omega_{21}^+(0,\bar{k})},\label{ak}\\
&b(k)=\mathrm{det}(\omega_2^-,\omega_2^+)=\overline{\omega_{21}^-(0,\bar{k})}\ \overline{  \omega_{11}^+(0,\bar{k})} -\overline{\omega_{21}^+(0,\bar{k})}\ \overline{\omega_{11}^-(0,\bar{k})},\label{bk}
\end{align}
where we denote the matrix eigenfunctions  $\omega^\pm(x,k) =( \omega_{ij}(x,k))_{2\times 2} $.
In terms of the solutions $\eta_{11}^{\pm}$ and $\eta_{21}^{\pm}$, the functions $a(k)$ and $b(k)$ defined in $(\ref{ak})$ and (\ref{bk}) are expressed as
\begin{equation}
a(k)-1=a_1(k)+a_2(k), \quad b(k)=b_1(k)+b_2(k),
\end{equation}
where
\begin{align}
&a_1(k)=\overline{\eta_{11}^+(0, \bar{k})}+\eta_{11}^-(0, k)+\eta_{11}^-(0, k) \overline{\eta_{11}^+(0, \bar{k})}, \nonumber\\
&a_2(k)=\frac{1}{4k^2}|u_x(0)|^{2}+\frac{i}{2k} \overline{u_x(0)} \ \overline{\eta_{21}^+(0,\bar{k})}-\frac{i}{2k} u_x(0) \eta_{21}^-(0, k)+\overline{\eta_{21}^+(0, \bar{k})} \eta_{21}^-(0,k),\nonumber
\end{align}
and
\begin{align}
&b_1(k)=\left(-\overline{\eta_{21}^{+}(0, \bar{k})}+\frac{i}{2k} u_x(0) \overline{\eta_{11}^+(0, \bar{k})}\right)+\left(\overline{\eta_{21}^{-}(0, k)}-\frac{i}{2k} u_x(0)\eta_{11}^-(0,k)\right),\nonumber \\
&b_2(k)=\overline{\eta_{11}^+(0, \bar{k})}\ \overline{ \eta_{21}^-(0,\bar{k})}-\overline{\eta_{11}^-(0, \bar{k})} \ \overline{ \eta_{21}^+(0, \bar{k})}.\nonumber
\end{align}

To prove that $u(x,0) \rightarrow kr(k)$ is Lipschitz continuous from $H^{3,3}(\mathbb{R})$ to $L^{2}\left(I_{\infty}\right)$, it suffices to show that $u(x,0)\rightarrow kb(k)$ has the same Lipschitz continuity.
\begin{proposition}
\label{p5}
Suppose that $u(x,0)\in H^{3,3}(\mathbb{R})$, then $u(x,0)\rightarrow kb(k)$ is Lipschitz from $H^{3,3}(\mathbb{R})$ to $L^{2}\left(I_{\infty}\right)$.
\end{proposition}
\begin{proof}
Here we give the proof of the case $k\in\mathbb{R}$, which the case $k\in i\mathbb{R}$ is analogous. We rewrite (\ref{e12}) as
\begin{equation}
\eta_{21}^{\pm}=-\frac{i}{2k} \bar{u}_x \eta_{11}^{\pm}-\frac{i}{2k} \int_{x}^{\pm \infty} e^{2 i k^2(x-y)} \widetilde{u}(y)\left(1+\eta_{11}^{\pm}\right) dy.\label{eta21}
\end{equation}
Setting $x=0$, (\ref{eta21}) becomes
\begin{equation}
\eta_{21}^{\pm}(0,k)+\frac{i}{2k} \bar{u}_x(0) \eta_{11}^{\pm}(0,k)=Q^\pm,
\end{equation}
with
\begin{equation}
Q^\pm=-\frac{i}{2k} \int_{0}^{\pm \infty} e^{-2ik^2y} \widetilde{u}(y)\left(1+\eta_{11}^{\pm}\right)dy.
\end{equation}
And for $kb(k)$, it is in a new form as
\begin{equation}
kb(k)=\overline{Q^{+}(k)}+\overline{Q^{-}(k)}-\overline{\eta_{11}^{-}(0, k)Q^{+}(k)}+\overline{\eta_{11}^{+}(0,k)Q^{-}(k)}.
\end{equation}
According to the following estimate
\begin{equation}
\left\|Q^{\pm}\right\|_{L^{2}\left(I_{\infty}\right)} \leqslant \left\|\widetilde{u}\right\|_{L^{2}(\mathbb{R})}\left(1+\left\|\eta_{11}^{\pm}\right\|_{L^{2}\left(\mathbb{R}^{\pm} \times I_\infty\right)}\right),
\end{equation}
it suffices to show that $u(x,0)\rightarrow Q^{\pm}(k)$ is Lipschitz from $H^{3,3}(\mathbb{R})$ to $L^{2}\left(I_{\infty}\right)$,  from which we obtain $u(x,0)\rightarrow kb(k)$ is Lipschitz from $H^{3,3}(\mathbb{R})$ to $L^{2}\left(I_{\infty}\right)$.
\end{proof}

\begin{lemma} \label{l3}
The $\omega^\pm(x,k)$   admit the asymptotics
\begin{align}
&\omega^\pm(x,k)\sim I,  \quad  as\ k \rightarrow \infty, \\
&\omega^\pm(x,k)\sim \varphi_0(x),\quad as\ k \rightarrow 0. \label{asymp11}
\end{align}

The $a(k)$ and $b(k)$ admit the asymptotics
\begin{align}
&a(k)=1+\mathcal{O}(k^{-1}), \quad b(k)=\mathcal{O}(k^{-1}), \quad  as\ k \rightarrow \infty, \label{asymp1}  \\
&a(k)=e^{-\frac{i}{2} d_0 \sigma_{3}}\left(1+\mathcal{O}\left(k^{3}\right)\right),\  \ b(k)=\mathcal{O}\left(k^{3}\right),   \quad as\ k \rightarrow 0,\label{asymp2}
\end{align}
where
\begin{equation}
d_0=\int_{-\infty}^{+\infty} |u_y|^2\left(y, t\right) d y.
\end{equation}
\end{lemma}

\begin{proof} The asymptotics (\ref{asymp1}) can be obtained by using (\ref{ak})-(\ref{bk}).
The $\omega^\pm(x,k)$ and the spectral data $ a(k), b(k) $ have the asymptotics at $k=0$ \cite{xujian2}.
\end{proof}

\subsection{Scattering map from initial data to scattering coefficients}
\hspace*{\parindent}
First, we define the reflection coefficient
\begin{equation}
r(k)=\frac{b(k)}{a(k)}. \label{rs}
\end{equation}
From $\det S(k)=1$, by using the symmetry (\ref{dcx}), we obtain
 \begin{align}
&1+ |r(k)|^2   =\frac{1}{ |a(k)|^2 }, \ \ \ k\in \mathbb{R}, \nonumber\\
 &1- |r(k)|^2  =\frac{1}{ |a(k)|^2 }, \ \ \ k\in i\mathbb{R}. \nonumber
\end{align}
In the absence of spectral singularities, there also exists a $c\in(0,1)$ such that $c<|a(k)|<1/c$ for $k\in  \Sigma =\mathbb{R}\cup i\mathbb{R}$.
To cope with the following derivation, we assume our initial data satisfies the assumption below.
\begin{Assumption} \label{asump1}
The initial data $u(x,0)\in H^{3,3}(\mathbb{R})$ generates generic scattering data that satisfies

1. $a(k)$ has no zeros on $\Sigma$,

2. $a(k)$ only has a finite number of simple zeros.

\end{Assumption}

Combining the above results, we immediately obtain the following theorem on $r(k)$.
\begin{theorem}\label{111}
For any given $u(x,0) \in H^{3,3}(\mathbb{R})$, we have $r(k) \in H^{1,1}(\Sigma)$.
\end{theorem}
\begin{proof}   By using (\ref{asymp1})-(\ref{asymp2}) and Definition  (\ref{rs}), we have
$$r(k) =\mathcal{O}(k^3), \ k\rightarrow 0, \ \  r(k) =\mathcal{O}(k^{-1}), \  k\rightarrow \infty.$$
Again $a(k)$ has no zeros on $\Sigma$  by Assumption \ref{asump1} above, we conclude  that $r(k)$ has no spectral singularities on the contour  $\Sigma$.

Combining   (\ref{ak})-(\ref{bk}),  Proposition \ref{po2} and \ref{po3}, we derive   that
 $$a(k), b(k) \in L^{\infty}(\Sigma)\cap{L^{2}}(\Sigma), $$
by definition  (\ref{rs}),  we then  have
 \begin{align}
 r(k)\in L^{2}(\Sigma). \label{rse}
 \end{align}

Noting that
 \begin{align}
&a'(k)=\mathrm{det}(\omega_{1,k}^-,\omega_2^+)+\mathrm{det}(\omega_{1}^-,\omega_{2,k}^+),\\
&b'(k)=\mathrm{det}(\omega_{2,k}^-,\omega_2^+)+\mathrm{det}(\omega_2^-,\omega_{2,k}^+),
\end{align}
by  Proposition \ref{po2} and \ref{po3}, we obtain  that  $a'(k), b'(k) \in L^2(\Sigma)$.
So we have
\begin{equation}
r'(k)=\frac{b'(k) a(k) - a'(k)b(k)}{a^2(k)} \in L^2(\Sigma),
\end{equation}
which, combined with (\ref{rse}) and Proposition \ref{p5}, finally gives the map $u(x,0)\rightarrow r(k)$ is Lipschitz continuous from $H^{3,3}(\mathbb{R})$ into $H^{1,1}(\Sigma)$.

\end{proof}

\section{Construction of the  RH  problem}
\label{sec:section3}
\hspace*{\parindent}
Since $a(k)$ is an even function, each zero $k_j$ of $a(k)$ is accompanied by another zero at $-k_j$. We assume that $a(k)$ has $2 N$ simple zeros $\left\{k_j\right\}_{j=1}^{2N} \subset D^+$ such that $\left\{k_j\right\}_{j=1}^{N}$ belong to the first quadrant and $k_{j+N}=-k_j, j=1, \ldots N$. Similarly, $\overline{a(\overline{k})}$ has $2 N$ simple zeros $\left\{\overline{k}_j\right\}_{j=1}^{2N} \subset D^-$ such that $\left\{\overline{k}_j\right\}_{j=1}^{N}$ belong to the second quadrant and $\overline{k}_{j+N}=-\overline{k}_j, j=1, \ldots N$.
They can be denoted by
 $$
 \mathcal{K}=\{k_j, j=1, \cdots, 2N\} \subset D^+, \ \ \mathcal{\overline{K}}=\{\overline{k}_j, j=1, \cdots, 2N\} \subset D^-.
 $$
We define a sectionally meromorphic matrix
\begin{equation}
M(k;x,t)=\begin{cases}
\left(\frac{\omega^-_1(k;x, t)}{a(k)}, \omega^+_2(k;x, t)\right),\quad k\in D^+,\\[8pt]
\left(\omega^+_1(k;x, t),\frac{\omega^-_2(k;x, t)}{\overline{a(\bar{k})}}\right),\quad k\in D^-,
\end{cases}
\end{equation}
which solves the following  RH  problem. \\
\noindent\textbf{RH Problem 3.1.} Find a matrix-valued function $M(k)=M(k;x,t)$ which satisfies:

(a) \emph{Analyticity}: $M(k)$ is meromorphic in $\mathbb{C} \backslash \Sigma$  and has single poles;

(b) \emph{Symmetry}: $\overline{M(\bar{k})}=\sigma_2 M(k) \sigma_2$;

(c) \emph{Jump condition}: $M(k)$ has continuous boundary values $M_{\pm}$ on $\Sigma$ and
\begin{equation}
M_+(k)=M_-(k) V(k), \quad k \in \Sigma,
\end{equation}
where
\begin{equation}
V(k)=\begin{pmatrix}
1+ r(k) \overline{ r(\bar k)} & \overline{r(\bar k)}e^{-2it\theta(k)} \\
r(k)e^{2it\theta(k)} &1
\end{pmatrix};\label{Vk1}
\end{equation}

(d) \emph{Asymptotic behaviors}:
\begin{equation}
\begin{cases}
M(k)=I+\mathcal{O}\left(k^{-1}\right), & k \rightarrow \infty,\\
M(k)=e^{-\frac{i}{2} c_- \sigma_{3}}\left(I+kU+\mathcal{O}\left(k^{2}\right)\right) e^{\frac{i}{2} d_0\sigma_3},  & k \rightarrow 0;
\end{cases}
\end{equation}

(e) \emph{Residue conditions}: $\mathrm{M}$ has simple poles at each point in
$\mathcal{K}$ and $\mathcal{\overline{K}}$ with:
\begin{align}
&\underset{k=k_j}{\operatorname{Res}} M(k)=\lim _{k \rightarrow k_j} M(k)R(k_j), \label{35} \\
&\underset{k=\overline{k}_j}{\operatorname{Res}} M(k)=\lim _{k \rightarrow \overline{k}_j} M(k)\sigma_2\overline{R(\bar k_j) } \sigma_2, \label{36}
\end{align}
where $R(k_j)$ is the nilpotent matrix
$$
R(k_j)=\begin{pmatrix}
0 & 0 \\
c_j e^{2it\theta(k_j)} & 0
\end{pmatrix},
$$
and $\theta(k)=k^2\frac{x}{t}+\eta^2(k)$, $\eta(k)=\sqrt{\alpha}(k-\frac{\beta}{2 k})$.  Intuitively,
 the contour, poles and the stationary phase points of $M(k)$ are illustrated in Figure \ref{din31}.
%%%%%%%%%%%%%% MODE
\begin{figure}[H]
\begin{center}
\begin{tikzpicture}[node distance=2cm]
%\draw[dashed](0,0)circle(2cm);
\draw [ ](-3,0)--(3,0)  node[right, scale=1] {Re $k$};
\draw [-latex](-3,0)--(1.2,0);
\draw [-latex](0,0)--(-1.2,0);
\draw [ ](0,-2.5)--(0,2.5)  node[above, scale=1] {Im $k$};
\draw [-latex](0, 2.5)--(0,1);
\draw [-latex](0, -2.5)--(0,-1);
%\draw [color=blue!90 ] (-2.2,0)--(2.2,0);
%\draw [color=blue!90 ] (0,-2.2)--(0,2.2);
%\node  [right]  at (-0.5,2.2) {$z_2$};
%\node  [right]  at (-0.5,-2.2) {$z_4$};
\node  [below]  at (0.2,0) { $0$ };
%\node  [below]  at (-2.2,0) {$z_3$};
\draw[fill,blue] (2,0) circle [radius=0.04];
\draw[fill,blue] (-2,0) circle [radius=0.04];
\draw[fill,blue] (0,2) circle [radius=0.04];
\draw[fill,blue] (0,-2) circle [radius=0.04];
%\node  [below]  at (0.3,-1) {$I_{41}$};
%\node  [below]  at (0.3,1.4) {$I_{21}$};
%\node  [below]  at (0.3,3) {$I_{22}$};
%\node  [below]  at (0.3,-2.7) {$I_{42}$};
%\node  [below]  at (-0.9,0.5) {$I_{31}$};
%\node  [below]  at (-2.7,0.5) {$I_{32}$};
%\node  [below]  at (2.7,0.5) {$I_{12}$};
%\node  [below]  at (0.9,0.5) {$I_{11}$};
%\draw[fill] (2,0) circle [radius=0.04];
%\draw[fill] (-2,0) circle [radius=0.04];
%\draw[fill] (0,2) circle [radius=0.04];
%\draw[fill] (0,-2) circle [radius=0.04];
\draw[fill,red] (1,1.2) circle [radius=0.04];
\draw[fill,red] (2,2) circle [radius=0.04];
\draw[fill,red] (2,-2) circle [radius=0.04];
\draw[fill,red] (-2,-2) circle [radius=0.04];
\draw[fill,red] (-2,2) circle [radius=0.04];
\draw[fill,red] (1,-1.2) circle [radius=0.04];
\draw[fill,red] (-1,-1.2) circle [radius=0.04];
\draw[fill,red] (-1,1.2) circle [radius=0.04];
\end {tikzpicture}
\end{center}
\caption{\small  The RH problem for $M(k)$, where the jump contour $\Sigma=\mathbb{R} \cup i\mathbb{R}$, the red dots  ( \textcolor{red}{$\centerdot$ } ) represent the poles, and the blue  dots  ( \textcolor{blue}{$\centerdot$ } ) represent the stationary phase points.     }
\label{din31}
\end{figure}
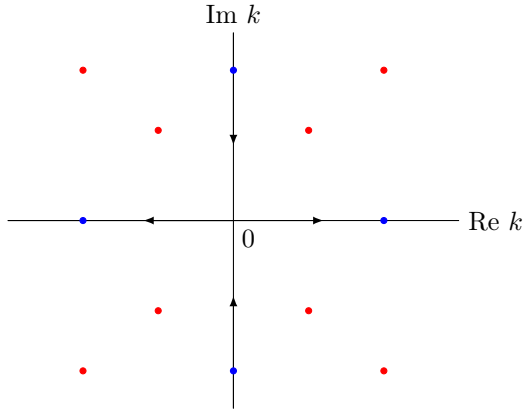

We denote the norming constant $c_j=c(k_j)=b(k_j)/a^{\prime}(k_j)$ and the collection $\mathcal{S}=\left\{k_j, c_j\right\}_{j=1}^{2N}$ is called
the scattering data.
The solution of the FL equation (\ref{cs2}) can be expressed by
\begin{equation}
u_{x}(x, t)=m(x, t) e^{-i \int_{-\infty}^{x}|m\left(x^{\prime}, t\right)|^{2} d x^{\prime}},\label{gs1}
\end{equation}
where
\begin{equation}
m(x, t)=2i\lim _{k \rightarrow \infty}(k M(x, t, k))_{12},
\end{equation}
we have
\begin{equation}
|u_{x}(x, t)|=4|m(x, t)|.\label{mgj}
\end{equation}

\section{Existence and uniqueness of  solution  to the  RH problem}\label{sec:section31}
\hspace*{\parindent}
In this section, we  prove the  existence and uniqueness of   solution for the  initial value problem   (\ref{cs2})-(\ref{cz}).
 We first transform the basic  {  RH Problem 3.1}  to  an  equivalent  RH problem---{  RH Problem 4.1},
whose  existence and uniqueness are  affirmed   by  proving  that the associated  Beals-Coifman initial integral equation   admits  a unique solution for
$r(k)\in H^{1,1}(\Sigma)$. We further  show that the solution $M(k)$ of the {RH Problem 3.1} obeys the Lax pair (\ref{Lax pair})
as a function of $x$ and $t$, and give reconstruction formulas for $u(x,t)$ in terms of the solution $\mu(x,k)$ of the Beals-Coifman initial integral equation.

  To replace the residue conditions (\ref{35})-(\ref{36}) of the {RH Problem 3.1} with Schwartz invariant jump conditions,
for all  poles  $k_j\in   \mathbb{D}^+$ and  $\bar k_j\in  \mathbb{D}^-$,  we trade their  residues
into the corresponding  jumps  on   circles  $\gamma_j$  centered at $k_j$ and
 $\overline{\gamma}_j$  centered at $\bar k_j$ respectively.   The radii of these circles
  are chosen  sufficiently small  such that they are  disjoint from all other circles.
  Please see Figure \ref{din1}.
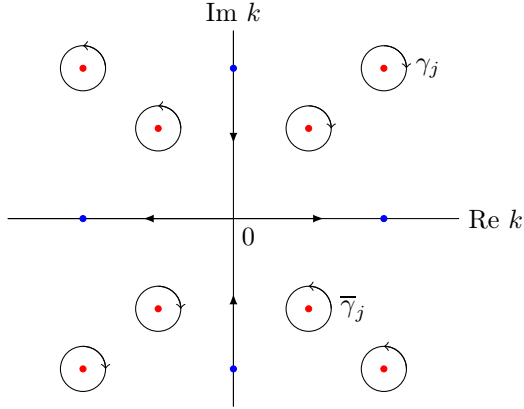
\begin{figure}[H]
\begin{center}
\begin{tikzpicture}[node distance=2cm]
\draw [ ](-3,0)--(3,0)  node[right, scale=1] {Re $k$};
\draw [-latex](-3,0)--(1.2,0);
\draw [-latex](0,0)--(-1.2,0);
\draw [ ](0,-2.5)--(0,2.5)  node[above, scale=1] {Im $k$};
\draw [-latex](0, 2.5)--(0,1);
\draw [-latex](0, -2.5)--(0,-1);
\node  [below]  at (0.2,0) { $0$ };
\draw[fill,blue] (2,0) circle [radius=0.04];
\draw[fill,blue] (-2,0) circle [radius=0.04];
\draw[fill,blue] (0,2) circle [radius=0.04];
\draw[fill,blue] (0,-2) circle [radius=0.04];
\draw[fill,red] (1,1.2) circle [radius=0.04];
\draw[fill,red] (2,2) circle [radius=0.04];
\draw[fill,red] (2,-2) circle [radius=0.04];
\draw[fill,red] (-2,-2) circle [radius=0.04];
\draw[fill,red] (-2,2) circle [radius=0.04];
\draw[fill,red] (1,-1.2) circle [radius=0.04];
\draw[fill,red] (-1,-1.2) circle [radius=0.04];
\draw[fill,red] (-1,1.2) circle [radius=0.04];
\draw [] (2,2)circle(0.3cm);
 \draw [  <-  ]  (2.3,2) to  [out=90, in=0]  (2, 2.3);	
\node  [right]  at (2.3,2) {$\gamma_j$};
\draw [](2,-2)circle(0.3cm);
 \draw [ -> ]  (2.3,-2) to  [out=90, in=0]  (2, -1.7);
\draw [](-2,-2)circle(0.3cm);
 \draw [  <-  ]  (-1.7,-2) to  [out=90, in=0]  (-2, -1.7);
\draw [] (-2,2)circle(0.3cm);
 \draw [  ->  ]  (-1.7,2) to  [out=90, in=0]  (-2, 2.3);
 \draw [] (1,1.2)circle(0.3cm);
 \draw [  <-  ]  (1.3,1.2) to  [out=90, in=0]  (1, 1.5);	
\node  [right]  at (1.3,-1.2) {$\overline{\gamma}_j$};
\draw [](1,-1.2)circle(0.3cm);
 \draw [ -> ]  (1.3,-1.2) to  [out=90, in=0]  (1, -0.9);
\draw [](-1,-1.2)circle(0.3cm);
 \draw [  <-  ]  (-0.7,-1.2) to  [out=90, in=0]  (-1, -0.9);
\draw [] (-1,1.2)circle(0.3cm);
 \draw [   ->  ]  (-0.7,1.2) to  [out=90, in=0]  (-1, 1.5);
\end {tikzpicture}
\end{center}
\caption{\small  The jump contour $\Sigma^{\prime}$ of RH problem for $M(k)$.}
\label{din1}
\end{figure}

Let
$$
\Sigma^{\prime}=\Sigma \cup\left\{ \gamma_j\right\}_{j=1}^{2N} \cup\left\{ \overline{\gamma}_j\right\}_{j=1}^{2N},
$$
we then transform the {RH Problem 3.1}  to an equivalent  RH problem (for simplicity, we still use the notation $M(k;x, t)$) on the jump contour $\Sigma^{\prime}$.

\noindent\textbf{RH Problem 4.1.}   Find a matrix-valued function $M(k)=M(k;x,t)$  which satisfies

(a) \emph{Analyticity}: $M(k;x,t)$ is analytic in $\mathbb{C} \backslash \Sigma^{\prime}$;

(b) \emph{Jump condition}: $M(k;x,t)$ has continuous boundary values $M_{\pm}(k;x,t)$ on $\Sigma^{\prime}$  and
\begin{equation}
M_+(k;x,t)=M_-(k;x,t) J(k),\label{csm1}
\end{equation}
where
\begin{numcases}{J(k) =}
\begin{pmatrix}
1+|r(k)|^2   & \overline{r(k)}  e^{-2it\theta(k)} \\
r(k) e^{2it\theta(k)}  & 1
\end{pmatrix}, \ \ k\in \mathbb{R},\label{herm11}\\
 \begin{pmatrix}
1-|r(k)|^2   &- \overline{r(k)} e^{-2it\theta(k)}  \\
r(k) e^{2it\theta(k)}  & 1
\end{pmatrix}, \ \ k\in i\mathbb{R},\label{herm21}\\
\begin{pmatrix}
1 & 0 \\
\frac{c_j e^{2it\theta(k)}}{k - k_j} & 1
\end{pmatrix}, \ \ \ \ \   k \in  \gamma_j \label{herm31}\\
\begin{pmatrix}
1 & \frac{  \overline{c}_j e^{-2 i t\theta(k)   }}{k - \overline{k}_j}, \\
0 & 1
\end{pmatrix}, \ \ \ \ \  k \in  \overline{\gamma}_j;\label{herm41}
\end{numcases}

(c) \emph{Asymptotic behaviors}:
\begin{equation}
\begin{cases}
M(k)=I+\mathcal{O}\left(k^{-1}\right), & k \rightarrow \infty,\\
M(k)=e^{-\frac{i}{2} c_- \sigma_{3}}\left(I+kU+\mathcal{O}\left(k^{2}\right)\right) e^{\frac{i}{2} d_0\sigma_3},  & k \rightarrow 0.
\end{cases}
\end{equation}

We suppress the dependence of all quantities on $t$  to  emphasize the role of $x$ (later $x, t$).
 To study the existence and the uniqueness of the solution
to the RH Problem 4.1, we need the following RH problem with the jump contour $\Sigma'$.

\noindent\textbf{RH Problem 4.2.} Find a  matrix-valued  function $M(k;x)$  with  the following properties:

(a) \emph{Analyticity}: $M(k;x)$ is analytic in $\mathbb{C} \backslash \Sigma^{\prime}$.

(b) \emph{Jump condition}: $M(k;x)$   has continuous boundary values $M_{\pm}(k;x)$ on $\Sigma^{\prime}$ and
\begin{equation}
M_+(x, k)=M_-(x, k) J(k),\label{csm}
\end{equation}
where
\begin{numcases}{J(k) =}
\begin{pmatrix}
1+|r(k)|^2   & \overline{r(k)}  e^{-2ik^2x} \\
r(k) e^{2ik^2x}  & 1
\end{pmatrix}, \ \ k\in \mathbb{R},\label{herm1}\\
\begin{pmatrix}
1-|r(k)|^2   &- \overline{r(k)} e^{-2ik^2x}  \\
r(k) e^{2ik^2x}  & 1
\end{pmatrix}, \ \ k\in i\mathbb{R},\label{herm2}\\
\begin{pmatrix}
1 & 0 \\
\frac{c_j e^{2ik^2x}}{k - k_j} & 1
\end{pmatrix}, \ \ \ \ \   k \in  \gamma_j \label{herm3}\\
\begin{pmatrix}
1 & \frac{  \overline{c}_j e^{-2 i k^2 x}}{k - \overline{k}_j}, \\
0 & 1
\end{pmatrix}, \ \ \ \ \  k \in  \overline{\gamma}_j.\label{herm4}
\end{numcases}
(c) \emph{Normalization}

{\bf Type I.}  $M(k;x)=I+\mathcal{O}\left(k^{-1}\right)$ as $k \rightarrow \infty$,

{\bf Type II.}  $M(k;x)=\mathcal{O}\left(k^{-1}\right)$ as $k \rightarrow \infty$,\\
where Type  I with inhomogeneous boundary condition is needed for the reconstruction of the potential $u_x(x)$.
Type II  is  used to show the existence and uniqueness of the solution of the RH Problem 4.2 with Type I.

\subsection{Beals-Coifman integral equation}
\hspace*{\parindent}
The unique solvability of the RH Problem 4.2 is equivalent to the unique solvability of Beals-Coifman integral equation.
We decompose the jump matrix   (\ref{csm}) as
\begin{align}
J(k)=(I-w_-)^{-1}(I+w_+),\label{ewet}
\end{align}
 and let
\begin{align}
\mu=M_{+}\left(I+w_+\right)^{-1}=M_-\left(I-w_-\right)^{-1},\label{wet}
\end{align}
from which we can write down $\mu$ explicitly:
For $k \in \Sigma$
\begin{equation}
\mu(x, k)=\begin{cases}
\begin{pmatrix}
\frac{\omega_1^-(x, k)}{a(k)} & \omega_2^{+}(x, k)
\end{pmatrix}
\begin{pmatrix}
1 & 0 \\
e^{2 i k^{2} x}{r}(k) & 1
\end{pmatrix}, \\
\begin{pmatrix}
\omega_1^{+}(x,k) & \frac{\omega_2^-(x, k)}{\overline{a(\bar{k})}}
\end{pmatrix}
\begin{pmatrix}
1 & e^{-2ik^2x} \overline{r(\bar k)} \\
0 & 1
\end{pmatrix}.
\end{cases}
\end{equation}
and for $k \in \gamma_j$
\begin{equation}
\mu(x, k)=\begin{cases}
\begin{pmatrix}
\frac{\omega_1^-(x, k)}{a(k)} & \omega_2^+(x,k)
\end{pmatrix}
\begin{pmatrix}
1 & 0 \\
-\frac{c_j e^{2 i x k^{2}}}{k-k_j} & 1
\end{pmatrix}, \\
\begin{pmatrix}
\frac{\omega_{11}^-(x, k)}{a(k)}-\frac{c_j e^{2ik^2x} \omega_{12}^{+}(x, k)}{k-k_j} & \omega_{12}^{+}(x, k) \\
\frac{\omega_{21}^-(x,k)}{a(k)}-\frac{c_j e^{2ik^2x} \omega_{22}^{+}(x, k)}{k-k_j} & \omega_{22}^{+}(x,k)
\end{pmatrix}
\begin{pmatrix}
1 & 0 \\
0 & 1
\end{pmatrix},
\end{cases}
\end{equation}
and for $k \in \overline{\gamma}_j$
\begin{equation}
\mu(x, k)=\begin{cases}
\begin{pmatrix}
\omega_{11}^{+}(x, k) & \frac{\omega_{12}(x, k)}{a(k)}+\frac{\bar{c}_j e^{-2 i k^{2}x} \omega_{11}^{+}(x,k)}{k-\bar{k}_j} \\
\omega_{21}^{+}(x, k) & \frac{\omega_{22}^{-}(x, k)}{a(k)}+\frac{\bar{c}_j e^{-2 ik^{2}x} \omega_{21}^{+}(x, k)}{k-\bar{k}_j}
\end{pmatrix}
\begin{pmatrix}
1 & 0 \\
0 & 1
\end{pmatrix}, \\
\begin{pmatrix}
\omega_1^{+}(x,k) & \frac{\omega_2^-(x, k)}{\overline{a(\bar{k})}}
\end{pmatrix}
\begin{pmatrix}
1 & \frac{\bar{c}_j e^{-2ik^2x}}{k -\bar{k}_j} \\
0 & 1
\end{pmatrix}.
\end{cases}
\end{equation}

From (\ref{wet}), we obtain
$$M_+-M_-=\mu (w_++w_-),$$
then Plemelj formula gives the following Beals-Coifman integral equation for Type I normalization in {RH Problem 4.2}
\begin{equation}
\mu=I+\mathcal{C}_{w} \mu=I+C_{\Sigma^{\prime}}^{+}\left(\mu w_{x}^-\right)+C_{\Sigma^{\prime}}\left(\mu w_+\right),\label{bcm}
\end{equation}
where $I$ is the $2 \times 2$ identity matrix. And for Type II normalization:
\begin{equation}
\mu=\mathcal{C}_{w} \mu=C_{\Sigma^{\prime}}^{+}\left(\mu w_{x}\right)+C_{\Sigma^{\prime}}^-\left(\mu w_+\right).
\end{equation}
In Type I,  for $k \in \Sigma$,
\begin{align}
\mu_{11}(x, k)=&1+\frac{1}{2 \pi i}\int_{\mathbb{R}} \frac{\mu_{12}(x, s) r(s) e^{2 i s^2 x}}{s-k+i0} ds+\frac{1}{2 \pi i}\int_{i\mathbb{R}} \frac{\mu_{12}(x, s) r(s) e^{2 is^2 x}}{s-k+0} ds
+\sum_{j=1}^{2N}\frac{\mu_{12}\left(x, k_j\right) c_j e^{2 i x k_j^{2}}}{k-k_j},\nonumber\\
\mu_{12}(x,k)=& \frac{1}{2 \pi i}\int_{\mathbb{R}} \frac{\mu_{11}(x, s) \bar{r}(s) e^{2 i s^2 x}}{s-k+i0} ds+\frac{1}{2 \pi i}\int_{i\mathbb{R}} \frac{\mu_{11}(x, s) \bar{r}(s) e^{2 i s^2 x}}{s-k+0}ds
-\sum_{i=1}^{2N}\frac{\mu_{11}\left(x, \overline{k_j}\right) \bar{c}_je^{-2 i x \bar{k}_j^{2}}}{k-\bar{k}_j},\nonumber\\
\mu_{21}(x, k)=&\frac{1}{2 \pi i}\int_{\mathbb{R}} \frac{\mu_{22}(x, s) r(s) e^{2 i s^2 x}}{s-k+i0}ds+\frac{1}{2 \pi i}\int_{i\mathbb{R}} \frac{\mu_{22}(x, s) r(s) e^{2 is^2 x}}{s-k+0} ds
+\sum_{i=1}^{2N}\frac{\mu_{22}\left(x, k_{i}\right) c_j e^{2 i x k_{l}^{2}}}{k-k_j},\nonumber\\
\mu_{22}(x,k)=&1+\frac{1}{2 \pi i}\int_{\mathbb{R}} \frac{\mu_{21}(x, s) \bar{r}(s) e^{2 i s^2 x}}{s-k+i0} ds+\frac{1}{2 \pi i}\int_{i\mathbb{R}} \frac{\mu_{21}(x, s) \bar{r}(s) e^{2 i s^2 x}}{s-k+0} ds
-\sum_{i=1}^{2N}\frac{\mu_{21}\left(x, \bar{k}_j\right) \bar{c}_j e^{2 i x \bar{k}_j^{2}}}{k-\bar{k}_j}.\nonumber
\end{align}
and in order to close the system, we have
\begin{align}
\mu_{11}\left(x,  \bar{k}_j\right)=&1+\frac{1}{2 \pi i}\int_{\Sigma} \frac{\mu_{12}(x, s)r(s) e^{2 i s^{2} x}}{s -\bar{k}_j}d s +\sum_{i=1}^{2N}\frac{\mu_{12}\left(x, k_i\right) c_i e^{2 i k_j^{2} x}}{ \bar{k}_j-k_i},\nonumber\\
\mu_{12}\left(x, k_j\right)=&\frac{1}{2 \pi i}\int_{\Sigma} \frac{\mu_{11}(x, s)\bar{r}(s) e^{-2 i s^{2} x}}{s -k_j}  {d s}  -\sum_{i=1}^{2N}\frac{\mu_{11}\left(x, \bar{k}_j\right) \bar{c}_i e^{-2 i \bar{k}_i^{2} x}}{ k_j-\bar{k}_i},\nonumber\\
\mu_{21}\left(x,  \bar{k}_j\right)=&\frac{1}{2 \pi i}\int_{\Sigma} \frac{\mu_{22}(x, s)r(s) e^{2 i s^{2} x}}{s \mp \bar{k}_j}  {d s}  +\sum_{i=1}^{2N}\frac{\mu_{22}\left(x, k_i\right) c_i e^{2 i k_i^{2} x}}{ \bar{k}_j-k_i},\nonumber\\
\mu_{22}\left(x, k_j\right)=&1+ \frac{1}{2 \pi i}\int_{\Sigma} \frac{\mu_{21}(x, s)\bar{r}(s) e^{-2 i s^{2} x}}{s-k_j}  {d s} -\sum_{i=1}^{2N}\frac{\mu_{21}\left(x, \bar{k}_i\right) \bar{c}_j e^{-2i \bar{k}_i^{2}x}}{ k_j-\bar{k}_i}.\nonumber
\end{align}
Finally, we obtain
\begin{equation}
u_x(x)=-\frac{1}{\pi} \int_{\Sigma} e^{-2i k^{2}x} \bar{r}(k) \mu_{11}(x, k) d k-\sum_{j=1}^{2N} 2 i \mu_{11}\left(x, \bar{k}_j\right) \bar{c}_j e^{-2 i \bar{k}_j^{2} x}.
\end{equation}

\subsection{Application of Zhou's vanishing lemma}
\hspace*{\parindent}
It is found that the matrix $J(k)$ in (\ref{herm1}) is Hermitian  for $k\in \mathbb{R}$. In this case, we
can use Zhou's Vanishing Lemma from \cite{zx} to obtain a unique solution of the {RH Problem 4.2}.  Though the  matrices (\ref{herm2})-(\ref{herm4}) are
 not Hermitian for $k\in i\mathbb{R}$ and $k\in \gamma_j,\  \overline{\gamma}_j$,    their jump contours are  oriented to preserve
 Schwartz reflection symmetry in the real axis $\mathbb{R}$, and the  matrix-valued function $F(k)=M(k) M(\bar{k})^{H}$
 is still analytic for $k \in \mathbb{C} \backslash \mathbb{R}$. This
 indeed  fits   into  the framework of Zhou's Vanishing Lemma  and
    obtain the same conclusion   in the  following Proposition.

\begin{proposition}
\label{ppl}
The solution to the {RH Problem 4.2} with Type II normalization is identically zero.
\end{proposition}
\begin{proof}
We first  recall that the symmetry reduction condition for the entries of the transition
matrix is given as follows: for $k\in i\mathbb{R}$, $r(\bar{k})=r(-k)=-r(k).$
So we conclude from (\ref{herm2}) that
\begin{equation}
J(k)=J(\bar{k})^{H}, \ \ k\in i\mathbb{R},\label{KH}
\end{equation}
where ${}^H$ denotes the complex conjugation and transpose of a given matrix.
 From(\ref{herm3}) and (\ref{herm4}), it is easy to check  that the same equality
 holds on $\left\{ \gamma_j \cup  \bar{\gamma}_j\right\}_{j=1}^{2N}$. So we conclude that (\ref{KH}) holds on $\Sigma^{\prime} \backslash \mathbb{R}$.

Now we formulate the matrix-valued function
\begin{align}
F(k)=M(k) M(\bar{k})^{H} \label{wewe}
\end{align}
and we want to show that
\begin{equation}
\int_{\mathbb{R}} F_{\pm}(k) dk=0.\label{f1}
\end{equation}
It is clear that $F(k)$ is analytic in $\mathbb{C} \backslash \Sigma^{\prime}$ by the Schwartz reflection principle.  For  $k \in \mathbb{C} \backslash \mathbb{R}$, we then have
$$
\begin{aligned}
&F_{+}(k)=M_{+}(k) M_{-}(\bar{k})^{H}
 =M_-(k) J(k)\left(J(\bar{k})^{-1}\right)^{H} M_{+}(\bar{k})^{H} \\
&=M_-(k) M_{+}(\bar{k})^{H} =F_{-}(k).
\end{aligned}
$$
By Morera's theorem $F(k)$ is analytic for $k\in\mathbb{C} \backslash \mathbb{R}$. Since $M_{\pm}(x, \cdot)\in L^{2} (\mathbb{R})$,   $F(k)$ is integrable.
By the Beal-Coifman theorem, we can write
$$
\begin{aligned}
M(x,k)&=\frac{1}{2 \pi i} \int_{\Sigma^{\prime}} \frac{\mu(x, s)\left(w_+(s)+w_{-}(s)\right)}{s-k}d s\\
&= \frac{1}{2 k\pi i} \int_{\Sigma^{\prime}} \frac{s}{s-k} \mu(x, s)\left(w_+(s)+w_-(s)\right) d s
 -\frac{1}{2 k\pi i} \int_{\Sigma^{\prime}} \mu(x, s)\left(w_+(s)+w_-(s)\right) d s,
\end{aligned}
$$
which implies that
$$
M(x, k) \sim \mathcal{O}\left(k^{-1}\right), \quad k \in \mathbb{C} \backslash \mathbb{R},
$$
and   $F(k) \sim \mathcal{O}\left(k^{-2}\right),\  k \in \mathbb{C} \backslash \mathbb{R}$ by the definition (\ref{wewe}).
Then the equality (\ref{f1}) follows from  Jordon theorem and Cauchy integral theorem.

For $k \in \mathbb{R}$, we have that
\begin{equation}
F_{+}(k)=M_{+}(k) M_-(k)^{H}=M_-(k)J(k) M_{-}(k)^{H},
\end{equation}
and
\begin{equation}
F_-(k)=M_-(k) M_{+}(k)^{H}=M_{-}(k) J(k)^{H} M_-(k)^{H},
\end{equation}
and from (\ref{f1}) we get
\begin{equation}
\int_{\mathbb{R}}  M_{-}(k)\left(J(k)+J(k)^{H}\right) M_{-}(k)^{H}=0, \label{fed}
\end{equation}
where for $k \in \mathbb{R}$,
\begin{equation}
J(k)+J(k)^{H}=2\begin{pmatrix}
1+|r(k)|^{2} & \overline{r(k)}e^{-2ik^{2}x} \\
r(k)e^{2ik^{2}x} & 1
\end{pmatrix}\nonumber
\end{equation}
is positive definite since it is Hermitian and has positive eigenvalues. Thus from (\ref{fed})  we conclude that $M_-(k)=0$ on $\mathbb{R}$.
Further we also  have
 $$ M_{+}(k)= M_-(k)J(k)=0, \ k \in \mathbb{R}.$$
From Morera's theorem we conclude that $M(k)$ is analytic in a neighborhood of every point on $\mathbb{R}$. Since $M(k)=0$ for $k \in \mathbb{R}$, the analytic continuation gives us that $M(k)=0$ holds all the way up to the first complex part of $\Sigma$. Applying the jump condition on this part shows that $M_{\pm}(k)$ vanishes. We can apply the same argument to the remaining parts of $\Sigma^{\prime}$ and conclude that $M(k) \equiv 0$ on the entire complex plane.
\end{proof}
The following proposition is a standard result from  the Fredholm alternative theorem.
\begin{proposition}   Let the initial date  $u(x,0) \in H^{3,3}(\mathbb{R})$,
then      the  RH Problem 4.2   with Type I normalization   has a  unique solution $M(k)$.
\end{proposition}

\subsection{Reconstruction of the potential}
\hspace*{\parindent}
In this part we show that the solution to the RH Problem 4.2  with Type I normalization solves the spatial spectral problem
  (\ref{L3}) and obtain explicit formulas for
the potentials.

\begin{lemma}
\label{l10}
Let
\begin{equation}
f(x, k)=\mu(x,k)\left(w_+(k)+w_-(k)\right),\label{fxk}
\end{equation}
then we have
\begin{equation}
\int_{\Sigma^{\prime}} s^{n} f(x, s) d s=\int_{\Sigma} s^{n} f(x, s) d s+\int_{\Sigma^{\prime} \backslash \Sigma} s^{n} f(x, s) d s
\end{equation}
is diagonal when $n$ is add and off-diagonal when $n$ is even.
\end{lemma}

\begin{proposition}\label{pro7}
The functions $M_{\pm} (k) $  obey the spatial spectral problem (\ref{L3}),
 where $M(k)$ and $U_x$ are constructed from the solution $\mu(k)$ of  the Beals-Coifman integral equation (\ref{bcm}) as follows
\begin{align}
&M(x, k)=I+\frac{1}{2 \pi i}\int_{\Sigma^{\prime}} \frac{\mu(x, s)\left(w_+(s)+w_{-}(s)\right)}{s-k}d s,\\
&U_x(x)=-\frac{1}{\pi}  \left[\sigma_3,\left(\int_{\Sigma'} \mu_{11}(x,s)\bar{r}(s)e^{-2is^2x} ds\right)\right].\label{uxd}
\end{align}
\end{proposition}
\begin{proof}
For $k \in \Sigma'$,  differentiating the    jump relation (\ref{csm})  with respect to $x$
leads to
\begin{equation}
\begin{aligned}
 M_{+,x}  &=M_{-,x} J+M_{-}\left(-i k^{2} [\sigma_3, J]\right) \\
&=M_{-,x}  J+M_{-}\left(-i k^{2}  [\sigma_3,M_{-}^{-1} M_{+}]\right) \\
&=M_{-,x}  J+i k^{2} [\sigma_3, M_-] J-i k^{2} [\sigma_3, M_+].
\end{aligned}
\end{equation}
We conclude that
\begin{equation}
M_{+,x} (x,k)+i k^{2}  [\sigma_3, M_+]=\left(M_{-,x} (x,k)+i k^{2}  [\sigma_3, M_-]\right) J.
\end{equation}
Using the fact
\begin{equation}
M_{\pm}(x,k)-I=C^{\pm}\left[\mu(x, \cdot)\left(w_-(\cdot)+w_+(\cdot)\right)\right],\label{m1c}
\end{equation}
 we conclude that
\begin{equation}
i k^{2}  [\sigma_3, M_\pm](x,k)=i \left[\sigma_3,C^{\pm}\left((\cdot)^{2} f(x, \cdot)\right)-\frac{k}{2 \pi i}
 \int_{\Sigma^{\prime}} f(x,s) d s-\frac{1}{2 \pi i} \int_{\Sigma^{\prime}} s f(x, s) d s\right],\nonumber
\end{equation}
where $f(x,k)$ is given by (\ref{fxk}). It follows from Lemma \ref{l10} that the matrix-valued integral
$
\int_{\Sigma} k f(x,k) d k
$
is a diagonal matrix. Hence,
\begin{equation}
\left[\sigma_3,\int_{\Sigma^{\prime}} sf(x,s) d s\right]=0.
\end{equation}
Defining $U_x$ by (\ref{uxd}), we have
\begin{equation}
ik^{2} [\sigma_3,\left(M_{\pm}\right)(x, k) ]=i \left[\sigma_3,C^{\pm}\left((\cdot)^{2} f(x, \cdot)\right)\right]+kU_x(x).\label{ik2}
\end{equation}
Note that Cauchy projection is bounded on $L^{2}$ and $\mu_{11}(k)-1 \in L^{2}$ for each fixed $x$. We can show that the first term defines an $L^{2}$ function of $k$ for each $x$. Next, we observe that, by (\ref{m1c}),
\begin{equation}
k U_x(x)\left(M_{\pm}-I\right)=U_x(x) C^{\pm}[(\cdot) f(x, \cdot)]-U_x(x)R(x),\label{ucr}
\end{equation}
where $R(x)$ is given by
\begin{equation}
R(x)=\frac{1}{2\pi i} \int_{\Sigma^{\prime}}\mu(x, k)\left(w_+(k)+w_-(k)\right) dk
\end{equation}
and the first right-hand term of (\ref{ucr}) is an $L^{2}$ function of $k$ for each $x$.

Now define
$$
W_{\pm}=M_{\pm,x} +i k^{2} [\sigma_3,M_{\pm}]-kU_x(x) M_{\pm}(x)-U_x(x)R(x) M_{\pm}(x).
$$
By (\ref{m1c}),(\ref{ik2}),(\ref{ucr}), and the identity
\begin{equation}
\begin{aligned}
-kU_x(x) M_{\pm}(x)-U_x(x)R(x) M_{\pm}(x)&=-kU_x(x)-k U_x(x)\left(M_{\pm}-I\right) \\
&-U_x(x)R(x)-U_x(x)R(x)\left(M_{\pm}-I\right),
\end{aligned}
\end{equation}
it now follows that $\left(W_{+}, W_-\right) \in \partial C_{\Sigma}\left(L^{2}\right)$ for each fixed $x$. More explicitly,
\begin{equation}
\begin{aligned}
W_{\pm}=&   C^{\pm}_x[f(x, \cdot)]+i [\sigma_x,C^{\pm}\left((\cdot)^{2} f(x, \cdot)\right)]-U_x(x)\left[C^{\pm}((\cdot) f(x,\cdot))\right] \\
&-U_x(x) R(x)\left[C^{\pm}((\cdot) f(x, \cdot))\right].
\end{aligned}
\end{equation}
Also,
\begin{equation}
W_+=W_- J.
\end{equation}
and we can check that $W(x,k)$ has the same residue condition in the complex plane as the RH Problem 4.2.
 It follows from Proposition  \ref{ppl} that $W_+=W_-=0$.
\end{proof}

\section{Deformation of the RH Problem}
\label{sec:section4}

\subsection{Conjugation}
\hspace*{\parindent}
In the jump matrix (\ref{Vk1}),  the oscillatory terms are  $e^{\pm it\theta(k)}$, where
\begin{equation}
\theta(k)=k^{2} \frac{x}{t}+\eta^{2}(k).\label{thetak}
\end{equation}
It will be found that the long-time asymptotic of RH Problem 3.1 is affected by the growth and the decay of the exponential function $e^{2it\theta(k)}$ appearing in both the jump relation and the residue conditions. In this section, we introduce a new transform $M(k)\rightarrow M^{(1)}(k)$, from which we make $M^{(1)}(k)$ well-behaved as
$t\rightarrow\infty$ along any characteristic line.
Without loss of generality, we focus on the case  $t>0$ below, while the case  $t<0$ can be handled with the same approach (Appendix \ref{sec:section11}).

Let $\xi=\frac{x}{t}+\alpha>0$,  from (\ref{thetak}),  we find  four phase   points of the
function $\theta(k)$
 $$z_n=  k_0e^{\frac{i(n-1)\pi}{2}},  \ \ n=1,2,3,4,  $$
where $|z_n|=k_0=(\frac{\alpha\beta^{2}}{4\xi})^{1/4}.$   We further rewrite  (\ref{thetak})   in the form
\begin{equation}
\theta(k)=\frac{\alpha\beta^2}{4}\left(\frac{k^2}{k_0^4}+\frac{1}{k^2}\right)-\alpha\beta,
\end{equation}
which leads to
\begin{equation}
\mathrm{Re}(i\theta(k))=-\frac{1}{4}\alpha\beta^2\mathrm{Im}k^2\left(\frac{1}{k_0^4}-\frac{1}{|k|^4}\right).\label{Re}
\end{equation}
The  sign   signature of $ \mathrm{Re}(i\theta(k))$  determines the    decay
domains of the oscillatory term $e^{it\theta(k)}$ with respect to $t\rightarrow \infty$, as Figure \ref{cstx} shows.

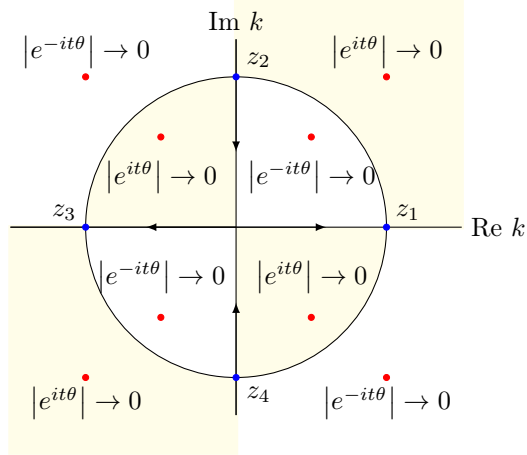
\begin{figure}[H]
\centering
\begin{tikzpicture}[node distance=2cm]
 \draw [fill=pink,ultra thick,color=yellow!10] rectangle (3,3);
 \draw [fill=blue,ultra thick,color=yellow!10] (0,0) rectangle (-3,-3);
\filldraw[white](0,0)-- (2,0) arc (0:90:2);
\filldraw[white](0,0)-- (-2,0) arc (180:270:2);
\filldraw[color=yellow!10](0,0)-- (0,-2) arc (270:360:2);
\filldraw[color=yellow!10](0,0)-- (0,2) arc (90:180:2);
\draw [ ](-3,0)--(3,0)  node[right, scale=1] {Re $k$};
\draw [-latex](-3,0)--(1.2,0);
\draw [-latex](0,0)--(-1.2,0);
\draw [ ](0,-2.5)--(0,2.5)  node[above, scale=1] {Im $k$};
\draw [-latex](0, 2.5)--(0,1);
\draw [-latex](0, -2.5)--(0,-1);
\draw(0,0)circle(2cm);
\draw[fill,blue] (0,2) circle [radius=0.04];
\draw[fill,blue] (0,-2) circle [radius=0.04];
\draw[fill,blue] (2,0) circle [radius=0.04];
\draw[fill,blue] (-2,0) circle [radius=0.04];
\node  [right]  at (2,0.2) {$z_1$};
\node  [above]  at (0.3,2) {$z_2$};
\node  [left]  at (-2,0.2) {$z_3$};
\node  [below]  at (0.3,-2) {$z_4$};
\node  [below]  at (1,1) {$\left|e^{ -i t \theta}\right|\rightarrow 0$};
\node  [above]  at (1,-1) {$\left|e^{ i t \theta}\right|\rightarrow 0$};
\node  [below]  at (-1,1) {$\left|e^{i  t\theta}\right|\rightarrow 0$};
\node  [above]  at (-1,-1) {$\left|e^{ -it  \theta}\right|\rightarrow 0$};
\node  [above]  at (2,2) {$\left|e^{i  t\theta}\right|\rightarrow 0$};
\node  [below]  at (2,-2) {$\left|e^{ -i t\theta}\right|\rightarrow 0$};
\node  [above]  at (-2,2) {$\left|e^{ -i t \theta}\right|\rightarrow 0$};
\node  [below]  at (-2,-2) {$\left|e^{ it \theta}\right|\rightarrow 0$};
\draw[fill,red] (1,1.2) circle [radius=0.04];
\draw[fill,red] (2,2) circle [radius=0.04];
\draw[fill,red] (2,-2) circle [radius=0.04];
\draw[fill,red] (-2,-2) circle [radius=0.04];
\draw[fill,red] (-2,2) circle [radius=0.04];
\draw[fill,red] (1,-1.2) circle [radius=0.04];
\draw[fill,red] (-1,-1.2) circle [radius=0.04];
\draw[fill,red] (-1,1.2) circle [radius=0.04];
\end {tikzpicture}
\caption{The decay domains of exponential oscillatory terms, where $|e^{ it\theta(k)} |\rightarrow 0$
when $t\rightarrow\infty$ in yellow-shaded regions, while  in blank regions, $|e^{-it\theta(k)}|\rightarrow 0$ when $t\rightarrow\infty$.}
\label{cstx}
\end{figure}

To facilitate the subsequent analysis in the remainder of this paper, we will uniformly deal with the four stationary phase points $z_n,   n=1,2,3,4$.
For this purpose,  we introduce some notations.  We first divide $\Sigma$   into the following eight intervals near the four stationary phase  points $z_n,   n=1,2,3,4$:
\begin{align}
&I_{n1}=(0,z_n),\ \ I_{n2}=(z_n,+\infty),\ n=1,2,\nonumber\\
&I_{n1}=(z_n,0),\ \ I_{n2}=(-\infty,z_n),\ n=3,4.\nonumber\\
&I_{+}=\cup_{n=1}^4I_{n2}, \ \ I_{-}=\cup_{n=1}^4I_{n1}, \nonumber
\end{align}
which is illustrated in Figure \ref{division31}.

The partition $\Delta_{k_0}^{\pm}$ for $k_0 \in \mathbb{R}_+$ is defined as follows:
$$
\Delta_{k_0}^+=\left\{j \in\{1, \ldots,2N\}: | k_j |>k_0\right\},
$$
$$
\Delta_{k_0}^-=\left\{j \in\{1, \ldots,2N\}:  | k_j |<k_0\right\},
$$
which splits the residue coefficients $c_j e^{2it\theta(k_j)}$ into two sets.

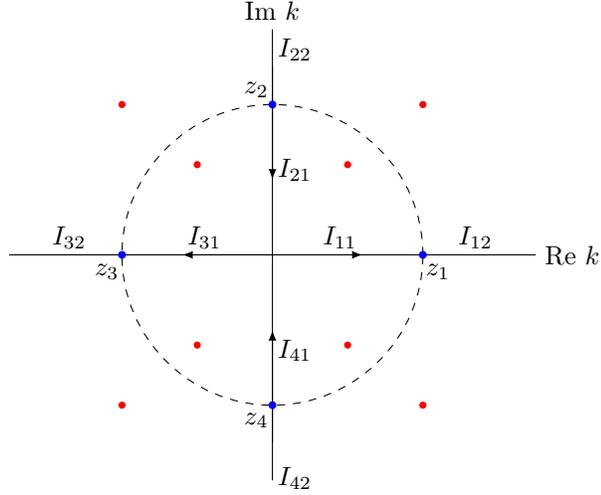
\begin{figure}[H]
\begin{center}
\begin{tikzpicture}[node distance=2cm]
\draw[dashed](0,0)circle(2cm);
\draw [ ](-3.5,0)--(3.5,0)  node[right, scale=1] {Re $k$};
\draw [-latex](-3,0)--(1.2,0);
\draw [-latex](0,0)--(-1.2,0);
\draw [ ](0,-3)--(0,3)  node[above, scale=1] {Im $k$};
\draw [-latex](0, 2.5)--(0,1);
\draw [-latex](0, -2.5)--(0,-1);
%\draw [color=green!90 ] (-2.1,0)--(2.1,0);
%\draw [color=green!90 ] (0,-2.1)--(0,2.1);
\node  [right]  at (-0.5,2.2) {$z_2$};
\node  [right]  at (-0.5,-2.2) {$z_4$};
\node  [below]  at (2.2,0) {$z_1$};
\node  [below]  at (-2.2,0) {$z_3$};
\draw[fill,blue] (2,0) circle [radius=0.045];
\draw[fill,blue] (-2,0) circle [radius=0.045];
\draw[fill,blue] (0,2) circle [radius=0.045];
\draw[fill,blue] (0,-2) circle [radius=0.045];
\node  [below]  at (0.3,-1) {$I_{41}$};
\node  [below]  at (0.3,1.4) {$I_{21}$};
\node  [below]  at (0.3,3) {$I_{22}$};
\node  [below]  at (0.3,-2.7) {$I_{42}$};
\node  [below]  at (-0.9,0.5) {$I_{31}$};
\node  [below]  at (-2.7,0.5) {$I_{32}$};
\node  [below]  at (2.7,0.5) {$I_{12}$};
\node  [below]  at (0.9,0.5) {$I_{11}$};
%\draw[fill] (2,0) circle [radius=0.04];
%\draw[fill] (-2,0) circle [radius=0.04];
%\draw[fill] (0,2) circle [radius=0.04];
%\draw[fill] (0,-2) circle [radius=0.04];
\draw[fill,red] (1,1.2) circle [radius=0.04];
\draw[fill,red] (2,2) circle [radius=0.04];
\draw[fill,red] (2,-2) circle [radius=0.04];
\draw[fill,red] (-2,-2) circle [radius=0.04];
\draw[fill,red] (-2,2) circle [radius=0.04];
\draw[fill,red] (1,-1.2) circle [radius=0.04];
\draw[fill,red] (-1,-1.2) circle [radius=0.04];
\draw[fill,red] (-1,1.2) circle [radius=0.04];
\end {tikzpicture}
\end{center}
\caption{The division of the jump contour  $\Sigma =\mathbb{R}\cup i\mathbb{R}$.}
\label{division31}
\end{figure}

We define
\begin{align}
&T(k) =\prod_{j \in \Delta_{k_0}^-} \left(\frac{k-\overline{k}_j}{k-k_j}\right) \delta(k),\label{Tk1}
\end{align}
where   $\delta(k)$   is  given by
\begin{align}
&\delta(k)=  \prod_{n=1,3} \left( \frac{ k-z_n }{k} \right)^{i\nu(z_n)}e^{\chi_n(k)}  \prod_{n=2,4}   \left(\frac{k}{k-z_n } \right)^{i  {\nu}(z_n)}e^{ {\chi}_n(k)},\\
&\nu(z_n)=-\frac{1}{2\pi}\ln[1+ r(z_n) \overline{r(\bar z_n)} ],\ \ n=1, 2,3,4,\nonumber\\
&\chi_{n}(k)=\frac{1}{2\pi i}\int_0^{  z_n}\ln\left(\frac{1+ r(k') \overline{r(\bar k')}}{1+ r(z_n) \overline{r(\bar z_n)}}\right)\frac{dk'}{k'-k},\ \ n=1, 2, 3, 4.\nonumber
\end{align}
Among all the formulas above, we choose the principal branch of   logarithm functions.

\begin{proposition}
The function defined by (\ref{Tk1}) has following properties:

(a) $T(k)$ is meromorphic in $\mathbb{C} \backslash I_-$, i.e., for each $j \in \Delta_{k_0}^-$, $T(k)$ has a simple pole at $k_j$ and $\overline{k}_j$ which belong to $\mathcal{K}\cup \mathcal{\overline{K}}$,

(b) For $k\in \mathbb{C} \backslash I_-$, $\overline{T(\overline{k})}=1/T(k)$,

(c) For $k \in I_- $, the boundary values $T_{\pm}(k)$ satisfy
\begin{equation}
T_+(k)/T_-(k)=1+ r(k)\overline{ r(\bar k)},
\end{equation}

(d) $|k|\rightarrow \infty$ with $|\arg(k)|\leq c<\pi$
\begin{equation}
T(k)=1+\frac{i}{k}\left[2 \sum_{j \in \Delta_{k_0}^-} \operatorname{Im}(k_j)-\int_{I_-} k(s) d s\right]+\mathcal{O}\left(k^{-2}\right),\label{tk}
\end{equation}

(e) $k\rightarrow   z_n, \ n=1,2,3,4$ along any ray $  z_n+e^{i\phi}\mathbb{R}_+$ with $|\phi|\leq c<\pi$
\begin{equation}
|T(k)-T_0(z_n)(k - z_n)^{i\nu(z_n)}| \leq c|k- z_n|^{1/2},\label{tk0}
\end{equation}
where
\begin{align}
&T_0(z_n) =\prod_{j \in \Delta_{k_0}^-}\left(\frac{z_n-\overline{k}_j}{z_n-k_j}\right) e^{i\beta(k,z_n)},\\
&\beta (k,z_n )=-\nu(z_n) \log \left(k -z_n+e^{\frac{i(n-1)\pi}{2}}\right)+\int_{I_-} \frac{\nu(s)-\chi_{n}(s) \nu(z_n)}{s-k} d s,\label{t0zn}
\end{align}
with $n=1,2,3,4$.
\label{p}
\end{proposition}

Using our partial transmission coefficient $T(k)$ defined above, we make
 a transformation
\begin{equation}
M^{(1)}(k)=M(k) T(k)^{-\sigma_{3}},\label{mktk}
\end{equation}
which  satisfies the following RH problem:\\
\textbf{RH Problem 5.1.} Find a  matrix  function $ M^{(1)}(k)=M^{(1)}(k;x,t)$  with the following properties:

(a)\emph{Analyticity}: $ M^{(1)}(k)$ is meromorphic  in $ \mathbb{C}\setminus \Sigma $;

(b)\emph{Symmetry}: $\overline{M^{(1)}(\bar{k})}=\sigma_2 M^{(1)}(k) \sigma_2$;

(c)\emph{Jump condition}: $ M^{(1)}(k)$   satisfies the jump condition
\begin{equation}
M^{(1)}_+ (k) =M^{(1)}_- (k)
V^{(1)}(k),  \label{mup}
\end{equation}
the jump matrix $ V^{(1)}(k) $ is defined  by
\begin{equation}
V^{(1)}=\begin{cases}
\begin{pmatrix}
1&0\\
(\frac{r(k)}{1+r(k)\overline{r(\overline{k})}} T_-^{-2}(k) e^{2it\theta} & 0
\end{pmatrix}
\begin{pmatrix}
1 & \frac{\overline{r(\bar{k})}}{1+r(k)\overline{r(\overline{k})}} T_+^2(k) e^{-2it\theta} \\
0 & 1
\end{pmatrix},  \quad k\in I_-,\\ \\
\begin{pmatrix}
1 & \overline{r(\bar{k})} T^2(k) e^{-2it\theta} \\
0 & 1
\end{pmatrix}
\begin{pmatrix}
1 & 0 \\
r(k) T^{-2}(k) e^{2it\theta} & 1
\end{pmatrix},\quad   k\in I_+;    \label{V11}
\end{cases}
\end{equation}

(d)\emph{Asymptotic conditions}:
\begin{equation}
\begin{cases}
M^{(1)}(k)=I+\mathcal{O}\left(k^{-1}\right), & k \rightarrow \infty,\\
M^{(1)}(k)=e^{-\frac{i}{2} c_- \sigma_{3}}\left(I+kU+\mathcal{O}\left(k^{2}\right)\right) e^{\frac{i}{2} d_0\sigma_3},  & k \rightarrow 0;
\end{cases}
\end{equation}

(e)\emph{Residue conditions}: $M^{(1)}(k)$ has simple poles at each $k_j \in \mathcal{K}\cup\mathcal{\overline{K}}$ at which
\begin{align}
&\underset{k=k_j}{\operatorname{Res}} M^{(1)}=
\underset{k \rightarrow k_j}{\lim } M^{(1)}R_{\pm}(k_j), \qquad \ k \in \Delta_{k_0}^\pm, \label{515}\\
&\underset{k=\overline{k}_j}{\operatorname{Res}} M^{(1)}=
\underset{k \rightarrow \overline{k}_j}{\lim}  M^{(1)}\sigma_2\overline{R_{ \pm} (\overline{k}_j)} \sigma_2, \ k \in \Delta_{k_0}^\pm,
\end{align}
where
\begin{equation}
R_{ -}(k_j)=\begin{pmatrix}
0&c_j^{-1}(1/T)^{\prime}(k_j)^{-2} e^{-2it\theta(k_j)}\\
0&0
\end{pmatrix},\quad
R_{ +}(k_j)=\begin{pmatrix}
0 & 0 \\
c_j T(k_j)^{-2} e^{2it\theta(k_j)} & 0
\end{pmatrix}.\label{rkj}
\end{equation}
\begin{proof}
With Proposition \ref{p} and the properties of $M$, $M^{(1)}(k)$ is unimodular, analytic in $\mathbb{C} \backslash (\Sigma \cup \mathcal{K}\cup \mathcal{\overline{K}})$, and approaches identity as $k \rightarrow \infty$ by its definition. For the residues, since $T(k)$ is analytic at each $k_j, \overline{k}_j$ with $k \in \Delta_{k_0}^+,$ the residue conditions at these poles are an immediate consequence of (\ref{Tk1}). For $k \in \Delta_{k_0}^-, T(k)$ has zeros at $\overline{k}_j$, and poles at $k_j$ so that $M_1^{(1)}(k)=M_1(k) T(k)^{-1}$ has a removable singularity at $k_j$ but acquires poles at $\overline{k}_j$. For $M_2^{(1)}(k)=M_2(k)T(k)$ the situation
is reversed where it has a pole at $k_j$ and a removable singularity at $\overline{k}_j. $ At $k_j$ we have
$$
\begin{aligned}
M_1^{(1)}(k_j)=& \underset{k=k_j}{\operatorname{Res}}  M_1(k)T(k)^{-1}=\underset{k=k_j}{\operatorname{Res}} M_1(k)\cdot(1/T)^{\prime}(k_j) \\
=&c_j e^{2i\theta(k_j)} M_2(k_j)(1/T)'(k_j), \\
\underset{k=k_j}{\operatorname{Res}}M_2^{(1)}(k)=& \underset{k=k_j}{\operatorname{Res}} M_2(k)T(k)=M_2(k_j)\left[(1/T)'(k_j)\right]^{-1} \\
=&c_j^{-1}\left[(1/T)'(k_j)\right]^{-2} e^{-2it\theta(k_j)} M_1^{(1)}(k_j).
\end{aligned}
$$
from which the first formula in (\ref{515}) clearly follows. The computation of the residue at $\overline{k}_j$ for $k \in \Delta_{k_0}^-$ can be calculated in the same way.
\end{proof}

\subsection{A mixed $\bar{\partial}$-RH problem}
\hspace*{\parindent}
The main purpose of this section is to construct  a new   matrix function $M^{(2)}(k)$  for deforming the contour $\Sigma$ into a contour $\Sigma^{(2)}$  such that: (i) $M^{(2)}(k)$  has no jump on the real and the imaginary axis. For this purpose, we choose the boundary values of $R^{(2)}(k)$ through the factorization of $V^{(1)}(k)$ in (\ref{mup}) where the new jumps on $\Sigma^{(2)}$ match a well-known model RH problem; (ii) We need to control the norm of $R^{(2)}(k)$, so that the $\overline{\partial}$-contribution to the long-time asymptotics of $u(x,t)$ can be ignored; (iii) The residues are unaffected by the transformation.

As Figure \ref{division1}, we define a   contour
$ \Sigma^{(2)}=L \cup L_{0} \cup \bar{L} \cup \bar{L}_{0}$,
where $L=L_1 \cup  {L}_2 \cup L_3 \cup  {L}_4$, and
\begin{align}
&L_j=\{k=z_j+  ue^{\frac{3\pi}{4}i}, u \in(-\infty, \frac{1}{\sqrt{2}}]\}, \ j=1,4; \nonumber\\
&{L}_j=\{k=z_j+ ue^{-\frac{\pi}{4}i}, u \in(-\infty, \frac{1}{\sqrt{2}}]\}, \ j=2,3; \nonumber\\
&L_{0}=\left\{k= u e^{  \frac{\pi}{4} i}, \quad u \in[-\frac{1}{\sqrt{2}}, \frac{1}{\sqrt{2}}]\right\}.\nonumber
\end{align}
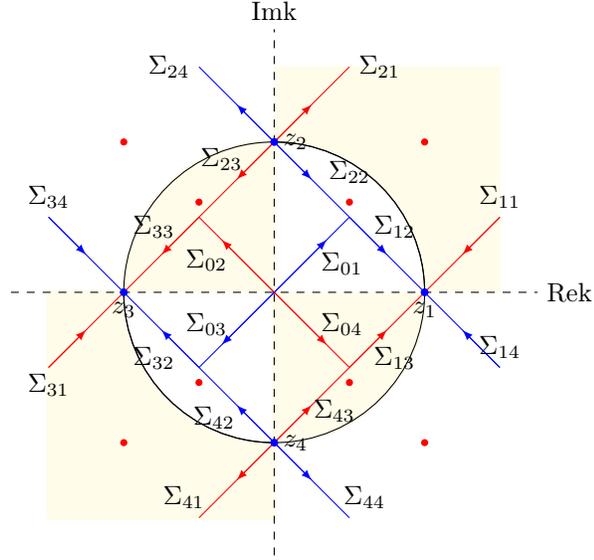
\begin{figure}[H]
\begin{center}
\begin{tikzpicture}[node distance=2cm]
\draw [fill=pink,color=yellow!10] (0,0) rectangle (3,3);
\draw [fill=pink,ultra thick,color=yellow!10] (0,0) rectangle (-3,-3);
\filldraw[white](0,0)-- (2,0) arc (0:90:2);
\filldraw[white](0,0)-- (-2,0) arc (180:270:2);
\draw [fill=white] (0,0) circle [radius=2];
\filldraw[fill=pink,ultra thick,color=yellow!10](0,0)-- (0,-2) arc (270:360:2);
\filldraw[fill=pink,ultra thick,color=yellow!10](0,0)-- (0,2) arc (90:180:2);
\draw(0,0)circle(2cm);
\draw [dashed](-3.5,0)--(3.5,0)  node[right, scale=1] {Rek};
\draw [dashed](0,-3.5)--(0,3.5)  node[above, scale=1] {Imk};
\draw [blue] (-1,3)--(3,-1)  node[right, scale=1] {};
\draw[red] (1,3)--(-3,-1)  node[above, scale=1] {};
\draw [blue] (-3,1)--(1,-3)  node[above, scale=1] {};
\draw[red] (-1,-3)--(3,1)  node[above, scale=1] {};
\draw[red]  (0,0)--(1,-1)  node[right, scale=1] {};
\draw[blue] (0,0)--(-1,-1)  node[above, scale=1] {};
\draw[blue] (0,0)--(1,1)  node[above, scale=1] {};
\draw[red] (0,0)--(-1,1)  node[above, scale=1] {};
\draw [-latex,blue] (0,0) -- (0.7,0.7);
\draw [-latex,red] (0,0) -- (0.7,-0.7);
\draw [-latex,red] (0,0) -- (-0.7,0.7);
\draw [-latex,blue] (0,0) -- (-0.7,-0.7);
%\draw [-latex] (0,0) -- (1,0);
%\draw [-latex] (0,0) -- (-1,0);
%\draw [-latex] (0,2) -- (0,1);
%\draw [-latex] (0,-2) -- (0,-1);
%\draw [-latex] (3,0) -- (2.5,0);
%\draw [-latex] (-3,0) -- (-2.5,0);
%\draw [-latex] (0,2) -- (0,2.5);
%\draw [-latex] (0,-2) -- (0,-2.5);
\draw [-latex,blue] (0,2) -- (0.5,1.5);
\draw [-latex,red] (0,2) -- (-0.5,1.5);
\draw [-latex,blue] (1,1) -- (1.5,0.5);
\draw [-latex,red] (-1,1) -- (-1.5,0.5);
\draw [-latex,red] (0,-2) -- (0.5,-1.5);
\draw [-latex,blue] (0,-2) -- (-0.5,-1.5);
\draw [-latex,red] (1,-1) -- (1.5,-0.5);
\draw [-latex,blue] (-1,-1) -- (-1.5,-0.5);
\draw [-latex,red] (3,1) -- (2.5,0.5);
\draw [-latex,blue] (3,-1) -- (2.5,-0.5);
\draw [-latex,blue] (-3,1) -- (-2.5,0.5);
\draw [-latex,red] (-3,-1) -- (-2.5,-0.5);
\draw [-latex,red] (0,2) -- (0.5,2.5);
\draw [-latex,blue] (0,2) -- (-0.5,2.5);
\draw [-latex,blue] (0,-2) -- (0.5,-2.5);
\draw [-latex,red] (0,-2) -- (-0.5,-2.5);
\node  [right]  at (0,2) {$z_2$};
\node  [right]  at (0,-2) {$z_4$};
\node  [below]  at (2,0) {$z_1$};
\node  [below]  at (-2,0) {$z_3$};
\node  [above]  at (3,1) {$\Sigma_{11} $};
\node  [above]  at (3,-1) {$\Sigma_{14}$};
\node  [below]  at (0.8,-1.3) {$\Sigma_{43}$};
\node  [above]  at (-0.7,1.5) {$\Sigma_{23}$};
\node  [right]  at (0.6,1.6) {$\Sigma_{22}$};
\node  [right]  at (1,3) {$\Sigma_{21}$};
\node  [left]  at (-1,3) {$\Sigma_{24}$};
\node  [below]  at (-1.6,-0.6) {$\Sigma_{32}$};
\node  [above]  at (-3,1) {$\Sigma_{34}$};
\node  [above]  at (-3,-1.5) {$\Sigma_{31}$};
\node  [above]  at (1.2,-3) {$\Sigma_{44}$};
\node  [above]  at (-1.2,-3) {$\Sigma_{41}$};
\node  [above]  at (-1.6,0.6) {$\Sigma_{33}$};
\node  [below]  at (-0.8,-1.4) {$\Sigma_{42}$};
\node  [below]  at (1.6,-0.6) {$\Sigma_{13}$};
\node  [above]  at (1.6,0.6) {$\Sigma_{12}$};
\node  [below]  at (-0.9,0.7) {$\Sigma_{02}$};
\node  [above]  at (-0.9,-0.7) {$\Sigma_{03}$};
\node  [above]  at (0.9,-0.7) {$\Sigma_{04}$};
\node  [below]  at (0.9,0.65) {$\Sigma_{01}$};
\draw[fill,blue] (2,0) circle [radius=0.045];
\draw[fill,blue] (-2,0) circle [radius=0.045];
\draw[fill,blue] (0,2) circle [radius=0.045];
\draw[fill,blue] (0,-2) circle [radius=0.045];
%\draw[fill] (2,0) circle [radius=0.04];
%\draw[fill] (-2,0) circle [radius=0.04];
%\draw[fill] (0,2) circle [radius=0.04];
%\draw[fill] (0,-2) circle [radius=0.04];
\draw[fill,red] (1,1.2) circle [radius=0.04];
\draw[fill,red] (2,2) circle [radius=0.04];
\draw[fill,red] (2,-2) circle [radius=0.04];
\draw[fill,red] (-2,-2) circle [radius=0.04];
\draw[fill,red] (-2,2) circle [radius=0.04];
\draw[fill,red] (1,-1.2) circle [radius=0.04];
\draw[fill,red] (-1,-1.2) circle [radius=0.04];
\draw[fill,red] (-1,1.2) circle [radius=0.04];
\end {tikzpicture}
\end{center}
\caption{\small The jump contour  $\Sigma^{(2)}$ for the  $M^{(2)}$, which is consist of blue lines and red lines.
The red lines  decay along yellow domains; The blue lines decay along blank  domains }
\label{division1}
\end{figure}

Furthermore, the contour $\Sigma^{(2)}$ contains $\Sigma_{ij}$ with $i=0,1,2,3,4$ and $j=1,2,3,4$. $ \Sigma$  and $\Sigma^{(2)}$ further divides the complex plane $\mathbb{C}$ into $24$ sectors denoted by $D_{nj}$, where $n = 1,2,3,4,5,6$ and $j= 1,2,3,4$, which is plotted in Figure \ref{division2}.

\begin{figure}[H]
\centering
\begin{tikzpicture}[node distance=2cm]
\draw(0,0)circle(2cm);
%\draw [blue!10, fill=yellow!10] (2,0)--(3,1)--(1,3)--(0,2);
%\draw [blue!10, fill=yellow!10] (-2,0)--(-3,1)--(-1,3)--(0,2);
%\draw [blue!10, fill=yellow!10] (2,0)--(3,-1)--(1,-3)--(0,-2);
%\draw [blue!10, fill=yellow!10] (-2,0)--(-3,-1)--(-1,-3)--(0,-2);
\draw [dashed](-3.5,0)--(3.5,0)  node[right, scale=1] {Rek};
\draw [dashed](0,-3.5)--(0,3.5)  node[above, scale=1] {Imk};
\draw [blue] (-1,3)--(3,-1)  node[right, scale=1] {};
\draw[red] (1,3)--(-3,-1)  node[above, scale=1] {};
\draw [blue] (-3,1)--(1,-3)  node[above, scale=1] {};
\draw[red] (-1,-3)--(3,1)  node[above, scale=1] {};
\draw[red]  (0,0)--(1,-1)  node[right, scale=1] {};
\draw[blue] (0,0)--(-1,-1)  node[above, scale=1] {};
\draw[blue] (0,0)--(1,1)  node[above, scale=1] {};
\draw[red] (0,0)--(-1,1)  node[above, scale=1] {};
\draw [-latex,blue] (0,0) -- (0.7,0.7);
\draw [-latex,red] (0,0) -- (0.7,-0.7);
\draw [-latex,red] (0,0) -- (-0.7,0.7);
\draw [-latex,blue] (0,0) -- (-0.7,-0.7);
%\draw [-latex] (0,0) -- (1,0);
%\draw [-latex] (0,0) -- (-1,0);
%\draw [-latex] (0,2) -- (0,1);
%\draw [-latex] (0,-2) -- (0,-1);
%\draw [-latex] (3,0) -- (2.5,0);
%\draw [-latex] (-3,0) -- (-2.5,0);
%\draw [-latex] (0,2) -- (0,2.5);
%\draw [-latex] (0,-2) -- (0,-2.5);
\draw [-latex,blue] (0,2) -- (0.5,1.5);
\draw [-latex,red] (0,2) -- (-0.5,1.5);
\draw [-latex,blue] (1,1) -- (1.5,0.5);
\draw [-latex,red] (-1,1) -- (-1.5,0.5);
\draw [-latex,red] (0,-2) -- (0.5,-1.5);
\draw [-latex,blue] (0,-2) -- (-0.5,-1.5);
\draw [-latex,red] (1,-1) -- (1.5,-0.5);
\draw [-latex,blue] (-1,-1) -- (-1.5,-0.5);
\draw [-latex,red] (3,1) -- (2.5,0.5);
\draw [-latex,blue] (3,-1) -- (2.5,-0.5);
\draw [-latex,blue] (-3,1) -- (-2.5,0.5);
\draw [-latex,red] (-3,-1) -- (-2.5,-0.5);
\draw [-latex,red] (0,2) -- (0.5,2.5);
\draw [-latex,blue] (0,2) -- (-0.5,2.5);
\draw [-latex,blue] (0,-2) -- (0.5,-2.5);
\draw [-latex,red] (0,-2) -- (-0.5,-2.5);
\node  [right]  at (0,2) {$ z_2$};
\node  [right]  at (0,-2) {$z_4$};
\node  [below]  at (2,0) {$z_1$};
\node  [below]  at (-2,0) {$z_3$};
\node  [right]  at (2.5,0.3) {$D_{11}$};
\node  [right]  at (2.5,-0.3) {$D_{16}$};
\node  [right]  at (1.6,0.6) {$D_{12}$};
\node  [right]  at (1.6,-0.6) {$D_{15}$};
\node  [right]  at (-1.3, 2) {$D_{25}$};
\node  [left]  at (-1.6,-0.6) {$D_{32}$};
\node  [left]  at (-2.5,0.4) {$D_{36}$};
\node  [left]  at (-1.6,0.6) {$D_{35}$};
\node  [left]  at (-2.5,-0.3) {$D_{31}$};
\node  [above]  at (0.4,2.5) {$D_{21}$};
\node  [above]  at (-0.4,2.5) {$D_{26}$};
\node  [above]  at (1,1.7) {$D_{22}$};
\node  [below]  at (0.4,-2.5) {$D_{46}$};
\node  [below]  at (-0.3,-2.5) {$D_{41}$};
\node  [below]  at (-0.6,-1.7) {$D_{42}$};
\node  [below]  at (1,-1.7) {$D_{45}$};
\node  [right]  at (0.5,0.3) {$D_{13}$};
\node  [left]  at (-0.5,0.3) {$D_{34}$};
\node  [right]  at (0.5,-0.3) {$D_{14}$};
\node  [left]  at (-0.5,-0.3) {$D_{33}$};
\node  [above]  at (0.5,0.7) {$D_{23}$};
\node  [above]  at (-0.5,0.7) {$D_{24}$};
\node  [below]  at (0.5,-0.7) {$D_{44}$};
\node  [below]  at (-0.5,-0.7) {$D_{43}$};
\draw[fill,blue] (2,0) circle [radius=0.045];
\draw[fill,blue] (-2,0) circle [radius=0.045];
\draw[fill,blue] (0,2) circle [radius=0.045];
\draw[fill,blue] (0,-2) circle [radius=0.045];
\draw[fill,red] (1,1.2) circle [radius=0.04];
\draw[fill,red] (2,2) circle [radius=0.04];
\draw[fill,red] (2,-2) circle [radius=0.04];
\draw[fill,red] (-2,-2) circle [radius=0.04];
\draw[fill,red] (-2,2) circle [radius=0.04];
\draw[fill,red] (1,-1.2) circle [radius=0.04];
\draw[fill,red] (-1,-1.2) circle [radius=0.04];
\draw[fill,red] (-1,1.2) circle [radius=0.04];
\end {tikzpicture}
\caption{\small  Jump contour $ \Sigma^{(2)}$ and axis   $ \mathbb{R}\cup i\mathbb{R}$ divide the complex plane $\mathbb{C}$   into 16 domains on which
we define the function $R^{(2)}$.}
\label{division2}
\end{figure}
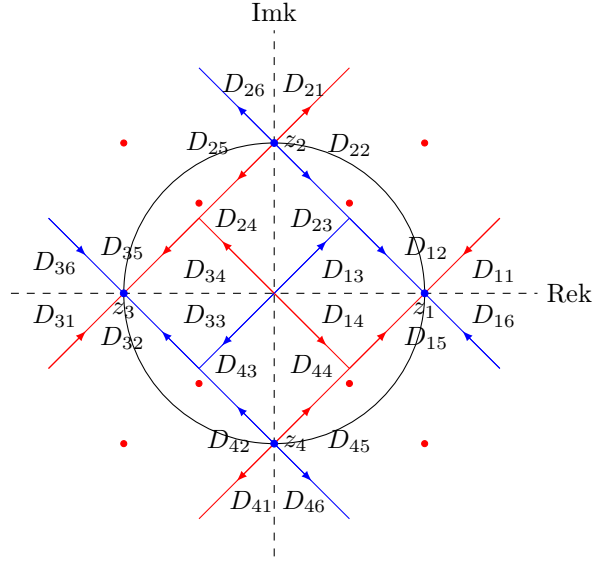

In addition, we define
\begin{equation}
\rho=\frac{1}{2} \underset{\lambda, \mu \in \mathcal{K}\cup\mathcal{\overline{K}}}{\min }|\lambda-\mu|.
\end{equation}
Note that, as poles come in conjugate pairs and  no pole lies on the real axis,
we have $\rho \leq \operatorname{dist}(\mathcal{K}\cup\mathcal{\overline{K}}, \Sigma)$. $\chi_{\mathcal{K}} \in \mathbb{C}_0^{\infty}(\mathbb{C},[0,1])$ is supported near the discrete spectrum such that
\begin{equation}
\chi_{k}(k)=\left\{\begin{array}{ll}
1 & \operatorname{dist}(k, \mathcal{K}\cup\mathcal{\overline{K}})<\rho /3, \\
0 & \operatorname{dist}(k, \mathcal{K}\cup\mathcal{\overline{K}})>2 \rho /3.
\end{array}\right.\label{xk}
\end{equation}

  The region division and the jump lines near each stationary phase point can be unified in the form of Figure \ref{division4}.

  %%%%%%%%%%%%%%%%%%%%%%%%%%% PPPPPP (RHP to Environment) %%%%
\begin{figure}[H]
 \centering
  \begin{tikzpicture}[node distance=2cm]
  \draw[thick ](0,0)--(1.5,1.5)node[above]{$\Sigma_{n1}$};
  \draw[thick ](0,0)--(1.5,-1.5)node[below]{$\Sigma_{n2}$};
  \draw[thick ](0,0)--(-1.5,1.5)node[above]{$\Sigma_{n3}$};
  \draw[thick ](0,0)--(-1.5,-1.5)node[below]{$\Sigma_{n4}$};
  \node[below] at (0, -0.1 )  { $z_n$};
  \draw[ ](1.5,1.5)--(0.5,0.5);
  \draw[thick ](0,0)--(-1.5,1.5){};
  \draw[ ](-1.5,1.5)--(-0.5,0.5);
  \draw[thick ](0,0)--(-1.5,-1.5){};
  \draw[ ](-1.5,-1.5)--(-0.5,-0.5);
  \draw[ ](1.5,-1.5)--(0.5,-0.5);
  \draw[thick ](-2.5,0)--(2.5,0)node[right]{};
  \draw[ ](-2.5,0)--(-1,0);
  \draw[ ](2.5,0)--(1,0);
  \coordinate (A) at (1,0.5);
  \coordinate (B) at (1,-0.5);
  \coordinate (G) at (-1,0.5);
  \coordinate (H) at (-1,-0.5);
  \coordinate (I) at (0,0);
  \coordinate (C) at (-0.8,0);
  \fill (C) circle (0pt) node[above] {\footnotesize$D_{n3}$};
  \coordinate (E) at (0,0.7);
  \fill (E) circle (0pt) node[above] {\footnotesize$D_{n2}$};
  \coordinate (D) at (0.8,0);
  \fill (D) circle (0pt) node[above] {\footnotesize$D_{n1}$};
  \coordinate (F) at (0.8,0);
  \fill (F) circle (0pt) node[below] {\footnotesize$D_{n6}$};
  \coordinate (J) at (0,-0.7);
  \fill (J) circle (0pt) node[below] {\footnotesize$D_{n5}$};
  \coordinate (k) at (-0.8,0);
  \fill (k) circle (0pt) node[below] {\footnotesize$D_{n4}$};
  \end{tikzpicture}
\caption{\small The   contours  and domains  near the stationary phase points $ z_n, n =1,2,3,4.$}
\label{division4}
\end{figure}
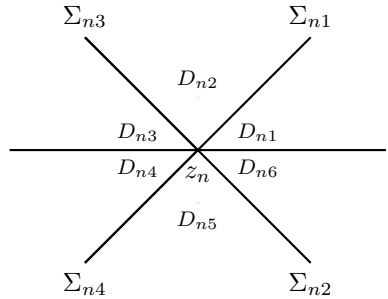

In this way, there are the same contours  and domains at  each  stationary phase  point  $z_n,   n=1,2,3,4$, as illustrated in Figure \ref{division4}. Further, we have the following proposition.

\begin{lemma}
 It is possible to define functions $R_{n,j}$: $\overline{D}_{nj}\rightarrow \mathbb{C}$, $n=1,2,3,4; j=1,3,4,6$  with boundary values satisfying:
\begin{align}
&R_{n,1}(k)=\begin{cases}
r(k)T(k)^{-2}, & k\in I_{n2},\\
f_{n,1}=r(z_n)T_0(z_n)^{-2}(k-z_n)^{-2i\nu(z_n)}[1-\chi_\mathcal{K}(k)], & k\in \Sigma_{n1},\label{R1k}
\end{cases}\\
&R_{n,3}(k)=\begin{cases}
\frac{ \overline{ r (\bar k)} T_+^2(k)}{1+r(k)\overline{ r (\bar k)}}, & k\in I_{n1}, \\[8pt]
f_{n,3}=\frac{\overline{ r(\bar z_n)} }{1+r(z_n)\overline{ r(\bar z_n)}} T_0(z_n)^2 (k-z_n)^{-2i\nu(z_n) }[1-\chi_\mathcal{K}(k)], & k\in \Sigma_{n2},
\end{cases}\\
&R_{n,4}(k)=\begin{cases}
\frac{r(k)T_-^{-2}(k)}{1+r(k)\overline{ r(\bar k)}},& k\in I_{n1},\\
f_{n,4}=\frac{r(z_n)}{1+r(z_n)\overline{ r(\bar z_n)}}T_0(z_n)^{-2}(k-z_n)^{-2i\nu(z_n)}[1-\chi_\mathcal{K}(k)], & k\in \Sigma_{n3},
\end{cases}\\
&R_{n,6}(k)=\begin{cases}
\overline{ r(\bar k)} T(k)^2,& k\in I_{n2},\\
f_{n,6}=\overline{ r(\bar z_n)}T_0(z_n)^2(k-z_n)^{-2i\nu(z_n) }[1-\chi_\mathcal{K}(k)], & k\in \Sigma_{n4},
\end{cases}\\
&R_{n, 2}(k)=R_{n, 5}(k)=1, \quad \text { if } k \in D_{n 2} \cup D_{n 5}.
\end{align}
The functions  $R_{n,j}$ defined above have the following property:
\begin{align}
&\left|R_{n,j}(k)\right| \lesssim \sin ^{2}\left(\arg \left(k-z_n\right)\right)+\langle\operatorname{Re}(k)\rangle^{-1 / 2}, \label{Rjgj}\\
&\left|\bar{\partial} R_{n,j}(k)\right| \lesssim\left|\bar{\partial} \chi_{\mathcal{K}}(k)\right|+\left|r^{\prime}(\operatorname{Re} k)\right|+\left|k-z_n\right|^{-1 / 2},\label{Rjk}
\end{align}
and
\begin{equation}
\bar{\partial} R_{n,j}(k)=0, \quad \text { if } k \in D_{n2} \cup D_{n5} \text { or } \operatorname{dist}(k, \mathcal{K}\cup\mathcal{\overline{K}})<\rho/3.\label{Rjk0}
\end{equation}

\end{lemma}

\begin{proof}
Without loss of generality, we only provide the detailed proof for $R_{1,1}$, as other cases can be proved in a similar way.
 Define the function
\begin{equation}
f_{1,1}(k)=r(z_1) T(k)^{2} T_{0}(z_1)^{-2}(k-z_1)^{-2 i \nu(z_1)}, \ k \in \overline{D}_{11}.\label{fj1}
\end{equation}
Let $k-z_1=s e^{i\phi}$, for $k \in \overline{D}_{11}$,
\begin{equation}
R_{1,1}(k)=\left[f_{1,1}(k)+\left(r(\operatorname{Re} k)-f_{1,1}(k)\right) \cos (2 \phi)\right] T(k)^{-2}\left(1-\chi_{\mathcal{K}}(k)\right).
\end{equation}
Clearly, $R_{1,1}(k)$ satisfies the boundary conditions in this lemma, as $\cos (2 \phi)$ vanishes on $\Sigma_{11}$ and $\chi_{\mathcal{K}}(k)$ is zero on the real axis. First, consider $\left|R_{1,1}(k)\right|.$   We have
\begin{equation}
\begin{aligned}
\left|R_{1,1}(k)\right| & \leq\left|2 f_{1,1}(k) T(k)^{-2}\left(1-\chi_{\mathcal{K}}(k)\right)\right| \sin ^{2}(\phi)+\left|T(k)^{-2}\left(1-\chi_{\mathcal{K}}(k)\right) \cos (2 \phi)\right||r(\operatorname{Re} k)| \\
& \leq c \left(\sin ^2(\phi)+\langle\operatorname{Re} k\rangle^{-1 / 2}\right),
\end{aligned}
\end{equation}
where, in the first line, we have bounded the first factor in each term of the sum on the left hand side using Proposition \ref{p}, Equation (\ref{xk}) and (\ref{fj1}), by noting that the poles of $T(k)^{-2}$ are outside the support of $\left(1-\chi_{\mathcal{K}}(k)\right)$. In the second line, we use the fact that $r(k) \in H^{1}(\mathbb{R})$ implies that $r$ is $1 / 2$-H\"{o}lder continuous and $|r(u)| \leq c\langle u\rangle^{-1 / 2}$. As $\bar{\partial}=\left(\partial_{x}+i \partial_{y}\right) / 2=e^{i \phi}\left(\partial_{s}+i \rho^{-1} \partial_{\phi}\right) / 2$, we have
\begin{equation}
\begin{aligned}
\bar{\partial} R_{1,1}(k) &=-\left[f_{1}(k)+\left(r(\operatorname{Re} k)-f_{1,1}(k)\right) \cos 2 \phi\right] T(k)^{-2} \bar{\partial} \chi_{\mathcal{K}}(k) \\
&+\left[\frac{1}{2} r^{\prime}(\operatorname{Re} k) \cos (2 \phi)-i e^{i \phi} \frac{r(\operatorname{Re}(k))-f_{1,1}(k)}{|k-k_j|} \sin (2 \phi)\right] T(k)^{-2}\left(1-\chi_{\mathcal{K}}(k)\right).
\end{aligned}
\end{equation}
According to the continuity and the decay of $r(\operatorname{Re} k)$ described above and the fact that both $1-\chi_{\mathcal{K}}(k)$ and $\bar{\partial} \chi_{\mathcal{K}}(k)$ are supported away from the discrete spectrum, the poles and the zeros of $T(k)$ do not affect the bound. Finally, we arrive at (\ref{Rjgj})-(\ref{Rjk0}), which gives the first two terms in the bound. To calculate the last term we write
\begin{equation}
\left|r(\operatorname{Re} k)-f_{1,1}(k)\right| \leq|r(\operatorname{Re} k)-r(z_1)|+\left|r(z_1)-f_{1,1}(k)\right|,
\end{equation}
and use Cauchy-Schwarz inequality to bound each term as follows:
\begin{equation}
|r(\operatorname{Re} k)-r(z_1)| \leq\left|\int_{z_1}^{\operatorname{Re} k} r^{\prime}(s) d s\right| \leq\|r\|_{H^1(\mathbb{R})}|k-z_1|^{1/2},
\end{equation}
and
\begin{equation}
\begin{aligned}
\left|r(z_1)-f_{1,1}(k)\right| &\leq|r(z_1)|\left(1+|r(z_1)|^2\right)\left|T(k)^2-T_0(z_1)^{2}(k-z_1)^{2 i \nu(z_1)}\right|\nonumber\\
& \leq c \|r\|_{H^1(\mathbb{R})}|k-z_1|^{1 / 2}.\nonumber
\end{aligned}
\end{equation}
The last estimate uses (\ref{tk0}) to bound $T(k)$ and $(k-z_1)^{i \nu(z_1)}$ in a neighborhood of $k=z_1$. (\ref{Rjgj})-(\ref{Rjk0}) for $k \in D_{11}$ follows immediately.
\end{proof}

Now, we construct a matrix function $R^{(2)}(k)$ for $n=1,2,3,4$ by
\begin{equation}
\mathcal{R}^{(2)}(k)=\begin{cases}
\begin{pmatrix}
1 & 0 \\
-R_{n,1}(k) e^{2 i t \theta} & 1
\end{pmatrix}, & k \in D_{n1}, \\
\begin{pmatrix}
1 & -R_{n,3}(k) e^{-2 i t \theta} \\
0 & 1
\end{pmatrix}, & k \in D_{n3}, \\
\begin{pmatrix}
1 & 0 \\
R_{n,4}(k) e^{2 i t \theta}& 1
\end{pmatrix}, & k \in D_{n4}, \\
\begin{pmatrix}
1 & R_{n,6}(k) e^{-2 i t \theta} \\
0 & 1
\end{pmatrix}, & k \in D_{n6}, \\
\begin{pmatrix}
1 & 0 \\
0 & 1
\end{pmatrix}, & k\in D_{n2}\cup D_{n5},
\end{cases}
\end{equation}
by   which  we   define a transformation
\begin{equation}
M^{(2)}(k)=M^{(1)}(k) R^{(2)}(k),\label{M2}
\end{equation}
and obtain the following mixed $\overline{\partial}$-RH problem.

\noindent\textbf{ RH Problem 5.2.} Find a function $M^{(2)}(k)=M^{(2)}(k;x,t)$ with the following properties.

(a) \textit{Analyticity}: $ M^{(2)}(k)$  is continuous in $\mathbb{C}$, sectionally continuous first partial derivatives in $\mathbb{C} \backslash (\Sigma^{(2)} \cup \mathcal{K}\cup\mathcal{\overline{K}} )$,  and meromorphic in $D_{n2}\cup D_{n5}$, $n=1,2,3,4$;

(b) \textit{Symmetry}: $\overline{M^{(2)}(\bar{k})}=\sigma_2 M^{(2)}(k) \sigma_2$;

(c) \textit{Jump condition}: The boundary value $ M^{(2)}(k)$ at $ \Sigma^{(2)}$ satisfies the jump condition
\begin{equation}
M^{(2)}_+ (k) =M^{(2)}_- (k)
V^{(2)}(k),\quad k\in \Sigma^{(2)},
\end{equation}
where
\begin{equation}V^{(2)}(k)=\begin{cases}
\begin{pmatrix}
1 & 0 \\
f_{n,1}(k) e^{2it \theta} & 1
\end{pmatrix}, & k \in \Sigma_{n1},\ n=1,2,3,4,\\
\begin{pmatrix}
1 & f_{n,3}(k) e^{-2 i t \theta}  \\
0 & 1
\end{pmatrix}, & k \in \Sigma_{n2},\ n=1,2,3,4,\\
\begin{pmatrix}
1 & 0 \\
f_{n,4}(k) e^{2 i t \theta} & 1
\end{pmatrix}, & k \in \Sigma_{n3},\ n=1,2,3,4,\\
\begin{pmatrix}
1 & f_{n,6}(k) e^{-2 i t \theta}  \\
0 & 1
\end{pmatrix}, & k \in \Sigma_{n4},\ n=1,2,3,4,\\
\begin{pmatrix}
1 & -(f_{j+1,3}-f_{j,3}) e^{-2it\theta} \\
0 & 1
\end{pmatrix}, & k \in \Sigma_{0j},\ j=1, 3,\\
\begin{pmatrix}
1 & 0 \\
(f_{j+1,4}(k)-f_{j,4}(k)) e^{2it\theta} & 1
\end{pmatrix}, & k \in \Sigma_{0j},\ j=2, 4,\ f_{5,4}\doteq f_{1,4};
\end{cases}
\end{equation}

(d)  \textit{Asymptotic conditions}:
\begin{equation}
\begin{cases}
M^{(2)}(k)=I+\mathcal{O}\left(k^{-1}\right), & k \rightarrow \infty,\\
M^{(2)}(k)=e^{-\frac{i}{2} c_- \sigma_{3}}\left(I+kU+\mathcal{O}\left(k^{2}\right)\right) e^{\frac{i}{2} d_0\sigma_3},  & k \rightarrow 0;
\end{cases}
\end{equation}

(e)  Away from $\Sigma^{(2)}$,  we have
\begin{equation}
\overline{\partial} M^{(2)}= M^{(2)} \overline{\partial} \mathcal{R}^{(2)},
\end{equation}
which holds in $\mathbb{C}\backslash\Sigma^{(2)}$, where
\begin{equation}
\overline{\partial}\mathcal{R}^{(2)}=\begin{cases}
\begin{pmatrix} 1&0\\ (-1)^j \overline{\partial}R_{n,j}(k) e^{2it\theta} &1\end{pmatrix},& k\in D_{nj},\ n=1,2,3,4;\ j=1,4,\\
\begin{pmatrix} 1&(-1)^j \overline{\partial}R_{n,j}(k) e^{-2it\theta}\\
0&1\end{pmatrix},& k\in D_{nj},\ n=1,2,3,4;\ j=3,6,\\
\begin{pmatrix} 0&0\\0&0\end{pmatrix},& k\in D_{nj}, \ n=1,2,3,4;\ j=2,5;
\end{cases}
\end{equation}

(f) \textit{Residue conditions}: $M^{(2)}(k)$ has simple poles at each point in $\mathcal{K}\cup\overline{\mathcal{K}}$ with:
\begin{align}
&\underset{k=k_j}{\operatorname{Res}} M^{(2)}(k)=
\underset{k \rightarrow k_j}{\lim } M^{(2)}(k)R_{\pm}(k_j), \qquad \ k \in \Delta_{k_0}^\pm,  \\
&\underset{k=\overline{k}_j}{\operatorname{Res}} M^{(2)}(k)=
\underset{k \rightarrow \overline{k}_j}{\lim}  M^{(2)}(k)\sigma_2\overline{R_{ \pm} (\overline{k}_j)} \sigma_2, \ k \in \Delta_{k_0}^\pm,
\end{align}
where $R_{-}(k_j)$ and $R_{+}(k_j)$ is defined in (\ref{rkj}).

\subsection{Decomposition of the mixed $\overline{\partial}$-RH problem}
\label{sec:section5}
\hspace*{\parindent}
To solve RH Problem 5.2, we decompose it into a model RH problem for $M^{rhp} ( k)=M^{rhp} ( k;x,t)$
with $\overline{\partial}R^{(2)}=0$ and a pure $\overline{\partial}$-Problem with $\overline{\partial}R^{(2)} \neq 0$. For the first step, we establish a RH problem for the $M^{rhp} (k)$ as follows.

\noindent\textbf{ RH Problem 6.1.} Find a matrix-valued function $M^{rhp} (k)$ with the following properties:

(a) \textit{Analyticity}: $M^{rhp} (k)$ is analytical in $\mathbb{C} \backslash (\Sigma^{(2)} \cup \mathcal{K}\cup\mathcal{\overline{K}}) $;

(b) \textit{Symmetry}: $\overline{M^{rhp}(\bar{k})}=\sigma_2 M^{rhp}(k)\sigma_2$;

(c) \textit{Jump condition}: $M^{rhp} ( k)$  has continuous boundary values  on   $\Sigma^{(2)}$ and
\begin{equation}
M_{+}^{rhp}( k)=M_{-}^{rhp}(k) V^{(2)}(k), \quad k \in \Sigma^{(2)};
\end{equation}

(d) $\bar{\partial}$\textit{-Derivative}: $\bar{\partial} R^{(2)}=0,$ for $k \in \mathbb{C}$;

(e) \textit{Asymptotic conditions}:
\begin{equation}
\begin{cases}
M^{rhp}(k)=I+\mathcal{O}\left(k^{-1}\right), & k \rightarrow \infty,\\
M^{rhp}(k)=e^{-\frac{i}{2} c_- \sigma_{3}}\left(I+kU+\mathcal{O}\left(k^{2}\right)\right) e^{\frac{i}{2} d_0\sigma_3},  & k \rightarrow 0;
\end{cases}
\end{equation}

(f) \textit{Residue conditions}: $M^{rhp}$ has simple poles at each point in $\mathcal{K}\cup\mathcal{\overline{K}}$ with:
\begin{align}
&\underset{k=k_j}{\operatorname{Res}} M^{rhp}( k)=
\underset{k \rightarrow k_j}{\lim } M^{rhp}( k)R_{\pm}(k_j), \qquad \ k \in \Delta_{k_0}^\pm, \label{mrhpr1}\\
&\underset{k=\overline{k}_j}{\operatorname{Res}} M^{rhp}( k)=
\underset{k \rightarrow \overline{k}_j}{\lim}  M^{rhp}( k)\sigma_2\overline{R_{\pm}(\overline{k}_j)}\sigma_2, \ k \in \Delta_{k_0}^\pm,\label{mrhpr2}
\end{align}
with the definition of $R_{\pm}(k_j)$  in (\ref{rkj}). The existence and the asymptotics of $M^{rhp} ( k)$ will be presented in Section \ref{sec:section8}.

$M^{rhp}(k)$ can be used to construct a new matrix function
\begin{equation}
M^{(3)}(k)=M^{(2)}(k) M^{rhp}(k)^{-1},\label{m3}
\end{equation}
in which the analytical component $M^{rhp}(k)$ disappears and it finally becomes a pure $\overline{\partial}$-problem.

\noindent\textbf{ RH Problem 6.2.} Find a matrix-valued function $M^{(3)}(k)=M^{(3)}(k;x,t)$ with the following properties:

(a) \textit{Analyticity}: $ M^{(3)}(k)$  is continuous in $\mathbb{C}$, continuous first partial derivatives in $\mathbb{C} \backslash (\Sigma^{(2)} \cup \mathcal{K}\cup \mathcal{\overline{K}})$;

(b) \textit{Symmetry}: $\overline{M^{(3)}(\bar{k})}=\sigma_2 M^{(3)}(k) \sigma_2$,

(c) \textit{Asymptotic conditions}:
\begin{equation}
\begin{cases}
M^{(3)}(k)=I+\mathcal{O}\left(k^{-1}\right), & k \rightarrow \infty,\\
M^{(3)}(k)=e^{-\frac{i}{2} c_- \sigma_{3}}\left(I+kU+\mathcal{O}\left(k^{2}\right)\right) e^{\frac{i}{2} d_0\sigma_3},  & k \rightarrow 0;
\end{cases}
\end{equation}

(d) $\bar{\partial}$\textit{-Derivative}: For $k\in \mathbb{C}$, we have
$$\bar{\partial} M^{(3)}(k)=M^{(3)} (k)W^{(3)}(k),$$
where
\begin{equation}
W^{(3)}(k)=M^{rhp}(k)(k) \bar{\partial} R^{(2)} M^{rhp}(k)^{-1}(k).
\end{equation}
\begin{proof}
By using the properties of the solutions $M^{(2)}(k)$ and $M^{rhp}(k)$ for  RH Problem 5.1  and RH Problem 6.1, the analyticity and the asymptotics are obtained immediately. As $M^{(2)}$ and $M^{rhp}$ have same jump matrix, we have
\begin{equation}
\begin{aligned}
M_-^{(3)}(k)^{-1} M_+^{(3)}(k) &=M_-^{(2)}(k)^{-1} M_-^{rhp}(k) M_+^{rhp}(k)^{-1} M_+^{(2)}(k) \\
&=M_-^{(2)}(k)^{-1}V^{(2)}(k)^{-1} M_+^{(2)}(k)=I,
\end{aligned}
\end{equation}
from which it follows that $M^{(3)}$ and its first partial derivative extend continuously to $\Sigma^{(2)}$, which immediately gives $M^{(3)}$ has no pole. For $k_j\in \mathcal{K}$, we can verify that the coefficient matrix of the leading term in the Laurent expansion of $M^{(2)}(k)$, denoted as $N_j$, is a nilpotent matrix. Hence, we have the Laurent expansion
of $k_j$ as
\begin{align}
&M^{(2)}(k)=a(k_j)\left[I+\frac{N_j}{k-k_j} \right]+\mathcal{O}(k-k_j), \\
&M^{rhp}(k)=A(k_j)\left[I+\frac{N_j}{k-k_j}\right]+\mathcal{O}(k-k_j),
\end{align}
where $a(k_j)$ and $A(k_j)$ are the constant terms in their Laurent expansion. Then from
$M^{rhp}(k)^{-1}=\sigma_2 M^{rhp}(k)^{T}\sigma_2$, $M^{(3)}$ becomes
\begin{equation}
\begin{aligned}
M^{(3)}(k)&=M^{(2)}(k)M^{rhp}(k)^{-1}\\
&=\left\{a(k_j)\left[I+\frac{N_j}{k-k_j} \right]\right\}\left\{\left[ I-\frac{ N_j}{k-k_j} \right]A(k_j)^{-1}\right\}+\mathcal{O}(k-k_j) \\ &=\mathcal{O}(1),
\end{aligned}
\end{equation}
which has only removable singularities at each $k_j$. The last property follows from the definition of $M^{(3)}(k)$, by exploiting the fact that $M^{rhp}(k)$ has no $\overline{\partial}$ component.
\end{proof}
We construct the solution $M^{rhp}(k)$ of the  RH problem 5.1 in the following form
\begin{equation}
M^{rhp}(k)=\begin{cases}
E(k) M^{(o u t)}(k), & k \notin \mathbb{C} \setminus U_{k_0},\\
E(k) M^{ (in )}(k), & k \in U_{k_0},
\end{cases}\label{mrhp}
\end{equation}
where  $U_{k_0}= \cup_{n=1}^4 U_{z_n}$, and  $U_{z_n}$ is  the neighborhood of $z_n$:
\begin{equation}
U_{z_n}=\left\{k:\left|k -z_n\right| \leq \min \left\{ {k_0}/{2}, \rho / 3\right\} \triangleq \varepsilon\right\},
\end{equation}
which implies that $M^{rhp}(k)$ has no poles in $U_{k_0}$ as $\text{dist}(\mathcal{K}\cup\mathcal{\overline{K}}, \Sigma) > \rho$. $M^{rhp}$ decomposes into two parts $M^{(out)}$ and $M^{(in)}$: $M^{(out)}$ solves a model RH problem obtained by ignoring the jump condition of RH Problem 6.1, which will be solved in Section \ref{sec:section6}; for $M^{(in)}$, its solution can be approximated with parabolic cylinder functions if we let $M^{(in)}$ match to $M^{(2)}$ and a parabolic cylinder model in $U_{k_0}$, which we elaborate on in Section \ref{sec:section7}. Besides, $E(k)$ is an error function, which is a solution of a small-norm RH problem, which is discussed in details in Section \ref{sec:section8}. The jump matrix in RH Problem 6.1 admits the following estimates.
\begin{proposition}
For the jump matrix $V^{(2)}(k)$, we have the following estimates
\begin{align}
&\|V^{(2)}(k)-I\|_{L^{\infty}(\Sigma^{(2)} \setminus U_{k_0})}=\mathcal{O}\left(e^{-\frac{\alpha\beta^2}{12k_0^3}|k-z_n|^3\left(k^{-2}-|k_0|^{-2}\right)t}\right) \label{V2I}, \\
&\left\|V^{(2)}(k)-I\right\|_{L^{\infty}\left(\Sigma_0^{(2)}\right)}=\mathcal{O}\left(e^{-\frac{3\alpha\beta^2}{8k_0^2}t}\right),\label{V2g}
\end{align}
where the contours are defined by
\begin{subequations}
\begin{align}
&\Sigma_{n}^{(2)}=\Sigma_{n1} \cup \Sigma_{n2} \cup \Sigma_{n3} \cup \Sigma_{n4}, \ \ n=1,2,3,4.
\end{align}
\end{subequations}
\end{proposition}
\begin{proof}
Without loss of generality, we only prove (\ref{V2g}) for $k \in \Sigma_{01}$, as other cases follow in
a similar way. By using the definition of $V^{(2)}$ and (\ref{Rjgj}), we have
\begin{equation}
\left\|V^{(2)}(k)-I\right\|_{L^{\infty}\left(\Sigma_{01}\right)} \leq\left\|(R_{2,3}-R_{1,3}) e^{-2it\theta}\right\|_{L^{\infty}\left(\Sigma_{01}\right)}\leq  c\left\|  e^{-2it\theta}\right\|_{L^{\infty}\left(\Sigma_{01}\right)}.\label{v2r}
\end{equation}
Note that $|k|<\sqrt{2} k_0 / 2$. Combined with (\ref{Re}), we find that
\begin{equation}
|e^{-2it\theta}|=e^{-\mathrm{Re}2it\theta}=e^{\frac{t}{2}\alpha\beta^2\mathrm{Im}k^2(\frac{1}{k_0^4}-\frac{1}{|k|^4})}\leq e^{-\frac{3\alpha\beta^2}{8k_0^2}t}\rightarrow 0,\quad as \ t\rightarrow \infty,
\end{equation}
which together with (\ref{v2r}) gives (\ref{V2g}). The calculation of $\Sigma_{n}^{(2)}$, $n=1,2,3,4$ is similar.
\end{proof}
This proposition implies that the jump matrix $V^{(2)}(k)$ goes to $I$ on both $\Sigma^{(2)}\setminus U_{ k_0} $ and $\Sigma^{(2)}_0$, so outside the $U_{k_0} $ there is only a exponentially small error of $t\rightarrow\infty$ by completely ignoring the jump condition of $M^{rhp}(k)$.
According to (\ref{V2g}), it is clear that $V^{(2)}(k)-I$ goes  exponentially on $\Sigma^{(2)}_0$ and uniformly goes to zero as $k\rightarrow 0$. Therefore, the case of the neighborhood $k\rightarrow0$ requires no further discussion.

\section{Outer model RH problem}
\label{sec:section6}
\hspace*{\parindent}
In this section, we build an outer model RH problem and show that its solution can approximated with a finite sum of
soliton solutions. The key lies in the fact that we need the property of $M^{(out)}( k )$ as $k\rightarrow\infty$, from which we obtain the solution of the following outer model problem.

\subsection{Existence of soliton solutions }

\noindent\textbf{RH Problem 7.1.}  Find a matrix-valued function $M^{(out)}( k )$ with the following properties:

(a) \textit{Analyticity}: $M^{(out)}( k )$  is analytical in $\mathbb{C} \backslash(  \mathcal{K}\cup \mathcal{\overline{K}})$;

(b) \textit{Symmetry}: $\overline{M^{(out)}(\bar{k})}=\sigma_2 M^{(out)}(k) \sigma_2$;

(c) \textit{Asymptotic conditions}:
\begin{equation}
\begin{cases}
M^{out}(k)=I+\mathcal{O}\left(k^{-1}\right), & k \rightarrow \infty,\\
M^{out}(k)=e^{-\frac{i}{2} c_- \sigma_{3}}\left(I+kU+\mathcal{O}\left(k^{2}\right)\right) e^{\frac{i}{2} d_0\sigma_3},  & k \rightarrow 0;
\end{cases}
\end{equation}

(d) \textit{Residue conditions}: $M^{(out)}$ has simple poles at each point in $\mathcal{K}\cup\mathcal{\overline{K}}$ satisfying the residue conditions (\ref{mrhpr1}) and  (\ref{mrhpr2}).

To show the existence and the uniqueness of the solution to RH Problem 7.1, we first
consider RH Problem 3.1 under the condition of no reflection, i.e., $r(k)\equiv 0$. In this case the unknown function $M(k)$ is meromorphic. The RH Problem 3.1 reduces to the following RH problem.

\noindent\textbf{RH Problem 7.2.} Given the discrete data $\mathcal{S}=\{(k_j, c_j)\}_{j=1}^{2N}$, find a matrix-valued
function $m(k|\mathcal{S})$ with the following properties:

(a) \textit{Analyticity}: $m( k |\mathcal{S})$  is analytical in $\mathbb{C} \backslash(  \mathcal{K}\cup\mathcal{\overline{K}})$,

(b) \textit{Symmetry}:
\begin{equation}
\overline{m(\overline{k}|\mathcal{S})}=\sigma_2 m(k|S)\sigma_2,
\end{equation}

(c) \textit{Asymptotic conditions}:
\begin{equation}
\begin{cases}
m(k|\mathcal{S})=I+\mathcal{O}\left(k^{-1}\right), & k \rightarrow \infty,\\
m(k|\mathcal{S})=e^{-\frac{i}{2} c_- \sigma_{3}}\left(I+kU+\mathcal{O}\left(k^{2}\right)\right) e^{\frac{i}{2} d_0\sigma_3},  & k \rightarrow 0;
\end{cases}
\end{equation}

(d) \textit{Residue conditions}: $m( k |\mathcal{S})$ has simple poles at each point in $\mathcal{K}\cup\mathcal{\overline{K}}$ satisfying
\begin{align}
&\underset{k=k_j}{\operatorname{Res}} m(  k  |\mathcal{S})=\lim _{k \rightarrow k_j} m( k |\mathcal{S}) \tau_j, \\
&\underset{k=\overline{k}_j}{\operatorname{Res}} m(  k  |\mathcal{ S})=\lim _{k \rightarrow \overline{k}_j} m(k| \mathcal{S}) \widehat{\tau}_j,
\end{align}
where $\tau_j$ is a nilpotent matrix,
\begin{equation}
\tau_j=\begin{pmatrix}
0 & 0 \\
\gamma_j & 0
\end{pmatrix}, \quad \widehat{\tau}_j=\sigma_2 \overline{\tau_j} \sigma_2, \quad \gamma_j=c_j e^{-2 i t \theta(k_j)},
\end{equation}
Moreover, the solution satisfies
\begin{equation}
\left\|m( k |\mathcal{S})^{-1}\right\|_{L^{\infty}(\mathbb{C}\backslash(\mathcal{K}\cup\mathcal{\overline{K}}))} \lesssim 1.\label{md1}
\end{equation}

\begin{proposition}
The RH Problem 7.2  has a unique solution.
\end{proposition}
\begin{proof}
The uniqueness of the solution follows from the Liouville's theorem. From the symmetries of
$m\left(k|\mathcal{S}\right)$ and the residue of $m( k |\mathcal{S})$ at $k=k_j$, we know that it admits a partial fraction expansion with the following form
\begin{equation}
m( k |\mathcal{S})=I+\sum_{j=1}^{2N}\left[\frac{1}{k-k_j}\begin{pmatrix}
\nu_{j}(x, t) & 0 \\
\zeta_{j}(x, t) & 0
\end{pmatrix}+\frac{1}{k-\overline{k}_j}\begin{pmatrix}
0& -\overline{\zeta}_{j}(x, t)  \\
0 & \nu_{j}(x,t)
\end{pmatrix}\right].\label{md2}
\end{equation}
Since $\mathrm{det}(m\left( k |\mathcal{S}\right))=1$, $\left\|m\left( k |\mathcal{S}\right)\right\|_{L^{\infty}(\mathbb{C} \backslash(\mathcal{K}\cup\mathcal{\overline{K}}))}$
is bounded, by using a similar technique as in Appendix B and from (\ref{md2}), we simply obtain (\ref{md1}) and prove the existence of the solution for the RH Problem 7.2.
\end{proof}
In the reflectionless case, the transmission coefficient admits the following trace formula
\begin{equation}
a(k)=\prod_{j=1}^{2N} \left(\frac{k-k_j}{k-\overline{k}_j}\right),
\end{equation}
whose poles can be split into two parts. Let $\triangle \subseteq\{1,2, \ldots, 2N\}$, we define
\begin{equation}
a_{\Delta}(k)=\prod_{j \in \Delta} \left(\frac{k-k_j}{k-\overline{k}_j}\right).
\end{equation}

With a renormalization transformation
\begin{equation}
m^{\Delta}(k |\mathcal{ D})=m\left(k|\mathcal{S}\right)a_{\Delta}(k)^{\sigma_3},\label{mkd}
\end{equation}
where the scattering data are given by
\begin{equation}
\mathcal{D}=\left\{\left(k_j, c_j^{\prime}\right)\right\}_{j=1}^{2N}, \quad c_j^{\prime}=c_j a_{\Delta}(k)^{2},\label{Dcj}
\end{equation}
it is clear that the transformation (\ref{mkd}) splits the poles of $m_\Delta(k|\mathcal{S})$ into two columns by the choice of $\Delta$, after which it satisfies the following modified discrete RH problem.

\noindent\textbf{RH Problem 7.3.} Given the discrete data (\ref{Dcj}), find a matrix-valued function $m^\Delta( k|\mathcal{D})$ with the following properties:

(a) \textit{Analyticity}: $m^\Delta( k|\mathcal{D})$ is analytical in $\mathbb{C} \backslash(\mathcal{K}\cup\mathcal{\overline{K}})$;

(b) \textit{Symmetry}: $\overline{m^{\Delta}(\overline{k}|\mathcal{D})}=\sigma_2m^{\Delta}(k|\mathcal{D})\sigma_2$;

(c) \textit{Asymptotic conditions}:
\begin{equation}
\begin{cases}
m^{\Delta}(k|\mathcal{D})=I+\mathcal{O}\left(k^{-1}\right), & k \rightarrow \infty,\\
m^{\Delta}(k|\mathcal{D})=e^{-\frac{i}{2} c_- \sigma_{3}}\left(I+kU+\mathcal{O}\left(k^{2}\right)\right) e^{\frac{i}{2} d_0\sigma_3},  & k \rightarrow 0;
\end{cases}
\end{equation}

(d) \textit{Residue conditions}: $m^{\Delta}( k|\mathcal{D})$ has simple poles at each point in $\mathcal{K}\cup\mathcal{\overline{K}}$ satisfying
\begin{align}
&\underset{k=k_j}{\operatorname{Res}} m^{\Delta}( k|\mathcal{D})=\lim _{k \rightarrow k_j} m^{\Delta}( k |\mathcal{D}) \tau_j^{\Delta},\label{714} \\
&\underset{k=\overline{k}_j}{\operatorname{Res}} m^{\Delta}( k |\mathcal{D})=\lim _{k \rightarrow \overline{k}_j} m^{\Delta}( k |\mathcal{D}) \widehat{\tau}_j^{\Delta},
\end{align}
where $\tau_j^{\Delta}$ is a nilpotent matrix satisfies
\begin{equation}
\tau_j^{\Delta}=\begin{cases}
\begin{pmatrix}
0 & 0 \\
\gamma_j a^{\Delta}(k_j)^2 & 0
\end{pmatrix},&  j \notin \Delta, \\
\begin{pmatrix}
0 & \gamma_j^{-1} a^{\prime \Delta}(k_j)^{-2} \\
0 & 0
\end{pmatrix}, & j \in \Delta,
\end{cases}
\end{equation}
where
\begin{align}
&\widehat{\tau}_j^{\Delta}=\sigma_2\overline{\tau^{\Delta}_j}\sigma_2,\ \ \ \gamma_j=c_je^{2it\theta(k_j)}.\label{ga}
\end{align}
Since (\ref{mkd}) is an explicit transformation of $m(k|\mathcal{S})$, by Proposition 5, we obtain the
existence and the uniqueness of the solution of the RH Problem 7.3.

In the  RH Problem 7.3, taking $\Delta=\Delta^-_{k_0}$ and replacing the scattering data $\mathcal{D}$ with the scattering data
\begin{equation}
\widetilde{\mathcal{D}}=\{(k_j,\widetilde{c_j})\}_{k=1}^{N}, \quad \widetilde{c}_j=c_j \delta(k_j)^2,\label{d}
\end{equation}
we have the following corollary.
\begin{corollary}
There exists a unique solution for the RH Problem 7.1, which has the form
\begin{equation}
M^{out}(k)=m^{\Delta^-_{k_0}}(k|\widetilde{\mathcal{D}}),
\end{equation}
where scattering data $\widetilde{\mathcal{D}}$ is given by (\ref{d}).
\end{corollary}

\subsection{Long-time behavior of soliton solutions}
\hspace*{\parindent}
By choosing $\Delta$ appropriately, the asymptotic limits $t\rightarrow \infty$ with $\xi=\frac{x}{t}+\alpha$ and $k_0=(\frac{\alpha \beta^{2}}{4\xi})^{\frac{1}{4}}$ bounded are under a better asymptotic control. Then we consider the long-time behavior of soliton solutions.

 By using the residue coefficients (\ref{ga}), if $N=1$, in the reflectionless case, one soliton solution is given by
\begin{equation}
u_x(x,t)=2\zeta \operatorname{sech}\left[4 \xi \zeta(x-vt)-\delta_{0}\right] e^{-2 i \chi(x, t)+i \kappa_{0}} e^{-i\left\{\frac{\zeta}{\xi} \tanh \left[4 \xi \zeta(x-vt)-\delta_{0}\right]+\left|\frac{\zeta}{\xi}\right|\right\}},
\end{equation}
in which
$\delta_0$ and $\kappa_0$ are real parameters and the speed of soliton is
\begin{equation}
v=-\alpha\left(1-\frac{\beta^{2}}{4|k|^{4}}\right).
\end{equation}

Given a pair of points $  x_{1} \leq x_{2}$ with $ x_1, x_2 \in \mathbb{R}$ with velocities $v_{1} \leq v_{2}$ with  $ v_{1},  v_{2} \in \mathbb{R}^{-}$, we define a cone
\begin{equation}
C(x_1, x_2, v_1, v_2)=\{(x,t) \in \mathbb{R}^2 | x=x_0+v t, x_{0} \in[x_1, x_2], v \in[v_1, v_2]\},
\end{equation}
and denote
\begin{align}
&\mathcal{I}=\{k: f(v_2) <|k|<f(v_1)\}, \ \ f(v) \doteq (\frac{\alpha\beta^2}{4(v+\alpha)})^{1/4},\nonumber\\
&\mathcal{K}(\mathcal{I})=\{k_j\in \mathcal{K}:  k_j \in \mathcal{I}\},\quad N(\mathcal{I})=|\mathcal{K}(\mathcal{I})|,\nonumber\\
&\mathcal{K}_+(\mathcal{I})=\{k_j\in \mathcal{K}:| k_j|> f(v_1)\}\nonumber, \quad \mathcal{K}_-(\mathcal{I})=\{k_j\in\mathcal{K}:| k_j|<f(v_2)\},\nonumber \\
&c_j^\pm(\mathcal{I})=c_j\prod_{\operatorname{Re} k_n \in I_\pm \backslash \mathcal{I}}\left(\frac{k_j-k_n}{k_j-\overline{k}_n}\right)^2 \exp [\pm\frac{1}{\pi i} \int_{I_\pm} \frac{\log [1+r(\zeta)\overline{r(\overline{\zeta}})]}{\zeta-k_j} d \zeta].\label{cji}
\end{align}

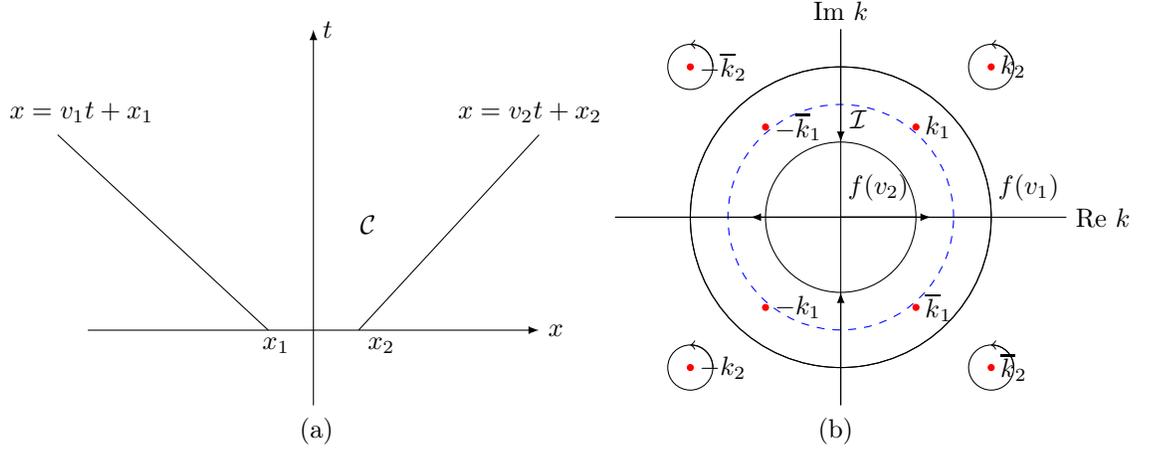
\begin{figure}[H]
\begin{minipage}[b]{5cm}
\begin{tikzpicture}[node distance=1cm]
%\draw [blue!10, fill=blue!10] (0.6,0) -- (3.8,3) -- (-3.4,3) -- (-0.6,0);
\draw [-latex](-3,0)--(3,0)  node[right, scale=1] {$x$};
\draw [-latex](0,-1)--(0,4)  node[right, scale=1] {$t$};
\draw [](0.6,0)--(3,2.6)  node[right, scale=1] {};
%\draw [](0.6,0)--(1.4,-3)  node[above, scale=1] {};
\draw [](-0.6,0)--(-3.4,2.6)  node[above, scale=1] {};
%\draw [](-0.6,0)--(-0.8,-3)  node[above, scale=1] {};
\node  [right]  at (0.6,-0.2) {$x_2$};
\node  [right]  at (-0.8,-0.2) {$x_1$};
\node  [right]  at (1.8,2.9) {$x=v_2t+x_2$};
%\node  [right]  at (1.4,-3) {$x=v_1t+x_2$};
\node  [left]  at (-2,2.9) {$x=v_1t+x_1$};
%\node  [left]  at (-0.8,-3) {$x=v_2t+x_1$};
%\draw[fill] (1,0) circle [radius=0.04];
%\draw[fill] (-1,0) circle [radius=0.04];
\node  [right]  at (0.5, 1.4) {$\mathcal{C}$};
\end {tikzpicture}

\centerline{$\quad\quad\quad\quad\quad\quad\quad\quad\quad\quad$(a) }
\end{minipage}
\hspace{3.0cm}
\begin{minipage}[b]{6cm}
\begin{tikzpicture}[node distance=2cm]
 \draw [] (0,0) circle (2cm);
\draw [fill=pink,ultra thick,white] (0,0) circle (1cm);
\draw [dashed,blue] (0,0) circle (1.5cm);
\draw [ ](-3,0)--(3,0)  node[right, scale=1] {Re $k$};
\draw [-latex](-3,0)--(1.2,0);
\draw [-latex](0,0)--(-1.2,0);
\draw [ ](0,-2.5)--(0,2.5)  node[above, scale=1] {Im $k$};
\draw [-latex](0, 2.5)--(0,1);
\draw [-latex](0, -2.5)--(0,-1);
\draw [](0,0)--(1,0)  node[above, scale=1] {};
\draw [](0,0)--(2,0)  node[above, scale=1] {};
\draw [](1,0)--(0,0)  node[above, scale=1] {};
\draw(0,0)circle(2cm);
\draw(0,0)circle(1cm);
\draw[fill,red] (1,1.2) circle [radius=0.04];
\draw[fill,red] (2,2) circle [radius=0.04];
\draw[fill,red] (2,-2) circle [radius=0.04];
\draw[fill,red] (-2,-2) circle [radius=0.04];
\draw[fill,red] (-2,2) circle [radius=0.04];
\draw[fill,red] (1,-1.2) circle [radius=0.04];
\draw[fill,red] (-1,-1.2) circle [radius=0.04];
\draw[fill,red] (-1,1.2) circle [radius=0.04];
\node  [right]  at (2,2) {$k_2$};
\draw [] (2,2)circle(0.3cm);
 \draw [  -> ]  (2.3,2) to  [out=90, in=0]  (2, 2.3);	
\node  [right]  at (2,-2) {$\overline{k}_2$};
\draw [](2,-2)circle(0.3cm);
 \draw [ -> ]  (2.3,-2) to  [out=90, in=0]  (2, -1.7);
\node  [right]  at (-2,-2) {$-k_2$};
\draw [](-2,-2)circle(0.3cm);
 \draw [  -> ]  (-1.7,-2) to  [out=90, in=0]  (-2, -1.7);
\node  [right]  at (-2,2) {$-\overline{k}_2$};
\draw [] (-2,2)circle(0.3cm);
 \draw [  -> ]  (-1.7,2) to  [out=90, in=0]  (-2, 2.3);
\node  [right]  at (1,1.2) {$k_1$};
\node  [right]  at (1,-1.2) {$\overline{k}_1$};
\node  [right]  at (-1,-1.2) {$-k_1$};
\node  [right]  at (-1,1.2) {$-\overline{k}_1$};
\node  [right]  at (0,1.3) {$\mathcal{I}$};
\node  [above]  at (0.5,0.1) {$f(v_2)$};
\node  [above]  at (2.5,0.1) {$f(v_1)$};
\end {tikzpicture}
\centerline{(b) }
\end{minipage}
\caption{\small (a) The time-spatial cone
$\mathcal{C}(x_1 ,x_2 ,v_1 ,v_2)$; (b) In this example, the original scattering data has two pairs of zero points of discrete
spectrums, but inside the cone $\mathcal{C}(x_1 ,x_2 ,v_1 ,v_2)$ only one pair of points with $\mathcal{K}(\mathcal{I}) = \{k_1,-k_1,\overline{k}_1,-\overline{k}_1\}$, and $f(v_2)$, $f(v_1)$.}
\label{zdtx}
\end{figure}
For the $(x,t) \in \mathcal{C}\left(x_1, x_2, v_1, v_2\right)$,  a direct calculation leads to
\begin{equation}
k \in \mathcal{I}=\left\{k: f\left(v_2\right) \leq|k| \leq f\left(v_1\right)\right\},
\end{equation}
which implies that a proportion of the discrete spectrums would fall in the circular region
$$
\mathcal{K}(\mathcal{I})=\left\{k_{j} \in \mathcal{K} \mid k_{j} \in \mathcal{I}\right\}.
$$

Define
$$
\mu=\min _{k_{j} \in \mathcal{K} \backslash \mathcal{K}(\mathcal{I})}\left\{\operatorname{Re} k_{j}\operatorname{Im} k_{j} {\rm dist}  \left(v_{k_j}-v\right)\right\},
$$
where
\begin{equation}
v_{k_j}=-\alpha\left(1-\frac{\beta^{2}}{4\left|k_{j}\right|^{4}}\right),
\end{equation}
is the velocity of one soliton corresponding to the discrete spectrum $k_j$.
The value  $v_{k_j}-v$ represents the difference between the velocity of one soliton corresponding to $k_j$.

\begin{proposition}
The choice of normalization $\Delta=\Delta_{k_0}^{\mp}$ in  RH Problem 7.3 ensures that as $t\rightarrow \infty$ with $(x,t) \in C\left(x_1, x_2, v_1, v_2\right)$ that
\begin{equation}
\left\|\tau_j^{\Delta_{k_0}^{\mp}}\right\|=\begin{cases}
\mathcal{O}(1), & k_j \in \mathcal{K}(\mathcal{I}), \\
\mathcal{O}(e^{-4\mu t}), & k_j \in \mathcal{K} \backslash \mathcal{K}(\mathcal{I}),
\end{cases} \quad t \rightarrow \infty.
\end{equation}
\end{proposition}

\begin{proof}
We take $\Delta=\Delta^-_{k_0}$ in the  RH Problem 7.3. Then, for $k_j\in   \mathcal{K}_-(\mathcal{I})$ and $(x,t)\in \mathcal{C}(x_1,x_2,v_1,v_2)$,
by using the residue condition (\ref{714}),
\begin{equation}
\left|\gamma_{j}(x_{0}+vt,k_j)\right| \leq c\left|e^{-2it\theta(k_{j})}\right|,
\end{equation}
for the exponential part on the right, we have
\begin{equation}
\begin{aligned}
-2 i t \theta(k_{j})&=-2 i x_0 k_j^2-2 i t(v k_j^2+\alpha k_j^2-\alpha \beta+\frac{\alpha \beta^{2}}{4 k_j^2}).
\end{aligned}
\end{equation}
To get the modulus of $e^{-2it\theta(k_{j})}$, we calculate the real parts of the right side
\begin{equation}
\begin{aligned}
\operatorname{Re}\left(-2ix_0 k_j^2-2 i t(v k_j^2+\alpha k_j^2-\alpha\beta+\frac{\alpha \beta^2}{4 k_j^2})\right)=2 x_0 \operatorname{Re} k_j\operatorname{Im} k_j-4t \operatorname{Re} k_j\operatorname{Im} k_j(v_j-v),
\end{aligned}
\end{equation}
where we set
\begin{equation}
v_{k_j}=-\alpha\left(1-\frac{\beta^{2}}{4\left|k_{j}\right|^{4}}\right).
\end{equation}
Therefore,
\begin{align}
&|  e^{-2it\theta(k_j)}| =e^{ 2x_0 {\rm Re}k_j{\rm Im}k_j} e^{-4t{\rm Re}k_j{\rm Im}k_j(v_{k_j}-v)} \leq c e^{-4\mu t}.
\end{align}
\end{proof}

By using the above estimates, we obtain the following result.
\begin{lemma}
Fix the reflectionless data $\mathcal{S}=\left\{\left(k_j, c_j\right)\right\}_{j=1}^N, \mathcal{D}^\pm(\mathcal{\mathcal{I}})=\left\{\left(k_j, c_j^\pm(\mathcal{I})\right) | k_j \in \mathcal{K}(\mathcal{I})\right\}$. Then as $t\rightarrow \infty$ with $(x,t) \in \mathcal{C}\left(x_1, x_2, v_1, v_2\right)$, we have
\begin{equation}
m^{\Delta_{k_0}^\mp}(k|\mathcal{D})=(I+\mathcal{O}(e^{-4\mu t})) m^{\Delta_{k_0}^\mp}(k|\mathcal{D}^\pm(\mathcal{I})),
\end{equation}
with
$$
c_j^\pm(\mathcal{I})=c_j\prod_{\operatorname{Re} k_n \in I_\pm \backslash \mathcal{I}}\left(\frac{k_j-k_n}{k_j-\overline{k}_n}\right)^2.
$$
\end{lemma}

\begin{proof}
Denote small disks centered at each $k_j \in \mathcal{K} \backslash \mathcal{K}(\mathcal{I})$ with a radius smaller than
$\mu(\mathcal{I})$ as $S_j$ respectively, with $\partial S_j$ representing the boundary of $S_j$. Then we introduce a new transformation which
can remove the poles $k_j \in \mathcal{K} \backslash \mathcal{K}(I)$ and these residues change to near-identity jumps.
\begin{equation}
\widetilde{m}^{\Delta_{k_0}^-}(k|\mathcal{D})=\begin{cases}
m^{\Delta_{k_0}^-}(k|\mathcal{D})(I-\frac{\tau_j}{k-k_j}),  & k \in S_j, \\
m^{\Delta_{k_0}^-}(k|\mathcal{D})(I-\frac{\sigma_{2} \overline{\tau_j} \sigma_{2}}{k-\overline{k_j}}),  & k \in \overline{S}_j, \\
m^{\Delta_{k_0}^-}(k|\mathcal{D}),  & \text{elsewhere}.
\end{cases}\label{mk0}
\end{equation}
Comparing with $m^{\Delta_{k_0}^{-}}(k|\mathcal{D})$  the new matrix function $\widetilde{m}^{\Delta_{k_0}^-}(k|\mathcal{D})$  has a new jump in each $\partial S_j$  which we denote by $\widetilde{V}(k)$
\begin{equation}
\widetilde{m}_+^{\Delta_{k_0}^-}(k|\mathcal{D})=\widetilde{m}_-^{\Delta_{k_0}^-}(k| \mathcal{D}  )\widetilde{V}(k), \quad k\in \widetilde{\Sigma},\label{mk0m}
\end{equation}
where
$$
\widetilde{\Sigma}=\cup_{k_j\in \mathcal{K} \backslash \mathcal{K}(\mathcal{I})}(\partial S_j \cup \partial \overline{S}_j).
$$
Then using (\ref{mk0}), we have
\begin{equation}
\|\widetilde{V}(k)-I\|_{L^{\infty}(\widetilde{\Sigma})}=\mathcal{O}(e^{-4\mu t}). \label{Vk}
\end{equation}
Since $\widetilde{m}^{\Delta_{k_0}^-}(k|\mathcal{D})$ has the same poles and residue conditions with $m^{\Delta_{k_0}^-}(k|\mathcal{D})$, then
\begin{equation}
m_{0}(k)=\widetilde{m}^{\Delta^-_{k_0}}(k|\mathcal{D}) m^{\Delta^-_{k_0}(\mathcal{I})}(k|\mathcal{D}^\pm(\mathcal{I}))^{-1},\label{m0k}
\end{equation}
has no poles, but it has a jump matrix for $k\in\widetilde{\Sigma}$
\begin{equation}
m_{0}^+(k)=m_{0}^-(k) V_{m_{0}}(k).\label{vm0}
\end{equation}
Combine (\ref{mk0m}), (\ref{m0k}), (\ref{vm0}) and under the fact that the solution of  RH Problem 7.2  has no jump, the jump matrix $V_{m_0}(k)$ is given by
\begin{equation}
V_{m_0}(k)=m(k|\mathcal{D}^\pm(\mathcal{I}))\widetilde{V}(k) m(k|\mathcal{D}^\pm(\mathcal{I}))^{-1},
\end{equation}
which, by using (\ref{Vk}), also admits the same decaying estimate
\begin{equation}
\left\|V_{m_0}(k)-I\right\|_{L^{\infty}(\widetilde{\Sigma})}=\|\widetilde{V}(k)-I\|_{L^{\infty}(\widetilde{\Sigma})}=\mathcal{O}(e^{-4 \mu t}), \quad t \rightarrow  \infty.
\end{equation}

\end{proof}

Using the reconstruction formula to $m^{\Delta_{k_0}^-}(k|\mathcal{D})$, we immediately obtain the following result.

\begin{corollary}
Let $m_{sol}(x,t|\mathcal{D})$ and $m_{sol}(x,t|\mathcal{D}(\mathcal{I}))$ denote the N-soliton solution of (\ref{cs})
corresponding to discrete scattering data $S$ and $\mathcal{D}(\mathcal{I})$ respectively, with $\mathcal{I}$, $\mathcal{C}(x_1 ,x_2 ,v_1 ,v_2)$, $\mathcal{D}(\mathcal{I})$ given above. As $t\rightarrow \infty$
with $(x,t)\in C(x_1 ,x_2 ,v_1 ,v_2)$, we have
\begin{equation}
2i\lim _{k \rightarrow \infty} k(m(k|\mathcal{D}))_{12}=m_{sol}(x,t|\mathcal{D})=m_{sol}(x,t|\mathcal{D}(\mathcal{I}))+O(e^{-4\mu t}).
\end{equation}
\end{corollary}
From the outer model we arrive at the following corollary.
\begin{corollary}
The  RH Problem 7.1  has a unique solution $M^{(out)}$ with
\begin{equation}
\begin{aligned}
M^{(out)}(k)&=m^{\Delta_{k_0}^-}(k|\mathcal{D}^{(out)}) \\
&=m^{\Delta_{k_0}^-}(k| \mathcal{D}(\mathcal{I})) \prod_{\operatorname{Re} k_j \in I_+ \backslash \mathcal{I}}\left(\frac{k-k_j}{k-\overline{k}_j}\right)^{-\sigma_3} \delta^{-\sigma_3}+\mathcal{O}(e^{-4\mu t}),\label{mout}
\end{aligned}
\end{equation}
where $\mathcal{D}^{(out)}=\left\{k_j, c_j(z_n)\right\}_{j=1}^{2N}$, $n=1,2,3,4$,
\begin{equation}
c_j(z_n)=c_j \exp \left[-\frac{1}{\pi i} \int_{I_+} \frac{\log \left[1+r(\zeta)\overline{r(\overline{\zeta})}\right]}{\zeta-z_n} d \zeta\right].
\end{equation}
Then substituting (\ref{mout}) into (\ref{md1}) we immediately have
\begin{equation}
\|M^{(o u t)}(k)^{-1}\|_{L^{\infty}(\mathbb{C}\backslash\mathcal{K}\cup\mathcal{\overline{K}})} \lesssim 1,\label{Mout}
\end{equation}
Moreover, we have the reconstruction formula
\begin{equation}
 M^{(out)}(k)  =m(x,t|\mathcal{D}^{(out)})=m(x,t|\mathcal{D}^{(out)}(\mathcal{I}))+\mathcal{O}(e^{-4 \mu t}),
\end{equation}
and
\begin{equation}
m_{sol}(x,t|\mathcal{D}^{(out)})=m_{sol}(x,t|\mathcal{D}(\mathcal{I}))+\mathcal{O}(e^{-4\mu t}),\quad t\rightarrow \infty.\label{uxout}
\end{equation}
\end{corollary}

\section{A local solvable RH model near phase points}
\label{sec:section7}
\hspace*{\parindent}
From Proposition 4, we find that $V^{(2)}(k)-I$ in the neighborhood $U_{ k_0}$ of $z_n$, $n=1,2,3,4$ does not have a uniformly small jump for $t\rightarrow\infty$. Therefore, we establish a local model $M^{(in)}(k)$ which exactly matches the jumps of $M^{rhp}(k)$ on $ \Sigma^{(2)}\cap U_{k_0}$ for the function $E(k)$. Then it has a uniform estimate on the decay of the jump.

\noindent\textbf{ RH Problem 8.1.}  Find a matrix-valued function
$M^{fl}(k)$ such that

(a) \textit{Analyticity}: $M^{fl}(k)$ is analytical in $\mathbb{C}\backslash \Sigma^{fl}$, $\Sigma^{fl}=\cup_{n,j=1}^4 \Sigma_{nj}$;

(b) \textit{Asymptotic conditions}:
\begin{equation}
\begin{cases}
M^{fl}(k)=I+\mathcal{O}\left(k^{-1}\right), & k \rightarrow \infty,\\
M^{fl}(k)=e^{-\frac{i}{2} c_- \sigma_{3}}\left(I+kU+\mathcal{O}\left(k^{2}\right)\right) e^{\frac{i}{2} d_0\sigma_3},  & k \rightarrow 0;
\end{cases}
\end{equation}

(c) \textit{Jump condition}: $M^{fl}(k)$ has continuous boundary values $M_\pm^{fl}(k)$ on $\Sigma^{fl}$ and
\begin{equation}
M_{+}^{fl}(k)=M_-^{fl}(k) V^{fl}(k), \quad k \in \Sigma^{fl},
\end{equation}
where the jump matrix $V^{fl}(k)$ is given by
\begin{equation}V^{fl}(k)=\begin{cases}
\begin{pmatrix}
1 & 0 \\
r(z_n) \delta^{-2}(z_n) (k-z_n)^{-2i\nu(z_n)} e^{2it\theta(k)} & 1
\end{pmatrix}, & k \in \Sigma_{n1},\ n=1,2,3,4,\\[12pt]
\begin{pmatrix}
1 &\frac{\overline{ r(\bar z_n)} }{1+r(z_n)\overline{ r(\bar z_n)}}\delta(z_n)^2 (k-z_n)^{2i\nu(z_n) } e^{-2 i t \theta(k)}  \\
0 & 1
\end{pmatrix}, & k \in \Sigma_{n2},\ n=1,2,3,4,\\[12pt]
\begin{pmatrix}
1 & 0 \\
\frac{r(z_n)}{1+r(z_n)\overline{ r(\bar z_n)}}\delta(z_n)^{-2}(k-z_n)^{-2i\nu(z_n)} e^{2 i t \theta(k)} & 1
\end{pmatrix}, & k \in \Sigma_{n3},\ n=1,2,3,4,\\[15pt]
\begin{pmatrix}
1 & \overline{ r(\bar z_n)}\delta(z_n)^2(k-z_n)^{2i\nu(z_n) } e^{-2it\theta(k)}  \\
0 & 1
\end{pmatrix}, & k \in \Sigma_{n4},\ n=1,2,3,4.
\end{cases}
\end{equation}

For the soliton-free case when there are no discrete spectrum, the formula (\ref{t0zn})
reduces to $T_{0}\left(z_n\right)=\delta\left(z_n\right)$. We take $z_3$ as an example and other three stationary-phase points can be handled in the same way. Expanding $\theta(k)$, we obtain
\begin{equation}
\begin{aligned}
\theta(k)&=\frac{\alpha\beta^2}{4}(\frac{k^2}{k_0^4}+\frac{1}{k^2})-\alpha\beta\\
&=\frac{\alpha\beta^2}{2k_0^2}-\alpha\beta+\frac{\alpha\beta^2}{k_0^4}(k-z_3)^2+(\frac{\alpha\beta^2}{4k_0^3 }-\frac{3\alpha\beta^2}{4k_0^4 })(k-z_3)^3.
\end{aligned}
\end{equation}
We define the following scaling transformation
\begin{equation}
N: f(k) \rightarrow(N f)(k)=f(\frac{k_{0}^{2}}{2 \beta \sqrt{\alpha t}} \zeta+z_3),\label{naf}
\end{equation}
which acts on $\delta(k) e^{-i t \theta(k)} $ and gives
\begin{equation}
\left(N\delta e^{-it\theta}\right)(k)=\delta^{0}(\zeta)\delta^{1}(\zeta),
\end{equation}
where
\begin{equation}
\delta^0(\zeta)=\frac{z_3^{i\nu(z_1)-2i\nu(z_2)}}{(\sqrt{\alpha t} \beta)^{i\nu(z_1)}} 2^{-i\nu(z_2)} e^{i(\alpha\beta- \frac{\alpha\beta^2}{2 z_3^2}) t} e^{\chi_3(z_3)} e^{\chi_3^{\prime}(z_3)},
\end{equation}
and
\begin{equation}
\begin{aligned}
\delta^{1}(\zeta)&=\zeta^{i \nu} e^{-i\frac{\zeta^{2}}{4}}e^{-i\frac{k_0^3 \zeta^2}{32\sqrt{\alpha t}\beta}(\frac{1}{\zeta}-\frac{3}{k_0})} \frac{z_3^{2i\nu(z_2)+i \nu(z_3)}}{2^{i \nu(z_3)-i\nu(z_2)}} \frac{(\frac{k_0^2}{2 \sqrt{\alpha t} \beta} \zeta+2 z_3)^{i \nu}}{(\frac{k_0^2}{2 \sqrt{\alpha t}\beta} \zeta+z_3)^{2i\nu}} \\
& \times((\frac{k_0^2}{2 \sqrt{\alpha t} \beta} \zeta+z_3+z_2)(\frac{k_0^2}{2 \sqrt{\alpha t} \beta} \zeta+z_3+z_4))^{-i \widetilde{\nu}} \\
& \times e^{\chi_3(\frac{k_{0}^{2}}{2 \sqrt{at}\beta} \zeta+z_3)-\chi_3(z_3)} e^{\tilde{\chi}_3^{\prime}(\frac{k_{0}^{2}}{2 \sqrt{at}\beta} \zeta+z_3)-\tilde{\chi}_3^{\prime}(z_3)},
\end{aligned}
\end{equation}
for which
\begin{equation}
\chi_n^{\prime}(k)=e^{-\frac{1}{2 \pi i} \int_{0}^{z_n} \ln \left|k-k^{\prime}\right| d \ln (1-r(k^{\prime})\overline{r(\overline{k}^{\prime})})},\quad n=1,2,3,4.
\end{equation}
The result here is based on \cite{Lta}. Instead of using the Taylor expansion of $\theta(k)$, we explicitly write the coefficients of the power terms in the form of splitting. The benefits of this method will be made clear in our subsequent calculations.

\begin{proposition}
As $t \rightarrow \infty$, for $\zeta\in\{\zeta=u k_0 e^{\pm \frac{i \pi}{4}}, -\varepsilon<u<\varepsilon\}$, by observing $\delta^{1}(\zeta)$, we have the following conclusion
\begin{equation}
\delta^{1}(\zeta)\sim \zeta^{i \nu} e^{-i\frac{\zeta^{2}}{4}},
\end{equation}
by using the result of
\begin{equation}
\left|e^{-i\frac{k_0^3 \zeta^2}{32\sqrt{\alpha t}\beta}(\frac{1}{\zeta}-\frac{3}{k_0})}\right|\rightarrow1, \quad as \ t\rightarrow \infty.
\end{equation}
\end{proposition}
\begin{proof}  For $\zeta=u k_0 e^{\pm \frac{i \pi}{4}}$, the index part becomes
\begin{equation}
  -i\frac{k_0^3 \zeta^2}{32\sqrt{\alpha t}\beta}(\frac{1}{\zeta}-\frac{3}{k_0}) = \frac{-i\sqrt{2}k_0^4}{64\sqrt{\alpha t}\beta}u+\frac{k_0^4}{32\sqrt{\alpha t}\beta}(\frac{\sqrt{2}}{2}u-3u^2),
\end{equation}
so we have
\begin{equation}
\begin{aligned}
&\left|e^{-i\frac{k_0^3 \zeta^2}{32\sqrt{\alpha t}\beta}(\frac{1}{\zeta}-\frac{3}{k_0})}\right|=e^{\text{Re}\left\{-i\frac{k_0^3 \zeta^2}{32\sqrt{\alpha t}\beta}(\frac{1}{k}-\frac{3}{k_0})\right\}} =e^{\frac{k_0^4}{32\sqrt{\alpha t}\beta}(\frac{\sqrt{2}}{2}u-3u^2)}\\
&\leq e^{\frac{k_0^4}{32\sqrt{\alpha t}\beta}(\frac{\sqrt{2}}{2}|u|+3|u|^2)} \leq e^{\frac{k_0^4}{32\sqrt{\alpha t}\beta}(\frac{\sqrt{2}}{2}\varepsilon+3\varepsilon^2)}\rightarrow 1, \ as\ t\rightarrow \infty.
\end{aligned}
\end{equation}
by which the effects of the third power can be ignored.
\end{proof}

To solve this problem, we need to begin from the PC-model of four stationary-phase points and their jump lines, as illustrated in Figure \ref{4tyx}.

\begin{figure}[H]
 \centering
  \begin{tikzpicture}[node distance=2cm]
  \draw [dashed] (-3,0)--(-1,0);
  \draw [dashed] (-1,2)--(1,2);
  \draw [dashed] (1,0)--(3,0);
  \draw [dashed] (-1,-2)--(1,-2);
  \draw[-latex](0,2)--(-0.5,2.5){};
  \draw[-latex](0,2)--(0.5,2.5){};
  \draw[-latex](0,2)--(-0.5,1.5){};
  \draw[-latex](0,2)--(0.5,1.5){};

  \draw[-latex](0,-2)--(0.5,-2.5){};
  \draw[-latex](0,-2)--(-0.5,-2.5){};
  \draw[-latex](0,-2)--(-0.5,-1.5){};
  \draw[-latex](0,-2)--(0.5,-1.5){};

  \draw[](-2,0)--(-2.5,0.5){};
  \draw[](-2,0)--(-2.5,-0.5){};
  \draw[](-2,0)--(-1.5,0.5){};
  \draw[](-2,0)--(-1.5,-0.5){};

  \draw[-latex](-2.5,0.5)--(-2.25,0.25){};
  \draw[-latex](-2.5,-0.5)--(-2.25,-0.25){};
  \draw[-latex](-1.5,0.5)--(-1.75,0.25){};
  \draw[-latex](-1.5,-0.5)--(-1.75,-0.25){};

  \draw[](2,0)--(1.5,0.5){};
  \draw[](2,0)--(2.5,0.5){};
  \draw[](2,0)--(1.5,-0.5){};
  \draw[](2,0)--(2.5,-0.5){};

  \draw[-latex](2.5,0.5)--(2.25,0.25){};
  \draw[-latex](2.5,-0.5)--(2.25,-0.25){};
  \draw[-latex](1.5,0.5)--(1.75,0.25){};
  \draw[-latex](1.5,-0.5)--(1.75,-0.25){};

 \node  [below]  at (2,0) {$\Sigma_{z_1}$};
 \node  [below]  at (-2,0) {$\Sigma_{z_3}$};
 \node  [below]  at (0,2) {$\Sigma_{z_2}$};
 \node  [below]  at (0,-2) {$\Sigma_{z_4}$};
 \draw[fill] (2,0) circle [radius=0.03];
\draw[fill] (-2,0) circle [radius=0.03];
\draw[fill] (0,2) circle [radius=0.03];
\draw[fill] (0,-2) circle [radius=0.03];
  \end{tikzpicture}
\caption{The jump contour for the local RH problem near the phase points $z_n$, $n=1,2,3,4$.}
\label{4tyx}
\end{figure}
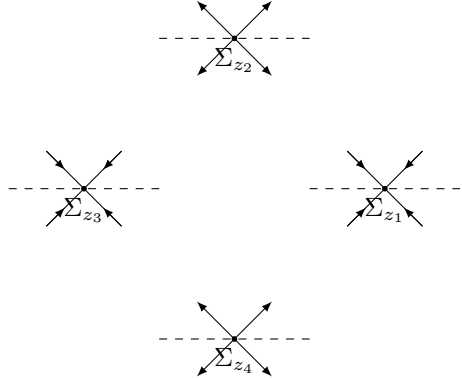

We set
\begin{equation}
r_0=r(z_3)\delta(z_3)^{-2} e^{i\nu(z_3)(\ln k_0^4-\ln 4\alpha\beta^2t)} e^{2it(\frac{\alpha\beta^2}{2k_0^2}-\alpha\beta)}.
\end{equation}
Moreover, with the notation in Proposition \ref{po3}, we have
\begin{equation}
M_{z_3}^{pc}(\zeta)=I+\frac{M_{1}^{pc}(z_3) }{i \zeta}+\mathcal{O}(\zeta^{-2}),
\end{equation}
where
\begin{equation}
M_{1}^{pc}(z_3) =\begin{pmatrix}
0 & \beta_{12}(r_{z_3}) \\
\beta_{21}(r_{z_3})  & 0
\end{pmatrix}.
\end{equation}
with
\begin{align}
&\beta_{12}(r_{z_3})=\displaystyle{\frac{ \sqrt{2 \pi} e^{i \pi / 4} e^{-\pi \nu(z_3) / 2}}{r_{z_3}\Gamma(-i \nu(z_3))} }. \\
&\beta_{21}(r_{z_3})=-\displaystyle{\frac{\sqrt{2 \pi} e^{-i \pi / 4} e^{-\pi \nu(z_3) / 2}}{\overline{r}_{z_3} \Gamma(i \nu(z_3))}}.
\end{align}
 For another three points, the calculation is proceeded in the same way.
 \begin{equation}
M_{z_n}^{pc}(\zeta)=I+\frac{M_{1}^{pc}(z_n) }{ i\zeta}+\mathcal{O}(\zeta^{-2}),\quad \zeta\rightarrow\infty,\quad n=1,2,4.
\end{equation}

 We note that in the model RH problem, the origin is the reference point from which the rays emanate.
In the following sections, we fix the notation $\zeta$ for convenience. Since $M^{fl}$ satisfies the asymptotic property
\begin{equation}
\begin{aligned}
M^{fl}&= I+ \frac{1}{i\zeta} \sum_{n=1}^4 M^{pc}_{1}(z_n) +\mathcal{O}(\zeta^{-2}),\quad \zeta\rightarrow\infty,\label{MFL}
\end{aligned}
\end{equation}
Substituting (\ref{naf}) into (\ref{MFL}), it becomes
\begin{equation}
M^{fl}=I+\frac{k_0^2}{2\beta\sqrt{\alpha t}} \sum_{n=1}^4 \frac{M^{pc}_{1}(z_n)}{k-z_n}+\mathcal{O}(\zeta^{-2}),\quad \zeta\rightarrow\infty.
\end{equation}
In the local circular domain of $z_n$
\begin{equation}
|\frac{1}{k-z_n}|\leq c,
\end{equation}
where $c$ is independent of $k$. The other three points can be controlled in the same way. We can also reach a consistent conclusion that
\begin{equation}
|M^{fl}-I|\lesssim \mathcal{O}(t^{-1/2}).\label{mflgj}
\end{equation}
Besides, it is shown that
\begin{equation}
\left\|M^{fl}(\zeta)\right\|_{\infty} \lesssim 1,
\end{equation}
We use $M^{fl}(\zeta)$ to define a local model
\begin{equation}
M^{(in)}(k)=M^{(out)}(k) M^{fl}(\zeta),\label{min}
\end{equation}
which is a bounded function in $U_{k_{0}}$ and has the same jump matrix as $M^{rhp}(k)$.

\section{Small-norm RH problem for error function}
\label{sec:section8}
\hspace*{\parindent}
From the definition (\ref{mrhp}) and (\ref{min}), we find a RH problem for the matrix function $E(k)$. In this section, we consider the error matrix-function $E(k)$.

\noindent\textbf{ RH Problem 9.1.} Find a matrix-valued function $E(k)$ with following properties:

(a) \textit{Analyticity}: $E(k)$ is analytical in $\mathbb{C} \backslash\Sigma^{(E)}$ and
$$
\Sigma^{(E)}=\cup_{n=1}^4\partial U_{z_n} \cup (\Sigma^{(2)} \backslash  U_{k_0}),
$$
where we orient $\partial U_{z_n}$ clockwise;

(b) \textit{Symmetry}:
\begin{equation}
\overline{E(\overline{k})}=\sigma_2 E(k) \sigma_2^{-1};
\end{equation}

(c) \textit{Asymptotic conditions}:
\begin{equation}
\begin{cases}
E(k)=I+\mathcal{O}\left(k^{-1}\right), & k \rightarrow \infty,\\
E(k)=e^{-\frac{i}{2} c_- \sigma_{3}}\left(I+kU+\mathcal{O}\left(k^{2}\right)\right) e^{\frac{i}{2} d_0\sigma_3},  & k \rightarrow 0;
\end{cases}
\end{equation}

(d) \textit{Jump condition}: $E(k)$ has continuous boundary values $E_{\pm}$ on $\Sigma^{(E)}$ satisfying
\begin{equation}
E_{+}(k)=E_{-}(k) V^{(E)},
\end{equation}
where the jump matrix $V^{(E)}$ is given by
\begin{equation}
V^{(E)}(k)=\begin{cases}
M^{(o u t)}(k) V^{(2)}(k) M^{(o u t)}(k)^{-1}, & k \in \Sigma^{(2)} \backslash U_{  k_0}, \\
M^{(o u t)}(k) M^{fl}(k) M^{(o u t)}(k)^{-1}, & k \in \cup_{n=1}^4\partial U_{z_n},
\end{cases}\label{vek1}
\end{equation}
which is shown in Figure  \ref{fig10}.

\begin{figure}[H]
\begin{center}
\begin{tikzpicture}[node distance=2cm]
\draw (-2,4)--(4,-2)  node[right, scale=1] {};
\draw (2,4)--(-4,-2)  node[above, scale=1] {};
\draw (-4,2)--(2,-4)  node[above, scale=1] {};
\draw (-2,-4)--(4,2)  node[above, scale=1] {};
\draw (0,0)--(1,-1)  node[right, scale=1] {};
\draw (0,0)--(-1,-1)  node[above, scale=1] {};
\draw (0,0)--(1,1)  node[above, scale=1] {};
\draw (0,0)--(-1,1)  node[above, scale=1] {};
\draw [-latex] (0,0) -- (0.7,0.7);
\draw [-latex] (0,0) -- (0.7,-0.7);
\draw [-latex] (0,0) -- (-0.7,0.7);
\draw [-latex] (0,0) -- (-0.7,-0.7);
%\draw [-latex] (0,0) -- (1,0);
%\draw [-latex] (0,0) -- (-1,0);
%\draw [-latex] (0,2) -- (0,1);
%\draw [-latex] (0,-2) -- (0,-1);
%\draw [-latex] (4,0) -- (3,0);
%\draw [-latex] (-4,0) -- (-3,0);
%\draw [-latex] (0,2) -- (0,3);
%\draw [-latex] (0,-2) -- (0,-3);
\draw [-latex] (0,2) -- (0.5,1.5);
\draw [-latex] (0,2) -- (-0.5,1.5);
\draw [-latex] (1,1) -- (1.5,0.5);
\draw [-latex] (-1,1) -- (-1.5,0.5);
\draw [-latex] (0,-2) -- (0.5,-1.5);
\draw [-latex] (0,-2) -- (-0.5,-1.5);
\draw [-latex] (1,-1) -- (1.5,-0.5);
\draw [-latex] (-1,-1) -- (-1.5,-0.5);
\draw [-latex] (4,2) -- (3,1);
\draw [-latex] (4,-2) -- (3,-1);
\draw [-latex] (-4,2) -- (-3,1);
\draw [-latex] (-4,-2) -- (-3,-1);
\draw [-latex] (0,2) -- (1,3);
\draw [-latex] (0,2) -- (-1,3);
\draw [-latex] (0,-2) -- (1,-3);
\draw [-latex] (0,-2) -- (-1,-3);
\draw [fill=pink,ultra thick,white] (2,0) circle (0.5cm);
\draw [fill=pink,ultra thick,white] (-2,0) circle (0.5cm);
\draw [fill=pink,ultra thick,white] (0,2) circle (0.5cm);
\draw [fill=pink,ultra thick,white] (0,-2) circle (0.5cm);
\draw [dashed](-4,0)--(4,0)  node[right, scale=1] {Rek};
\draw [dashed](0,-4)--(0,4)  node[above, scale=1] {Imk};
\node  [right]  at (0,2) {$z_2$};
\node  [right]  at (0,-2) {$z_4$};
\node  [below]  at (2,0) {$z_1$};
\node  [below]  at (-2,0) {$z_3$};
\node  [above]  at (3,1.2) {$\Sigma_{11}$};
\node  [below]  at (3,-1.2) {$\Sigma_{14}$};
\node  [above]  at (1.5,0.6) {$\Sigma_{12}$};
\node  [below]  at (1.5,-0.6) {$\Sigma_{13}$};
\node  [below]  at (0.9,-1.2) {$\Sigma_{43}$};
\node  [below]  at (-0.9,-1.2) {$\Sigma_{42}$};
\node  [right]  at (0.7,1.4) {$\Sigma_{22}$};
\node  [left]  at (-0.7,1.4) {$\Sigma_{23}$};
\node  [below]  at (0.8,2.7) {$\Sigma_{21}$};
\node  [below]  at (-0.8,2.7) {$\Sigma_{24}$};
\node  [above]  at (-1.5,0.6) {$\Sigma_{33}$};
\node  [below]  at (-1.5,-0.6) {$\Sigma_{32}$};
\node  [above]  at (-3,1.2) {$\Sigma_{34}$};
\node  [below]  at (-3,-1.2) {$\Sigma_{31}$};
\node  [right]  at (0.5,0.5) {$\Sigma_{01}$};
\node  [left]  at (-0.5,0.5) {$\Sigma_{02}$};
\node  [left]  at (-0.5,-0.5) {$\Sigma_{03}$};
\node  [right]  at (0.5,-0.5) {$\Sigma_{04}$};
\node  [above]  at (1.3,-3) {$\Sigma_{44}$};
\node  [above]  at (-1.3,-3) {$\Sigma_{41}$};
\draw[fill] (2,0) circle [radius=0.04];
\draw[fill] (-2,0) circle [radius=0.04];
\draw[fill] (0,2) circle [radius=0.04];
\draw[fill] (0,-2) circle [radius=0.04];
\draw[thick,red](0,2)circle(0.5cm);
\draw [thick,red,  -> ]  (0.5,2) to  [out=90, in=0]  (0, 2.5);
\draw[thick,red](0,-2)circle(0.5cm);
\draw [thick,red,  -> ]  (0.5,-2) to  [out=90, in=0]  (0, -1.5);
\draw[thick,red](2,0)circle(0.5cm);
\draw [thick,red,  -> ]  (2.5,0) to  [out=90, in=0]  (2, 0.5);
\draw[thick,red](-2,0)circle(0.5cm);
\draw [thick,red,  -> ]  (-1.5,0) to  [out=90, in=0]  (-2, 0.5);
\end {tikzpicture}
\caption{The jump contour $\Sigma^{(E)}$ for the $E(k)$.}
\label{fig10}
\end{center}
\end{figure}
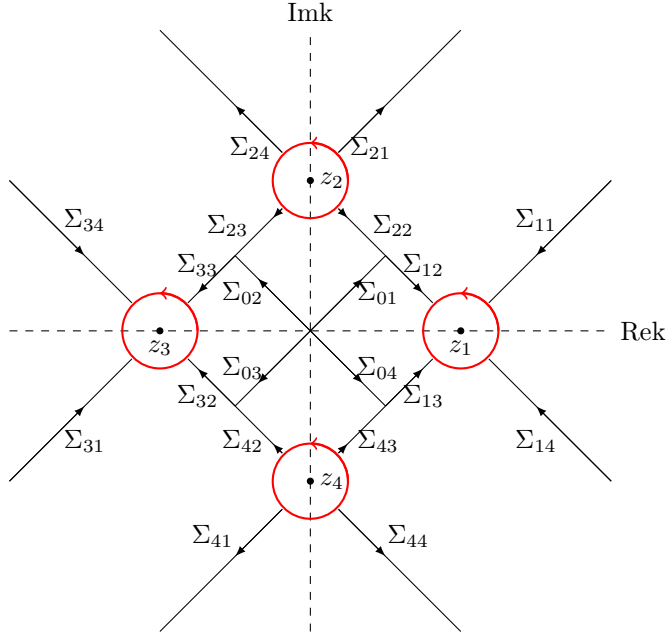

We will show that for large times, the error function $E(k)$ solves the following small-norm RH problem.

By using (\ref{Mout}) and Proposition 2, we have the following estimates
\begin{equation}
\left|V^{(E)}(k)-I\right| \lesssim\begin{cases}
\mathcal{O}\left(e^{-\frac{\alpha\beta^2}{6k_0^3}|k-z_n|^3t}\right) , & k \in \Sigma^{(2)} \backslash U_{k_0}, \\[8pt]
\mathcal{O}\left(e^{-\frac{3\alpha\beta^2}{8k_0^2}t}\right), & k \in \Sigma_{0}^{(2)}.
\end{cases}\label{VE}
\end{equation}

By using  (\ref{Mout}) with (\ref{mflgj}), we show that
\begin{proposition}
For $k \in \cup_{n=1}^4\partial U_{z_n}$,    we find that
\begin{equation}
\left|V^{(E)}(k)-I\right|=\left|M^{(o u t)}(k)^{-1}(M^{fl}(k)-I) M^{(o u t)}(k)\right|=\mathcal{O}(t^{-1 / 2}).\label{VEI}
\end{equation}
\end{proposition}

By using the Beal-Cofiman theory, the solution of  RH Problem 9.1 can be expressed as
\begin{equation}
E(k)=I+\frac{1}{2\pi i} \int_{\Sigma^{(E)}} \frac{  \rho(s) (V^{(E)}(s)-I)}{s-k} d s,\label{Ek}
\end{equation}
where $\rho(s) \in L^2(\Sigma^{(E)})$  is the  solution of the following equation
\begin{equation}
\left(1-C_{E}\right) \rho(s) =I. \label{ce1}
\end{equation}
The singular integral  operator $C_{E}$ is defined by
\begin{align}
&  C_{E}f= C_-(f(V^{(E)}-I)), \label{nlscauchy65}
\end{align}
and $C_-$ is the Cauchy projection operator  defined as
\begin{equation}\label{eq:5.23}
C_{-}f(k)=\operatorname*{lim}\limits_{k'\to k\in\Sigma^{(E)}}  \frac{1}{2\pi i}\int_{\Sigma^{(E)}}\frac{f(s)}{s-k'}ds.\nonumber
\end{equation}

By using estimates (\ref{VE}), (\ref{VEI}) and  (\ref{nlscauchy65}), we show that
\begin{align}
&\|C_{E}\|_{L^{2}(\Sigma^{(E)})\rightarrow L^{2}(\Sigma^{(E)})} \lesssim  t^{-1 / 2},\nonumber\\
& \|\rho(k) -I\|_{L^{2}\left(\Sigma^{(E)}\right)} \lesssim \frac{\left\|C_{E}\right\|}{1-\left\|C_{E}\right\|} \lesssim t^{-1 / 2},\label{ce5}
\end{align}
which implies that  the operator equation (\ref{ce1}) has a unique solution for a sufficiently large $t$.
Then the existence and the uniqueness of the  RH Problem 9.1  is shown by the theorem of a small-norm RH problem.

To reconstruct the solution
$u_x(x,t)$ of the FL  equation (\ref{cs2}), we need the asymptotic behavior of $E(k)$ as $k\rightarrow\infty$.
\begin{proposition}
As $k\rightarrow\infty$,  we have the asymptotic expansion
\begin{equation}
E(k)=I+\frac{E_1}{k}+\mathcal{O}\left(\frac{1}{k^2}\right),\label{Ek}
\end{equation}
where
\begin{equation}
\begin{aligned}
E_1=\frac{k_0^2}{2i\beta \sqrt{\alpha t}}\sum_{n=1}^4(-1)^{n+1}M^{(out)}(z_n)M^{pc}_1(z_n)M^{(out)}(z_n)^{-1}+\mathcal{O}(t^{-1}).
\end{aligned}
\end{equation}
\end{proposition}

\begin{proof}

Combining (\ref{VE}), (\ref{VEI})  and (\ref{ce5}),  we have
\begin{equation}
\begin{aligned}
E_{1}&=-\frac{1}{2 \pi i} \oint_{\sum_{i=1}^4\partial U_{z_n}}\left(V^{(E)}(s)-I\right) ds+\mathcal{O}\left(t^{-1}\right)\\
&=\frac{k_0^2}{2i\beta\sqrt{\alpha t}}\sum_{n=1}^4(-1)^{n+1}M^{(out)}(z_n)M^{pc}_1(z_n)
M^{(out)}(z_n)^{-1}+\mathcal{O}\left(t^{-1}\right).
\end{aligned}
\end{equation}
Consider
\begin{equation}
2 i\left(E_{1}\right)_{12}=t^{-1 / 2} f(x, t)+\mathcal{O}(t^{-1}),\label{2ie}
\end{equation}
where
\begin{equation}
f(x, t)=\frac{k_0^2}{\beta\sqrt{\alpha}}\sum_{n=1}^4(-1)^{n+1}\left(\beta_{12}(r_{z_n}) M_{11}^{(out)}(z_n)^{2}+\beta_{21}(r_{z_n}) M_{12}^{(out)}(z_n)^2\right).\label{fxt}
\end{equation}
\end{proof}

\section{Asymptotic analysis on  the pure $\overline{\partial}$-problem}
\label{sec:section9}
\hspace*{\parindent}
Under the discussion of $M^{rhp} (k)$ in Section \ref{sec:section5}  and  \ref{sec:section6}, the part with $\overline{\partial}R^{(2)}(k)=0$ can be eliminated based on (\ref{m3}),
after which we consider the properties and the long-time asymptotics behavior of $M^{(3)} (k)$. The solution of  RH Problem 6.2 is equivalent to the integral equation
\begin{equation}
M^{(3)}(k) =I-\frac{1}{\pi} \int_{\mathbb{C}} \frac{M^{(3)}(s)W^{(3)}(s)}{k-s}dm(s),
\end{equation}
where $dm(s)$ is the Lebesgue measure on the $\mathbb{C}$. If we denote $C_k$ is the left Cauchy-Green integral operator,
\begin{equation}
f C_{k}(k)=-\frac{1}{\pi}\int_{\mathbb{C}}\frac{f(s)W^{(3)}(s)}{k-s} dm(s),
\end{equation}
then above equation can be rewritten as
\begin{equation}
M^{(3)}(k)=I\cdot\left(I-C_{k}\right)^{-1},
\end{equation}
To prove the existence of operator $(I-C_k)^{-1}$, we need the following lemma.
\begin{lemma}
The norm of the integral operator $C_k$ decays to zero as $t\rightarrow\infty$:
\begin{equation}
\left\|C_{k}\right\|_{L^{\infty} \rightarrow L^{\infty}}\leq ct^{-1/4},
\end{equation}
which implies that $(I-C_k)^{-1}$ exists.
\end{lemma}
\begin{proof}
For any $f\in L^{\infty}$
\begin{equation}
\begin{aligned}
\left\|f C_{k}\right\|_{L^{\infty}} & \leq\|f\|_{L^{\infty}} \frac{1}{\pi} \int_{\mathbb{C}} \frac{\left|W^{(3)}(s)\right|}{|k-s|} d m(s) \\
& \lesssim\|f\|_{L^{\infty}} \frac{1}{\pi} \int_{\mathbb{C}} \frac{\left|\overline{\partial} R^{(2)}(s)\right|}{|k-s|} d m(s),
\end{aligned}
\end{equation}
on account of
\begin{equation}
|W^{(3)}(s)| \leq\|M^{rhp}\|_{L^{\infty}}|\overline{\partial} R^{(2)}(s)|\|M^{rhp}\|_{L^{\infty}}^{-1} \lesssim|\overline{\partial} R^{(2)}(s)|.
\end{equation}
After simplifying, we only need to estimate
\begin{equation}
\frac{1}{\pi} \int_{\mathbb{C}} \frac{\left|\overline{\partial} R^{(2)}(s)\right|}{|k-s|} d m(s).
\end{equation}
As $\bar{\partial} R^{(2)}(s)$
is a piece-wise function, we detail the case in the region $D_{11}$, and the other regions can be handled in a similar way. From (\ref{Rjk}), we have
\begin{equation}
\int_{D_{11}} \frac{\left|\overline{\partial} R^{(2)}(s)\right|}{|k-s|} d m(s) \leq F_1+F_2+F_3,
\end{equation}
where
\begin{align}
&F_{1}=\int_{0}^{+\infty} \int_{z_1+v}^{+\infty} \frac{|\overline{\partial} X_{\mathcal{K}}(u)| e^{\alpha\beta^2uvt\left(\frac{1}{(u^2+v^2)^2}-\frac{1}{k_0^4}\right)}}{\sqrt{(u-x)^{2}+(v-y)^{2}}}dudv, \\
&F_{2}=\int_{0}^{+\infty} \int_{z_1+v}^{+\infty} \frac{\left|r^{\prime}(u)\right| e^{\alpha\beta^2uvt\left(\frac{1}{(u^2+v^2)^2}-\frac{1}{k_0^4}\right)}}{\sqrt{(u-x)^{2}+(v-y)^{2}}} d u d v, \\
&F_{3}=\int_{0}^{+\infty} \int_{z_1+v}^{+\infty} \frac{((u-z_1)^{2}+v^{2})^{-1/4}e^{\alpha\beta^2uvt\left(\frac{1}{(u^2+v^2)^2}-\frac{1}{k_0^4}\right)}}{\sqrt{(u-x)^{2}+(v-y)^{2}}} d u d v.
\end{align}
We denote $s=u+iv, k=x+iy$.

In the following calculation, we will use the inequality
\begin{equation}
\||s-k|^{-1}\|_{L^2(z_1,+\infty)}^{2}=\int_{z_1}^{+\infty} \frac{1}{|v-y|}[(\frac{u-x}{v-y})^{2}+1]^{-1} d(\frac{u-x}{|v-y|}) \leq \frac{\pi}{|v-y|}.
\end{equation}
Without loss of generality, we suppose $y>0$, because, if $y<0$, we can directly remove the absolute value sign and use the same way for estimation.

For $F_1$, noting that $\alpha\beta^2uvt\left(\frac{1}{(u^2+v^2)^2}-\frac{1}{k_0^4}\right)$ is a monotonic decreasing function of $u$, we have
\begin{equation}
\begin{aligned}
F_{1} & \leq \int_0^{+\infty}\||s-k|^{-1}\|_{L^{2}(v+z_1,+\infty)}\|\bar{\partial} X_{\mathcal{K}}(s)\|_{L^{2}(v+z_1,+\infty)}e^{\alpha\beta^2v(v+z_1)t\left(\frac{1}{(z_1^2+v^2)^2}-\frac{1}{k_0^4}\right)}dv\\
& \leq c\int_0^{+\infty}|v-y|^{-1 / 2} e^{\alpha\beta^2v(v+z_1)t\left(\frac{1}{(z_1^2+v^2)^2}-\frac{1}{k_0^4}\right)}dv\\
&\leq c \int_0^{\infty}|v-\beta|^{-1/2} e^{-\frac{\alpha\beta^2}{k_0^4}tv^2} d v \leq c t^{-1/4}.\nonumber
\end{aligned}\label{int1}
\end{equation}
The $F_2$ has the same estimate with $F_1$. For $F_3$, we choose $p>2$ and $q$ H\"{o}lder conjugate to $p$. Then
\begin{equation}
\begin{aligned}
&\|(u-z_1)^2+v^2)^{-1 / 4} \|_{L^p(z_1,+\infty)} \\
&=\left\{\int_{z_1}^{+\infty}[(u-z_1)^2+v^2]^{-p/4} d v\right\}^{1/p} \\
&=\left\{\int_{z_1}^{+\infty}[1+(\frac{u-z_1}{v})^2]^{-p/4} d(\frac{u-z_1}{v})\right\}^{1/p}v^{1/p-1/2}\\
&\leq c v^{1/p-1/2},
\end{aligned}\label{uk0}
\end{equation}
and
\begin{equation}
\begin{aligned}
\||s-k|^{-1}\|_{L^q(z_1,+\infty)} &=\left\{\int_{z_1}^{+\infty}[\left(\frac{u-x}{v-y}\right)^2+1]^{-q/2} d\left(\frac{u-x}{|v-y|}\right)\right\}^{1/q}|v-y|^{1/q-1} \\
& \leq|v-y|^{1/q-1},
\end{aligned}\label{sk1}
\end{equation}
where $p>2$ and $1/p+1/q=1$. Then we have
\begin{equation}
\begin{aligned}
F_{3} & \leq \int_{0}^{+\infty}\||s-k|^{-1}\|_{L^{q}}\|((u-z_1)^{2}+v^{2})^{-1 / 4}\|_{L^{p}} e^{\alpha\beta^2v(v+z_1)t\left(\frac{1}{(z_1^2+v^2)^2}-\frac{1}{k_0^4}\right)} d v \\
& \leq \int_0^{+\infty} v^{1 / p-1/2}|v-y|^{1/q-1} e^{\alpha\beta^2v(v+z_1)t\left(\frac{1}{(z_1^2+v^2)^2}-\frac{1}{k_0^4}\right)} d v \\
& \leq \int_0^y v^{1/p-1/2}(y-v)^{1 / q-1} e^{\alpha\beta^2v(v+z_1)t\left(\frac{1}{(z_1^2+v^2)^2}-\frac{1}{k_0^4}\right)} d v \\
&+\int_y^{+\infty} v^{1/p-1/2}(v-y)^{1/q-1} e^{\alpha\beta^2v(v+z_1)t\left(\frac{1}{(z_1^2+v^2)^2}-\frac{1}{k_0^4}\right)} d v.
\end{aligned}
\end{equation}
For the first term, using the inequality $e^{-z} \leq c z^{-1/4}$ for all $z>0$ leads to
\begin{equation}
\begin{aligned}
&\int_{0}^{y} v^{1 / p-1 / 2}(y-v)^{1 / q-1} e^{\alpha\beta^2v(v+z_1)t\left(\frac{1}{(z_1^2+v^2)^2}-\frac{1}{k_0^4}\right)} d v \\
&\leq ct^{-1/4} \int_{0}^{y} v^{1 / p-1}(y-v)^{1 / q-1} d v \leq c t^{-1/4}.
\end{aligned}
\end{equation}
Similarly, we estimate the second term in the same way of estimating $F_1$. Let $w=v-y$. Then
\begin{equation}
\begin{aligned}
&\int_{y}^{+\infty} v^{1 / p-1 / 2}(v-y)^{1 / q-1}e^{\alpha\beta^2v(v+z_1)t\left(\frac{1}{(z_1^2+v^2)^2}-\frac{1}{k_0^4}\right)} d v \\
&\leq \int_{0}^{+\infty} w^{1 / q-1}(w+y)^{1 / p-1 / 2} e^{-\frac{\alpha\beta^2}{k_0^4}t(w+y)^2}dw \\
&\leq c t^{-1/4}\int_{0}^{+\infty} w^{-1/2} e^{-\frac{\alpha\beta^2}{k_0^4}tw^2}dw  \leq c t^{-1/4}.
\end{aligned}
\end{equation}
Finally, we have
\begin{equation}
F_{3} \leq c t^{-1/4}.
\end{equation}
Based on the previous formula we arrive at the primary result.
\end{proof}

Consider  the asymptotic  expansion of $M^{(3)}(k)$ at $k=\infty$
\begin{equation}
M^{(3)}(k)= I+  \frac{M^{(3)}_1(x,t)}{k}+\mathcal{O}(k^{-2}),\ \ k\to \infty\label{expM3},
\end{equation}
where
\begin{align}
&M^{(3)}_1(x,t)= \frac{1}{\pi}\int_\mathbb{C} M^{(3)}(s)W^{(3)}(s)  dm(s), \label{m30}
\end{align}
To reconstruct the solution $u_x(x,t)$ of the  FL equation  (\ref{cs2}),  we need  the asymptotic behavior of    $M^{(3)}_1(x,t)$.

\begin{lemma}
For a large $t$, we have
\begin{equation}
|M_{1}^{(3)}(x, t)| \leq ct^{-3/4}.\label{m13}
\end{equation}
\end{lemma}

\begin{proof}
By using (\ref{expM3}) and (\ref{m30}), also noting  the boundedness of  $M^{(3)}(k)$ and  $ M^{rhp}(k)$, we obtain that
\begin{align}
\parallel M^{(3)}_1(x,t) \parallel&\leq \frac{1}{\pi}  \iint_{D_{11}} | M^{(3)} M^{rhp}\bar\partial R^{(2)}   {M^{rhp}}^{-1}|dm(s)\nonumber \\
&\leq c\int^{+\infty}_0\int^{+\infty}_{z_1+v} |\bar{\partial}R_1(s)|e^{\alpha\beta^2uvt\left(\frac{1}{(u^2+v^2)^2}-\frac{1}{k_0^4}\right)}dudv\nonumber \\
& \leq c(I_1+I_2+I_3), \label{sfes}
\end{align}
with
\begin{align}
&I_{1}=\int_{0}^{+\infty} \int_{z_1+v}^{+\infty}  |\bar{\partial} X_{\mathcal{K}}(s)| e^{\alpha\beta^2uvt\left(\frac{1}{(u^2+v^2)^2}-\frac{1}{k_0^4}\right)}dudv, \\
&I_{2}=\int_0^{+\infty} \int_{z_1+v}^{+\infty} |r^{\prime}(u)| e^{\alpha\beta^2uvt\left(\frac{1}{(u^2+v^2)^2}-\frac{1}{k_0^4}\right)}d u d v, \\
&I_{3}=\int_0^{+\infty} \int_{z_1+v}^{+\infty} ((u-z_1)^2+v^2)^{-1/4}e^{\alpha\beta^2uvt\left(\frac{1}{(u^2+v^2)^2}-\frac{1}{k_0^4}\right)} dudv.
\end{align}
We bound $I_1$ by applying the Cauchy-Schwarz inequality:
\begin{equation}
\begin{aligned}
I_1=&\int_0^{+\infty}\int_{z_1+v}^{+\infty} |\bar{\partial} X_{\mathcal{K}}(s)| e^{\alpha\beta^2uvt\left(\frac{1}{(u^2+v^2)^2}-\frac{1}{k_0^4}\right)}dudv \\
&\leq \int_0^{+\infty} \|\bar{\partial} X_{\mathcal{K}}\|_{L_{u}^2(v+z_1, \infty)}\left(\int_{z_1+v}^{+\infty}e^{2\alpha\beta^2uvt\left(\frac{1}{(u^2+v^2)^2}-\frac{1}{k_0^4}\right)}du\right)^{1/2} d v \\
&\leq c\left(\int_0^{+\infty} \left(\int_{z_1+v}^{+\infty}e^{-\frac{2\alpha\beta^2}{k_0^4}uvt} d u\right)^{1/2}\right)dv \\
&\leq ct^{-1/2}\int_0^{+\infty} v^{-1/2}e ^{-\frac{2\alpha\beta^2}{k_0^4}(z_1+v)vt}dv\\
&\leq c t^{-3/4}.
\end{aligned}
\end{equation}
The bound for $I_2$ could be attained with the same method as for $I_1$. For $I_3$ we proceed as with (\ref{uk0}) applying the H\"{o}lder's inequality with $2<p<4$
\begin{equation}
\begin{aligned}
I_3&=\int_0^{+\infty} \int_{z_1+v}^{+\infty} ((u-z_1)^2+v^2)^{-1/4} e^{\alpha\beta^2uvt\left(\frac{1}{(u^2+v^2)^2}-\frac{1}{k_0^4}\right)} dudv \\
&\leq \int_0^{+\infty}\|((u-z_1)^{2}+v^{2})^{-1/4}\|_{L^p} \left(\int_{z_1+v}^{+\infty}e^{q\alpha\beta^2uvt\left(\frac{1}{(u^2+v^2)^2}-\frac{1}{k_0^4}\right)}\right)^{1/q} dv \\
&\leq \int_0^{+\infty} v^{1/p-1/2}\left(\int_{z_1+v}^{+\infty}e^{-\frac{q\alpha\beta^2}{k_0^4}uvt} du\right)^{1/q}dv\\
&\leq ct^{-1/q}\int_0^{+\infty} v^{2 / p-3/2}e ^{-\frac{\alpha\beta^2}{k_0^4}(z_1+v)vt}dv\\
&\leq ct^{-3/4} \int_0^{\infty} w^{2/p-3/2} e^{-\frac{\alpha\beta^2}{k_0^4}w^2} d w \leq ct^{-3/4},
\end{aligned}
\end{equation}
where we use the substitution $w=t^{1 / 2} v$ and the fact that $-1<2/p-3/2<-1/2$.
\end{proof}

\section{Large  time asymptotic behavior  for the   FL equation }
\label{sec:section10}
\hspace*{\parindent}
Now we begin to construct the long-time asymptotics of the FL equation (\ref{cs2}). Inverting
the sequence of transformations (\ref{mktk}), (\ref{M2}), (\ref{m3}) and (\ref{mrhp}), we have
\begin{equation}
M(k)=M^{(3)}(k) E(k) M^{(out)}(k) R^{(2)}(k)^{-1} T(k)^{\sigma_{3}}, \quad k \in \mathbb{C} \backslash U_{k_0},
\end{equation}
where $T(k)^{\sigma_3}$ is a diagonal matrix.

To reconstruct the solution $u_x(x,t)$, we take $k\rightarrow\infty$ along the straight line
$k\in D_{12}\cup D_{22}\cup D_{35}\cup D_{45}$, which means $R^{(2)}(k)=I$. From (\ref{tk}), (\ref{uxout}), (\ref{Ek}) and
(9.14), we have
\begin{equation}
M=(I+\frac{M_1 ^{(3)}}{k}+\ldots)(I+\frac{E_1}{k}+\ldots)(I+\frac{M_1^{(out)}}{k}+\ldots)(I+\frac{T_1^{\sigma_3}}{k}+\ldots),
\end{equation}
which means the coefficient of the $k^{-1}$ in the Laurent expansion of $M$ is
\begin{equation}
M_1=M_1^{(3)}+E_1+M_1^{(out)}+T_1^{\sigma_{3}}.
\end{equation}
We construct the solution $u_x(x,t)$ of (\ref{cs2}) with initial data $u_0$ by the transformation and the final result is as follows:

\begin{theorem}
Assume that  $u(x,t)$ be the solution for the initial-value problem (\ref{cs2})-(\ref{cz}) with the appropriate  generic
data $u_0 (x)$.  For fixed $x_1 ,x_2 ,v_1 ,v_2 \in \mathbb{R}$ with $x_1\leq x_2$ and $v_1\leq v_2\in \mathbb{R}^-$, we define two zones for the spectral variable $k$
\begin{equation}
\mathcal{I}=\{k:f(v_2)<|k|<f(v_1)\},
\end{equation}
with
$$
\mathcal{K}(\mathcal{I})=\{k_j \in \mathcal{K}: k_j \in \mathcal{I}\},\nonumber \quad \mathcal{N}(\mathcal{I})=|\mathcal{K}(\mathcal{I})|,\nonumber
$$
and a cone for variables $x,t$
\begin{equation}
\mathcal{C}(x_1, x_2, v_1, v_2)=\{(x, t) \in \mathbb{R}^2 \mid x=x_0+v t, \text { with } x_0 \in[x_1, x_2], v \in[v_1, v_2]\}
\end{equation}
as is shown in Figure \ref{zdtx}. Denote $m_{sol}(x,t|\mathcal{D}(\mathcal{I}))$ be the $\mathcal{N}(\mathcal{I})$-soliton solution corresponding to the scattering data $\{k_j, c_j(\mathcal{I})\}_{j=1}^{\mathcal{N}(\mathcal{I})}$ which is given in (\ref{cji}).  Then as $t\rightarrow \infty$ with $(x,t) \in \mathcal{C}(x_1, x_2, v_1, v_2)$, from (\ref{gs1}), (\ref{uxout}), (\ref{2ie}) and (\ref{m13}) we have
\begin{equation}
m(x,t)=m_{sol}(x,t|\mathcal{D}(\mathcal{I}))+t^{-1/2} f(x,t)+\mathcal{O}(t^{-3/4}).
\end{equation}
Thus
\begin{equation}
\begin{aligned}
|m(x,t)|^{2}&=|(m_{sol}(x,t|\mathcal{D}(\mathcal{I}))+t^{-1/2} f(x,t)+\mathcal{O}(t^{-3/4}))|^2\\
&=|m_{sol}(x,t|\mathcal{D}(\mathcal{I}))|^2+2t^{-1/2}f(x,t)m_{sol}(x,t|\mathcal{D}(\mathcal{I}))+\mathcal{O}(t^{-3/4}).
\end{aligned}
\end{equation}
Based on the above discussion we can obtain
\begin{equation}
u_x(x,t)=\left(m_{s o l}(x, t |\mathcal{D}(\mathcal{I}))+t^{-1/2} f(x, t)+\mathcal{O}(t^{-3/4})\right)e^{-i \int_{-\infty}^{x}|m|^2d x^{\prime}},
\end{equation}
where $f(x,t)$ has been established in (\ref{fxt}).
\end{theorem}

\begin{appendix}
\section{ $\overline{\partial}$-steepest descent analysis for large negative times}
\label{sec:section11}
\hspace*{\parindent}
The steps in the steepest descent analysis of  RH Problem 3.1 for $t \rightarrow-\infty$ mirror those presented in Sections 3-10 for $t \rightarrow\infty$. The differences that appear can be traced back to the fact that the regions of growth and decay of the exponential factors $e^{2it\theta}$ are reversed when one considers $t \rightarrow-\infty$, as Figure \ref{cstx1} hows. In this part, we briefly sketch those changes, leaving the detailed calculations to the interested reader.
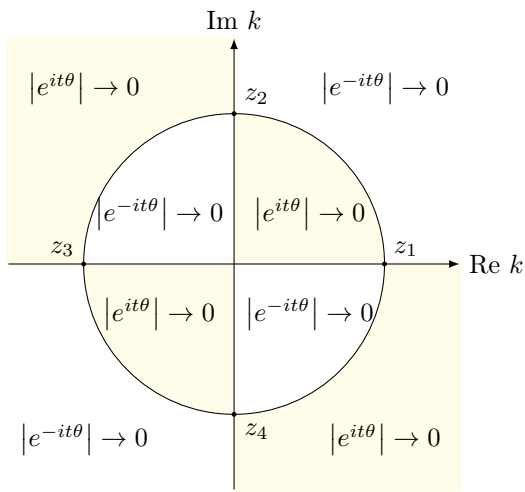
\begin{figure}[H]
\centering
\begin{tikzpicture}[node distance=2cm]
\draw [fill=pink,ultra thick,color=yellow!10] (0,0) rectangle (3,-3);
\draw [fill=pink,ultra thick,color=yellow!10] (0,0) rectangle (-3,3);
\filldraw[white](0,0)-- (0,-2) arc (270:360:2);
\filldraw[white](0,0)-- (0,2) arc (90:180:2);
\filldraw[color=yellow!10](0,0)-- (2,0) arc (0:90:2);
\filldraw[color=yellow!10](0,0)-- (-2,0) arc (180:270:2);
\draw [-latex](-3,0)--(3,0)  node[right, scale=1] {Re $k$};
\draw [-latex](0,-3)--(0,3)  node[above, scale=1] {Im $k$};
\draw(0,0)circle(2cm);
\draw[fill] (0,2) circle [radius=0.025];
\draw[fill] (0,-2) circle [radius=0.025];
\draw[fill] (2,0) circle [radius=0.025];
\draw[fill] (-2,0) circle [radius=0.025];
\node  [right]  at (2,0.2) {$z_1$};
\node  [above]  at (0.3,2) {$z_2$};
\node  [left]  at (-2,0.2) {$z_3$};
\node  [below]  at (0.3,-2) {$z_4$};
\node  [below]  at (1,1) {$\left|e^{ i t \theta}\right|\rightarrow 0$};
\node  [above]  at (1,-1) {$\left|e^{ -i t \theta}\right|\rightarrow 0$};
\node  [below]  at (-1,1) {$\left|e^{-it\theta}\right|\rightarrow 0$};
\node  [above]  at (-1,-1) {$\left|e^{ it  \theta}\right|\rightarrow 0$};
\node  [above]  at (2,2) {$\left|e^{-i  t\theta}\right|\rightarrow 0$};
\node  [below]  at (2,-2) {$\left|e^{ i t\theta}\right|\rightarrow 0$};
\node  [above]  at (-2,2) {$\left|e^{ i t \theta}\right|\rightarrow 0$};
\node  [below]  at (-2,-2) {$\left|e^{ -it \theta}\right|\rightarrow 0$};
\end {tikzpicture}
\caption{ In the shaded regions, $|e^{ it\theta} |\rightarrow 0$ when $t\rightarrow\infty$, while, in the blank regions, $|e^{-it\theta}|\rightarrow 0$ when $t\rightarrow\infty$.}
\label{cstx1}
\end{figure}
The first step in the analysis, as in Section 3, is a conjugation to well-condition the problem for large-time analysis. Similar to (\ref{mktk}) define
\begin{equation}
M^{(1)}(k)=M(k) T(k)^{-\sigma_3}
\end{equation}
$T(k)$ can be defined in Proposition \ref{p} with $t<0$. Non-analytic extensions of the jump matrices (\ref{V11}) are introduced to deform jump matrices onto contours along which they decay to the identity as was done in Section 4 whose contours and domains are shown below.

Accordingly, the definition of $R_j$ is also slightly different from $t>0$, i.e., we exchange the definition of $R_j$ on the two sides of steady-state phase points on the real and the imaginary axis. Once the functions are constructed, the transformation
\begin{equation}
M^{(2)}(k)=M^{(1)}(k) \mathcal{R}(k).
\end{equation}
where $\mathcal{R}(k)$ is defined in each sector in Figure \ref{fig10}. We define a new unknown $M^{(2)}(k)$ which satisfies

\noindent\textbf{RH Problem A.1.} Find a function $M^{(2)}(k)=M^{(2)}(k;x,t) $ with the following properties.

(a)\emph{Analyticity}: $M^{( 2)}(k)$  is continuous in $\mathbb{C}$, has sectionally continuous first partial derivatives in $\mathbb{C} \backslash (\Sigma^{(2)} \cup \mathcal{K}\cup \mathcal{\overline{K}} )$, and meromorphic in $D_{n2} \cup D_{n5}$, $n=1,2,3,4$;

(b)\emph{Symmetry}: $\overline{M^{(2)}(\bar{k})}=\sigma_2 M^{(2)}(k) \sigma_2$;

(c)\emph{Jump condition}: The boundary value $ M^{(2)}(k;x,t)$ at $ \Sigma^{(2)}$ satisfies the jump condition
\begin{equation}
M^{(2)}_+ (k) =M^{(2)}_- (k)V^{(2)}(k),\quad k\in \Sigma^{(2)};
\end{equation}

(d)\emph{Asymptotic condition}
\begin{equation}
\begin{cases}
M^{(2)}(k)=I+\mathcal{O}\left(k^{-1}\right), & k \rightarrow \infty,\\
M^{(2)}(k)=e^{-\frac{i}{2} c_- \sigma_{3}}\left(I+kU+\mathcal{O}\left(k^{2}\right)\right) e^{\frac{i}{2} d_0\sigma_3},  & k \rightarrow 0;
\end{cases}
\end{equation}

(e) Away from $\Sigma^{(2)}$ we have
\begin{equation}
\overline{\partial} M^{(2)}(k)= M^{(2)} (k)\overline{\partial} \mathcal{R}^{(2)}
\end{equation}
holds in $\mathbb{C}\backslash\Sigma^{(2)}$, where
\begin{equation}
\overline{\partial}\mathcal{R}^{(2)}(k)=\begin{cases}
\begin{pmatrix} 1&0\\ (-1)^j \overline{\partial}R_{n,j}(k) e^{2it\theta} &1\end{pmatrix},& k\in D_{nj},\ n=1,2,3,4;\ j=1,4,\\
\begin{pmatrix} 1&(-1)^j \overline{\partial}R_{n,j}(k) e^{-2it\theta}\\
0&1\end{pmatrix},& k\in D_{nj},\ n=1,2,3,4;\ j=3,6,\\
\begin{pmatrix} 0&0\\0&0\end{pmatrix},& k\in D_{nj}, \ n=1,2,3,4;\ j=2,5;
\end{cases}
\end{equation}

(f) \textit{Residue conditions}: $M^{(2)}$ has simple poles at each point in $\mathcal{K}\cup\mathcal{\overline{K}}$ with:
\begin{align}
&\underset{k=k_j}{\operatorname{Res}} M^{(2)}(k)=
\underset{k \rightarrow k_j}{\lim } M^{(2)}(k) R_{k_j,\mp}, \qquad \ k \in \Delta_{k_0}^\pm,  \\
&\underset{k=\overline{k}_j}{\operatorname{Res}} M^{(2)}(k)=
\underset{k \rightarrow \overline{k}_j}{\lim}  M^{(2)}(k)\sigma_2\overline{R}_{\overline{k}_j,\mp}\sigma_2, \ k \in \Delta_{k_0}^\pm,
\end{align}
where $R_{k_j,-}$ and $R_{k_j,+}$ is defined in (\ref{rkj}).

Mimicking Sections $5-10$, the final steps of the analysis are to first construct a solution $M^{rhp}(k)$ of the RH  components of RH problem A.1, and then to use the solid Cauchy integral operator to prove that the remainder $M^{(3)}(k)=M^{(2)}(k)(M^{rhp})^{-1}(k)$ is uniformly near the identity with estimates identical to (\ref{m13}). When $t\rightarrow-\infty$ with $(x,t)\in \mathcal{C}(x_1, x_2, v_1, v_2)$, the outer model takes the form
\begin{equation}
M^{(out)}(k)=\left[I+\mathcal{O}\left(e^{-4\mu t}\right)\right] m^{\Delta_{k_0}^+(\mathcal{I})}\left(k\mid \mathcal{D}^-(\mathcal{I})\right)
\end{equation}
The local model $M^{(in)}$ is constructed as in Section 8. Define
\begin{equation}
N: f(k) \rightarrow(N f)(k)=f(\frac{k_{0}^{2}}{2 \sqrt{\alpha} \beta \sqrt{-t}} \zeta-z_3).
\end{equation}

Then the local model $M^{(in)}$ is given by
$$
M^{(in)}(k)=M^{(out)}(k) \sigma_3 M^{fl}(-\zeta(k))\sigma_3
$$
where $M^{fl}(\zeta,r)$ is the solution of  RH Problem 8.1.
The residual error $E(k)$ now satisfies RH Problem 9.1  but with (\ref{vek1}) now given by
\begin{equation}
V^{(E)}(k)=\begin{cases}
M^{(o u t)}(k) V^{(2)}(k) M^{(o u t)}(k)^{-1}, & k \in \Sigma^{(2)} \backslash \Sigma_{n=1}^4U_{z_n}, \\
M^{(o u t)}(k)\sigma_3 M^{fl}(-\zeta(k))\sigma_3 M^{(o u t)}(k)^{-1}, & k \in \Sigma_{n=1}^4\partial U_{z_n}.
\end{cases}
\end{equation}
The small-norm theory again can be used to show that $E(k)$ exists and satisfies
\begin{equation}
E(k)=I+k^{-1}E_1+\mathcal{O}\left(k^{-2}\right),
\end{equation}
with
\begin{equation}
\begin{aligned}
E_{1}&=-\frac{1}{2 \pi i} \oint_{\sum_{i=1}^4\partial U_{z_n}}\left(V^{(E)}(s)-I\right) ds+\mathcal{O}\left(t^{-1}\right)\\
&=\frac{k_0^2}{2i\beta\sqrt{\alpha t}}\sum_{n=1}^4(-1)^{n+1}M^{(out)}(z_n)M^{pc}_1(z_n)
M^{(out)}(z_n)^{-1}+\mathcal{O}\left(t^{-1}\right).
\end{aligned}
\end{equation}
The rest of the results can be given in the same way as showed in Section 11.
\end{appendix}

\section{The parabolic cylinder model  }
\label{sec:section12}
\hspace*{\parindent}
Here we describe the solution of the parabolic cylinder model problem introduced by Its \cite{Its}, which was later widely used to study the long-time asymptotics of
 integrable systems in the literature \cite{asd13,rwo34}. For our FL equation,
 there are two   phase points $z_1, z_3 \in \mathbb{R}$ and two  $z_2, z_4 \in i\mathbb{R}$ respectively,
 which need two kinds of parabolic cylinder models to describe.

Therefore,
for  $r_0\in \mathbb{R},$ define
\begin{align}
& \ \nu=\nu(r_0)=-\frac{1}{2\pi}\ln(1 +\varepsilon_n | r_0|^2), \nonumber
\end{align}
where $ \varepsilon_n=1, \ {\rm for}\ n=1,3;  \varepsilon_n=-1, \ {\rm for}\ n=2,4.$  Further define contours
$$
\Sigma^{pc}= \cup_{j=1}^4\Sigma_j,   \ \Sigma_j=\{\zeta=\mathbb{R}^+e^{\frac{i(2j-1)\pi}{4}}, j=1,2,3,4\}.\nonumber
$$
Then we have the following parabolic cylinder model problem.

\noindent\textbf{RH Problem B.1.}  Find a $2\times2$ matrix-valued function $M^{pc}(\zeta,r_0)$ with the following properties:

(a) $M^{pc}(\zeta,r_0)$ is analytic for $\mathbb{C}\backslash \Sigma^{pc}$;

(b) The boundary value $M^{pc}(\zeta,r_0)$ at $\Sigma^{pc}$ satisfies the jump condition
\begin{equation}
M_+^{pc}(\zeta,r_0)=M_-^{pc}(\zeta,r_0)V^{pc}(\zeta),\quad \zeta\in \Sigma^{pc},
\end{equation}
where
\begin{equation}
V^{pc}(\zeta)=\begin{cases}
\begin{pmatrix}
1 & 0 \\
r_0 \zeta^{-2i\nu}e^{i\frac{\zeta^2}{2}} & 1
\end{pmatrix}, &  \zeta\in \Sigma_1,\\
\begin{pmatrix}
1 & \frac{\varepsilon_n\overline{r}_0}{1+  \varepsilon_n|r_0|^2 }\zeta^{2i\nu}e^{-i\frac{\zeta^2}{2}}  \\
0 & 1
\end{pmatrix}, & \zeta\in \Sigma_2,\\
\begin{pmatrix}
1 & 0 \\
\frac{r_0}{1+  \varepsilon_n|r_0|^2}\zeta^{-2i\nu}e^{i\frac{\zeta^2}{2}} & 1
\end{pmatrix}, & \zeta\in \Sigma_3,\\
\begin{pmatrix}
1 &\varepsilon_n \overline{r}_0\zeta^{2i\nu}e^{-i\frac{\zeta^2}{2}}  \\
0 & 1
\end{pmatrix}, & \zeta\in \Sigma_4.\\
\end{cases}
\end{equation}

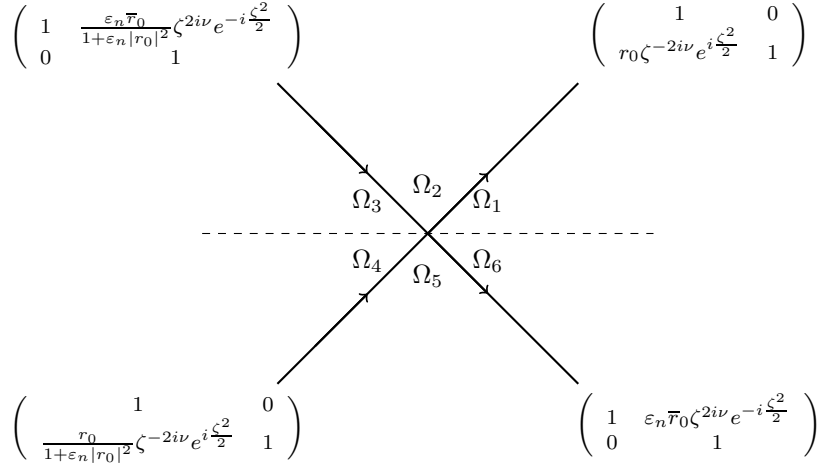
\begin{figure}[H]
\begin{center}
\begin{tikzpicture}
\draw [dashed] (-3,0)--(3,0);
\draw[thick ] (-2,-2)--(2,2);
\draw[thick,-> ](-1.5,-1.5)--(-0.8,-0.8);
\draw[thick,-> ](0,0)--(0.8,0.8);
\draw[thick ](-2,2)--(2,-2);
\draw[thick,-> ](-1.5,1.5)--(-0.8,0.8);
\draw[thick,-> ](0,0)--(0.8,-0.8);
\node  [below]  at (0.8,0.7) {$ \Omega_1$};
\node  [below]  at (0, 0.9) {$ \Omega_2$};
\node  [below]  at (-0.8, 0.7) {$ \Omega_3$};
\node  [below]  at (-0.8,-0.1) {$ \Omega_4$};
\node  [below]  at (0, -0.3) {$ \Omega_5$};
\node  [below]  at (0.8,-0.1) {$ \Omega_6$};
\node [thick ] [below]  at (3.6, 3.2) {\footnotesize $   \left(\begin{array}{cc} 1&0\\
r_0 \zeta^{-2i\nu}e^{i\frac{\zeta^2}{2}}&1\end{array}  \right) $};
\node [thick ] [below]  at (-3.6,3.2) {\footnotesize $  \left(\begin{array}{cc} 1&\frac{\varepsilon_n\overline{r}_0}{1+  \varepsilon_n|r_0|^2}\zeta^{2i\nu}e^{-i\frac{\zeta^2}{2}} \\
0&1\end{array}  \right) $};
\node [thick ] [below]  at (-3.6,-2) {\footnotesize $  \left(\begin{array}{cc} 1&0\\
\frac{r_0}{1+  \varepsilon_n|r_0|^2}\zeta^{-2i\nu}e^{i\frac{\zeta^2}{2}}&1\end{array}  \right)$};
\node [thick ] [below]  at (3.6,-2) {\footnotesize $  \left(\begin{array}{cc} 1& \varepsilon_n\overline{r}_0\zeta^{2i\nu}e^{-i\frac{\zeta^2}{2}}\\
0&1\end{array}  \right)$};
\end{tikzpicture}
\end{center}
\caption{ $ \Sigma^{pc}$ and domains $\Omega_j$, $j=1,2,3,4,5,6$.}
\label{vn5}
\end{figure}

It can be shown that the RH Problem B.1 admits a solution
\begin{equation}
M^{pc}(\zeta, r_0)=I+\frac{M^{pc}_1(  r_0)}{i\zeta}+\mathcal{O}(\zeta^{-2}),
\end{equation}
where
\begin{align}
&M^{pc}_1(  r_0)=\begin{pmatrix}
0 & \beta_{12}  \\
-\beta_{21}  & 0
\end{pmatrix}, \nonumber
\end{align}
with $\beta_{12}$ and $\beta_{21}$ being two complex constants
\begin{align}
&  \ \beta_{12} =\frac{\sqrt{2 \pi} e^{i \pi / 4} e^{-\pi \nu / 2}}{r_0 \Gamma(-i \nu)},\ \ \ \
\beta_{21} =-\frac{\sqrt{2 \pi} e^{-i \pi / 4} e^{-\pi \nu / 2}}{ \varepsilon_n\overline{r}_0  \Gamma(i \nu)}.\nonumber
\end{align}
\vspace{5mm}

\noindent\textbf{Acknowledgements}

This work is supported by  the National Natural Science
Foundation of China (Grant No. 11671095, 51879045).

\end{document}